\documentclass{./Article-eng}

\usepackage{hyperref}
\hypersetup{colorlinks,allcolors=blue}
\usepackage{url}
\usepackage{csquotes}

\DeclareMathOperator{\ff}{f} 
\DeclareMathOperator{\plat}{flat}

\DeclareMathOperator{\rat}{rat}
\newcommand{\RefSym}{\hat{\text{S}}}
\DeclareMathOperator{\Chow}{Chow}

\DeclareMathOperator{\free}{f} 
\DeclareMathOperator{\base}{b} 
\DeclareMathOperator{\nilpotent}{nilp}

\title{Structure of projective varieties with nef anticanonical divisor: the case of log terminal singularities}
\author{Juanyong Wang}
\date{}

\begin{document}
\maketitle
\paragraph{Abstract:} 
In this article we study the structure of klt projective varieties with nef anticanonical divisor (and more generally, varieties of semi-Fano type), especially the canonical fibrations associated to them. We show that 
\begin{itemize}
\item The Albanese map for such variety is a locally constant fibration (that is, an analytic fibre bundle with connected fibres that $X$ is equal to the product of the universal cover of the Albanese torus by the fibre of the Albanese map quotient by a diagonal action of the fundamenatl group of the Albanese torus).   
\item If the smooth locus is simply connected, the MRC fibration of such variety is an everywhere defined morphism and induces a decomposition into a product of a rationally connected variety and of a projective variety with trivial canonical divisor. 
\end{itemize}
These generalize the corresponding results for smooth projective varieties with nef anticanonical bundle in Cao (2019) and Cao-H\"oring (2019) to the klt case, and can be also regarded as a partial extension of the singular Beauville-Bogomolov decomposition theorem proved by successive works of Greb-Kebekus-Peternell (2016), Druel (2018), Guenencia-Greb-Kebekus (2019) and H\"oring-Peternell (2019). 
\tableofcontents

\section*{Introduction}
\label{sec_intro}
\addcontentsline{toc}{section}{Introduction}
In this article we work over $\CC$. By the philosophy of the Minimal Model Program, 
(hopefully, if the non-vanishing conjecture holds) projective varieties (or more generally, complex varieties in the Fujiki class $\mathscr{C}$) are divided into two classes: the ones with positive Kodaira dimension and the uniruled ones. The study of the first class relies on studying the Iitaka-Kodaira fibrations, which is predicted by the Abundance conjecture; while for the second class (for which the Itaka-Kodaira fibration do not provide any information), one studies the Albanese maps and the MRC fibrations. This article is among the works which intend to understand the structure of uniruled varieties. 

A general philosophy in the study of uniruled varieties is that a variety whose anticanonical bundle or the tangent bundle admits certain positivity, should exhibit certain birational rigidity, e.g. the aforementioned canonical fibrations should have some rigid structure (e.g. being locally constant fibration). This is inspired by the fundamental works of S.Mori \cite{Mori79} and of Siu-Yau \cite{SY80}, proving the conjecture of Hartshorne-Frankel; their works characterize the projective spaces in terms of the amplitude of the tangent bundle (also true in positive characteristics), or equivalently, the positivity of the bisectional curvature (also true for compact K\"ahler manifolds). An analytic generalization of Mori-Siu-Yau's result is obtained by N.Mok in \cite{Mok88} for compact K\"ahler manifolds with nonnegative holomorphic bisectional curvature based on the previous work \cite{HSW81}: he proved that the universal covers of these manifolds are decomposed into products of $\CC^q$, of projective spaces and of (irreducible) compact Hermitian symmetric spaces of rank $\geqslant 2$; recently the result of Mok is generalized by S.Matsumura to the case of semipositive holomorphic sectional curvature in \cite{Mat18a,Mat18b,Mat18c}. In order to establish the algebro-geometric counterpart of the main result of \cite{Mok88}, considerations are given to compact Kähler manifolds with nef tangent bundles, whose structures are settled by \cite{DPS94}, modulo the Campana-Peternell conjecture (it conjectures that smooth Fano varieties with nef tangent bundle are rationally homogeneous), by showing that the Albanese map is a locally constant fibration with Fano fibres. Then attention are further paid to smooth projective varieties (or more generally, compact K\"ahler manifolds) with nef anticanonical bundle. By MMP methods, the $3$-dimensional case is extensively studied by Th.Peternell and his collaborators in \cite{PS98,BP04}. Recently the structure theorem for these varieties is fully established in \cite{Cao19,CH19} by applying the method of positivity of direct images and by the previous works \cite{Zhang96,Paun97,Paun01,Zhang05,LTZZ10}; moreover, the result is extended to klt pairs by \cite{CCM19} when the variety is smooth projective. According to the general philosophy of MMP, it is then natural to extend this structure theorem to the klt singular case, i.e. to prove the following conjecture:
\begin{conj}
\label{conj_klt-anti-nef}
Let $X$ be a projective varieties with klt singularities. Suppose that the anticanonical divisor $-K_X$ of $X$ is nef, then up to replacing $X$ by a finite quasi-\'etale cover, the universal cover $\tilde X$ of $X$ can be decomposed into a product
\[
\tilde X\simeq \CC^q\times Z\times F\,,
\]
where $q$ is the augmented irregularity of $X$, $Z$ is a klt projective variety with trivial canonical divisor and $F$ is a rationally connected variety. 
\end{conj}

Moreover, by the klt Beauville-Bogomolov decomposition theorem, established by the successive works \cite{GKP16a,Druel18a,GGK19,HP19}, the factor $Z$ in the decomposition above can be further decomposed as a product of Calabi-Yau varieties and of irreducible symplectic varieties. However, different from the case of varieties with numerically trivial canonical divisor, even in the smooth case one cannot in general get a product structure up to finite (quasi-)\'etale cover for varieties with nef anticanonical divisor due to the appearance of the rationally connected factor, e.g. there are ruled surfaces over an elliptic curve which cannot split into a product of the elliptic curve and $\PP^1$ up to finite étale cover (c.f. \cite[Example 9.1, Example 11.2]{Druel18b} and \cite[Example 6.2]{EIM20}). In order to prove {\hyperref[conj_klt-anti-nef]{Conjecture \ref*{conj_klt-anti-nef}}} we follow the idea of \cite{Cao19,CH19} and intend to show:

\begin{enumerate}
\item The Albanese map $\alb_X: X\dashrightarrow \Alb_X$ is a(n) (everywhere defined) locally constant fibration, that is, an (locally trivial) analytic fibre bundle with connected fibres such that $X$ is equal to the product of the universal cover of $\Alb_X$ by the fibre of $\alb_X$ quotient by a diagonal action of $\pi_1(\Alb_X)$ (c.f. {\hyperref[defn_local-const-fibration]{Definition \ref*{defn_local-const-fibration}}});
\item The fundamental group of $X_{\reg}$ is of polynomial growth, equivalently (by \cite[Main Theorem]{Gro81}), $\pi_1(X_{\reg})$ is virtually nilpotent (i.e. admits a nilpotent subgroup of finite index); 
\item If $\pi_1(X_{\reg})=\{1\}$ then the maximal rationally connected (MRC) fibration of $X$ is an everywhere defined morphism and induces a decomposition of $X$ into a product of a rationally connected variety and of a projective variety with trivial canonical divisor.
\end{enumerate}

The Points 1 and 3 above will be shown in the present work while the Point 2 seems quite difficult, at least the method in \cite{Paun97} does not seem to apply to this case. Apart from trying to prove the Point 2, there is also hope that one can directly prove the {\hyperref[conj_klt-anti-nef]{Conjecture \ref*{conj_klt-anti-nef}}} without studying the fundamental group (or at least by proving something much weaker on the fundamental group), c.f. \cite{CCM19} and 
{\hyperref[ss_MRC_decomp-thm]{\S \ref*{ss_MRC_decomp-thm}}}. In this article, we will prove the following generalized version of point 1 and 3. Recall that a normal projective variety $X$ is called of {\it Fano type} (resp. {\it semi-Fano type}), if there is an effective $\QQ$-divisor $\Delta$ on $X$ such that $(X,\Delta)$ is a klt pair and that the twisted anticanonical divisor $-(K_X+\Delta)$ is ample (resp. nef), c.f. \cite[Definition 2.5, Lemma-Definition 2.6]{PS09}.

\begin{mainthm}
\label{mainthm_Albanese}
Let $X$ be a normal projective variety of semi-Fano type. Then the Albanese map $\alb_X:X\dashrightarrow \Alb_X$ is an everywhere defined locally constant fibration, i.e. $\alb_X$ is an (isotrivial) analytic fibre bundle with fibre $V$ connected such that $X$ is equal to the product of the universal cover of $\Alb_X$ by $F$ quotient by a diagonal action of $\pi_1(\Alb_X)$.
\end{mainthm}

\begin{mainthm}
\label{mainthm_MRC}
Let $X$ be a normal projective variety of semi-Fano type with simply connected smooth locus $X_{\reg}$. Then the MRC fibration of $X$ induces a decomposition of $X$ into a product $F\times Z$ with $F$ rationally connected and $K_Z\sim 0$. 
\end{mainthm}

Let us remark that the local triviality (also known as the "isotriviality", especially in algebraic geometry) of the Albanese map of $X$ is obtained in the work of Z.Patakafalvi and M.Zdanowicz \cite[Corollary 1.17 (Corollary A.14)]{PZ19} under the additional assumption that $X$ is $\QQ$-factorial. The strategy in their paper is to show that every (closed) fibre is isomorphic by proving the numerical flatness of the direct images on every complete intersection curve. In this article, we can use analytic methods to prove more generally the global numerical flatness of the direct images, and thus can obtain the stronger result that the Albanese map is not only locally isotrivial but also a locally constant fibration. 

As a consequence of {\hyperref[mainthm_Albanese]{Theorem \ref*{mainthm_Albanese}}} and {\hyperref[mainthm_MRC]{Theorem \ref*{mainthm_MRC}}} we can reduce {\hyperref[conj_klt-anti-nef]{Conjecture \ref*{conj_klt-anti-nef}}} to the following {\hyperref[conj_pi1-anti-nef]{Conjecture \ref*{conj_pi1-anti-nef}}}. The detailed proof of this reduction will be given in {\hyperref[ss_fundamental-group_decomposition]{\S \ref*{ss_fundamental-group_decomposition}}}.

\begin{conj}
\label{conj_pi1-anti-nef}
Let $X$ be a normal projective variety of semi-Fano type. Then the fundamental group of $X_{\reg}$ is of polynomial growth.
\end{conj}

As to be shown in {\hyperref[sec_fundamental-group]{\S \ref*{sec_fundamental-group}}}, this conjecture extends the Gurjar-Zhang conjecture on the finiteness of the fundamental group of the smooth locus of varieties of Fano type (c.f. \cite{GZ94,GZ95,Zhang95,Sch07,Xu14,GKP16a,TX17}), which is recently settled by L.Braun in \cite{Braun20}. It can also be regarded as a natural generalization of the following folklore conjecture (c.f. \cite{GGK19}):

\begin{conj}
\label{conj_pi1-Fano-CY}
Let $X$ be a klt projective variety with trivial canonical divisor and vanishing augmented irregularity. Then the fundamental group of $X_{\reg}$ is finite.
\end{conj}

We will see in {\hyperref[sec_fundamental-group]{\S \ref*{sec_fundamental-group}}} that {\hyperref[conj_pi1-anti-nef]{Conjecture \ref*{conj_pi1-anti-nef}}} implies {\hyperref[conj_pi1-Fano-CY]{Conjecture \ref*{conj_pi1-Fano-CY}}}. In the sequel let us briefly explain the ideas of the proof of {\hyperref[mainthm_Albanese]{Theorem \ref*{mainthm_Albanese}}} and {\hyperref[mainthm_MRC]{Theorem \ref*{mainthm_MRC}}}:
\begin{itemize}
\item First, an easy observation shows that \cite[2.8.Proposition]{Cao19} is still valid even the total space is singular (c.f. {\hyperref[prop_num-flat--local-const]{Proposition \ref*{prop_num-flat--local-const}}}), hence the problem of proving that a fibration is locally constant can be reduced to proving that the direct images of the powers of a relative ample line bundle are numerically flat.
\item By \cite[Proposition 2.9]{CH19} (c.f. {\hyperref[prop_criterion-num-flat]{Proposition \ref*{prop_criterion-num-flat}}}) the proof of the numerical flatness of a reflexive sheaf can be divided into two parts: first, prove that the direct image admits weakly semipositive singular Hermitian metrics; second, prove that the determinant bundle of the direct image sheaf is numerically trivial. The first part can be deduced from the general positivity result of direct image sheaves (c.f. \cite[Theorem 2.2]{CCM19} or {\hyperref[cor_positivity-direct-image]{Corollary \ref*{cor_positivity-direct-image}}}) by using the fact that $-K_X$ is nef, c.f. \cite[Lemma 3.4]{CCM19} or  {\hyperref[prop_positivity-anti-nef]{Proposition \ref*{prop_positivity-anti-nef}}}; while the second part can be established, at least birationally, with the help of the main result of \cite{Zhang05} ({\hyperref[prop_bir-geometry-psi]{Proposition \ref*{prop_bir-geometry-psi}}}), c.f. {\hyperref[prop_det-direct-image-power]{Proposition \ref*{prop_det-direct-image-power}}}.
\item By 
using the method of \cite{LTZZ10} we can prove that the Albanese map of $X$ is flat, then we can further improve the aforementioned birational version of the numerical flatness result and show that the direct image of powers of some relatively very ample line bundle is numerically flat; by {\hyperref[prop_num-flat--local-const]{Proposition \ref*{prop_num-flat--local-const}}} this proves {\hyperref[mainthm_Albanese]{Theorem \ref{mainthm_Albanese}}} .
\item As for {\hyperref[mainthm_MRC]{Theorem \ref*{mainthm_MRC}}}, a similar yet much more subtle argument as that in \cite[\S 3.C ]{CH19} applied to the MRC fibration of $X$ shows that birationally $X$ can be decomposed into a product, which gives rise to a splitting of $T_X$ into direct sum of two algebraically integrable foliations, one having rationally connected Zariski closures of leaves, the other having trivial canonical class. However, $X$ being singular and these foliations being singular, one cannot directly apply \cite[2.11.Corollary]{Hor07}. To overcome this difficulty, we observe that the decomposition implies that the two foliations are weakly regular, then we can use the related results in \cite{Druel17,Druel18b} to show that, up to a $\QQ$-factorial terminal model, the MRC fibration is everywhere defined. In this situation, we can use a similar argument as the one in the proof of {\hyperref[mainthm_Albanese]{Theorem \ref*{mainthm_Albanese}}} to show the numerical flatness of the direct images up to a base change, and finally 
\cite[Lemma 4.6]{Druel18a} permits us to conclude.
\end{itemize}


The structure of the article is as follows: In {\hyperref[sec_preliminary]{\S \ref*{sec_preliminary}}} we recall some basic results on weakly semipositively curved and numerically flat vector bundles as well as their relation to the local constancy of fibre spaces. In {\hyperref[sec_positivity]{\S \ref*{sec_positivity}}} we set up the general setting of the problem, which permits us to treat the two theorems uniformly, then we give the proof of a result of \cite{Zhang05} on the birational geometry of any dominant map from $X$ to a smooth non-uniruled projective variety. Based on this result, we recall some results on the positivity of direct images shown in \cite{CH19,CCM19}; for sake of clarity we reprove some of these results. In {\hyperref[sec_Albanese]{\S \ref*{sec_Albanese}}} we prove the flatness of the Albanese map and deduce the {\hyperref[mainthm_Albanese]{Theorem \ref*{mainthm_Albanese}}}. The {\hyperref[sec_MRC]{\S \ref*{sec_MRC}}} is dedicated to the proof of the splitting theorem of the tangent sheaf and of {\hyperref[mainthm_MRC]{Theorem \ref*{mainthm_MRC}}}. In {\hyperref[sec_fundamental-group]{\S \ref*{sec_fundamental-group}}} we study the Albanese map of the smooth locus $X_{\reg}$ of $X$, and deduce related results on the fundamental groups of $X_{\reg}$, especially prove that the {\hyperref[conj_klt-anti-nef]{Conjecture \ref*{conj_klt-anti-nef}}} can be reduced to the {\hyperref[conj_pi1-anti-nef]{Conjecture \ref*{conj_pi1-anti-nef}}}. The {\hyperref[sec_foliation-num-trivial]{\S \ref*{sec_foliation-num-trivial}}} is added after all the other parts of the article has been finished, where we discuss the foliations with numerically trivial canonical class by following the suggestions of St\'ephane Druel and give an alternative proof of {\hyperref[mainthm_MRC]{Theorem \ref*{mainthm_MRC}}}.

\paragraph{Acknowledgement}
The author would like to thank his thesis advisors Sébastien Boucksom and Junyan Cao for their enormous help to accomplish this work as well as to polish the presentation of the article. He would like to express his gratitude to St\'ephane Druel for helpful discussions on the important work \cite{Druel18b}. He also owes a lot to Chen Jiang and to Andreas H\"oring for helpful suggestions. After the first version of this article is put online, the author has received many helpful comments, and the author would like to thank the the commentators for their precious help: he would like to thank Jie Liu and Maciej Zdanowicz for pointing out the article \cite{PZ19}, to thank Shin-ichi Matsumura for pointing out an error in the previous version of the article and for bringing to the author's attention his work with S.Ejiri and M.Iwai \cite{EIM20}, to thank De-Qi Zhang for pointing out his work with R.V.Gurjar \cite{GZ94,GZ95} and to thank Beno\^it Claudon for pointing out an error in the previous version of the article and for helping the author to correct it. The author would like to take this opportunity to acknowledge the support he has benefited from the ANR project "GRACK" during the preparation of the present article.
\section{Preliminary results}
\label{sec_preliminary}

In this section, let us recall some preliminary results, which are surely well known to experts. 

\subsection{Weakly semipositively curved vector bundles}
\label{ss_preliminary_direct-image}
In this subsection we recall some definitions about (semipositive) singular Hermitian metrics on vector bundles. Throughout this subsection let $W$ be a complex manifold. For a (holomorphic) vector bundle $E$ over $W$, a singular Hermitian metric $h$ on $E$ is given by a measurable family of semipositive definite Hermitian functions on each fibre of $E$ which is non-singular almost everywhere. 
Singular metrics on line bundles have been introduced by J.-P. Demailly in \cite{Dem92}; the vector bundle case first appears explicitly in \cite{Raufi15}, while it is indicated in \cite{Raufi15} that the notion has already been implicitly mentioned in the previous works of B. Berndtsson.
Let $\theta$ be a smooth $(1,1)$-form on $W$, then $h$ is called $\theta$-semipositive if for every open subset $U$ of $W$ and for every local holomorphic section $s\in\Coh^0(W,E^\ast)$ of the dual bundle of $E$, the function $\log|s|_{h^\ast}^2$ is $(-\theta)$-plurisubharmonic ($(-\theta)$-psh), i.e. $\text{dd}^c\log|s|^2_{h^\ast}-\theta$ is a positive current on $U$. 
\begin{itemize}
\item $h$ is called {\it semipositively curved} if it is $\theta$-semipositive for $\theta=0$. When $h$ is a smooth Hermitian metric, then the notion coincide with the classical notion of Griffiths semipositivity. 
\item Suppose that $(W,\omega)$ is a compact K\"ahler manifold with K\"ahler metric $\omega$. Then the vector bundle $E$ is called $\theta$-{\it weakly semipositively curved} if for every $\epsilon>0$ small there exists a singular Hermitian metric $h_\epsilon$ on $E$ which is $(-\epsilon\omega+\theta)$-semipositive. If $\theta$ is a smooth form in the first Chern class of some ($\QQ$-)line bundle $L$, then a $\theta$-semipositive metric is also called {\it $L$-semipositive}. In particular, $E$ is called {\it weakly semipositively curved} if it is $\theta$-weakly semipositively curved for $\theta=0$. 
\end{itemize}
Let us remark that if $E$ is a line bundle on $W$ projective, then being (weakly) semipositively curved is equivalent to being pseudoeffective.

The $\theta$-semipositivity of singular Hermitian metrics is preserved by tensor products (up to multiplying $\theta$), pullback by proper surjective morphisms (up to pulling back $\theta$), and by generically surjective morphisms of vector bundles (thus by symmetric and exterior products, up to multiplying $\theta$). Moreover $\theta$-semipositive singular Hermitian metrics extend (and remaining $\theta$-semipositive) across closed analytic subsets of codimension at least $2$
and across closed analytic subsets of codimension $1$ under the condition that the metric is locally bounded (c.f. \cite[Proposition 2.4]{CH19}). In virtue of the aforementioned extension theorem and of \cite[Corollary 5.5.15, p.~147]{Kob87} one can naturally extend of the notion of $\theta$-semipositive singular Hermitian metrics to torsion free sheaves. 

\begin{rmq}[Comparison with the algebro-geometric notion of weak positivity]
\label{rmq_metric-weak-positivity}
Suppose that $W$ is projective. For a torsion free sheaf $\scrF$ on $W$ projective, being (weakly) semipositively curved implies the weak positivity (in the sense of Nakayama \cite{Nak04,Fuj17}), c.f. \cite{PT18,Pau16}; the reciprocal implication is also expected to be true (but still open), and can be regarded as a singular version of Griffiths's conjecture.
\end{rmq}

The principal aim of introducing the notion of $\theta$-semipositive singular Hermitian metrics is to provide a powerful tool for the study of (the canonical $L^2$ metrics on) the direct images of the twisted relative pluricanonical bundles. In fact, as a result of the Ohsawa-Takegoshi type extension theorem with optimal estimate \cite[Theorem 1.1]{Cao17}, we have the following 
\begin{thm}
\label{thm_positivity-direct-image}
Let $f:V\to W$ be a proper surjective morphism between K\"ahler manifolds with connected general fibre $V_w$ (a fibre space, c.f. {\hyperref[defn_local-const-fibration]{Definition \ref*{defn_local-const-fibration}}}) and let $(L,h_L)$ be a line bundle on $V$ equipped with a singular Hermitian metric $h_L$ such that the curvature current $\Theta_{h_L}(L)\geqslant f^\ast\theta$ for some smooth closed $(1,1)$-form on $Y$. Suppose that there is an $m\in \ZZ_{>0}$ such that $\scrJ(h_L^{1/m}|_{V_w})\simeq\scrO_{V_w}$ for general $w\in W$. Then the canonical $L^2$ metric $g_{V/W\!,L}^{(m)}$ on $f_\ast\scrO_V(mK_{V/W}+L)$ is $\theta$-semipositive on $Y$.  
\end{thm}
\begin{proof}
This is essentially proved in \cite[Lemma 5.4]{CP17}, see also \cite[2.8.Proposition]{CH19} and \cite[Theorem 2.2(1)]{CCM19}. For the convenience of the readers, we briefly recall the proof. As in \cite[Proof of Theorem B]{JWang19}, we construct the $m$-Bergman kernel metric $h_{V/W\!,h_L}^{(m)}$ on the twisted relative canonical bundle $mK_{V/W}+L$, and equip the line bundle $L_{m-1}:=(m-1)K_{V/W}+L$ with the metric
\[
h_{L_{m-1}}:=\left(h_{V/W\!,h_L}^{(m)}\right)\ptensor[\frac{m-1}{m}]\otimes h_{L}\ptensor[\frac{1}{m}];
\]
then the metric $g_{V/W\!,L}^{(m)}$ is constructed as the canonical $L^2$ metric on the direct image 
\[
\scrG_{m,L}:=f_\ast\scrO_V(K_{V/W}+L_{m-1})=f_\ast\scrO_V(mK_{V/W}+L).
\]
Since the construction of the $m$-Bergman kernel metric and of the $L^2$ metric is local over $W$ (c.f. \cite{BP08,Pau16,HPS18}), we can assume (by the $dd^c$-lemma) that $\theta$ is given by a weight function $\rho$, i.e. $\theta=dd^c\rho$. Then $h_{1,L}:=h_L\cdot e^{\;\rho\circ f}$ defines a new singular Hermitian metric on $L$ whose curvature current is positive:
\[
\Theta_{h_{1,L}}(L)=\Theta_{h_L}(L)-dd^c(\rho\circ f)=\Theta_{h_L}(L)-f^\ast\theta\geqslant 0.
\]
Now by \cite[Theorem 1.2]{Cao17} the $m$-Bergman kernel metric $h_{V/W\!,h_{1,L}}^{(m)}$ on the twisted relative canonical bundle $mK_{V/W}+L$ is semipositively curved. Now equip the line bundle $L_{m-1}$ with the singular Hermitian metric
\[
h_{1,L_{m-1}}:=\left(h_{V/W\!,h_{1,L}}^{(m)}\right)\ptensor[\frac{m-1}{m}]\otimes h_{L}\ptensor[\frac{1}{m}],
\]
since $\scrJ(h_{1,L}^{1/m}|_{V_w})\simeq\scrJ(h_L^{1/m}|_{V_w})\simeq\scrO_{V_w}$ for general $w\in W$, by \cite[Remark 2.12]{JWang19}, the natural inclusion
\[
f_\ast(\scrO_V(mK_{V/W}+L)\otimes\scrJ(h_{1,L_{m-1}}))\hookrightarrow\scrG_{m,L}
\]
is a generic isomorphism. Thus by \cite[Theorem 9.3]{DWZZ18} or \cite[Theorem 2.6]{JWang19} the canonical $L^2$ metric 
\[
g_{V/W\!,h_{1,L_{m-1}}}=g_{V/W\!,L}^{(m)}\cdot e^{-\rho}
\]
is semipositively curved. In other word, for every local section $s$ of the dual sheaf of $\scrG_{m,L}$\,, we have
\[
0\leqslant dd^c\log|s|_{g_{V/W\!,h_{1,L_{m-1}}}^\ast}^2=dd^c\log|s|^2_{g_{V/W\!,L}^{(m)\ast}}-dd^c\rho\,,
\]
which means that the metric $g_{V/W\!,L}^{(m)}$ is $\theta$-semipositive.
\end{proof}

As a result of the above {\hyperref[thm_positivity-direct-image]{Theorem \ref*{thm_positivity-direct-image}}}, we have:

\begin{cor}[{\cite[Theorem 2.2(2)]{CCM19}}]
\label{cor_positivity-direct-image}
Let $f:V\to W$, $(L,h_L)$ and $m$ as in the {\hyperref[thm_positivity-direct-image]{Theorem \ref*{thm_positivity-direct-image}}}. Assume further that $f$ is projective, $V$ and $W$ are compact and $L$ is $f$-big. Then for any nef line bundle $N$ on $V$, the direct image sheaf $f_\ast\scrO_V(mK_{V/W}+N+L)$ is $\theta$-weakly positively curved. 
\end{cor}
\begin{proof}
For the convenience of the readers, we briefly recall the proof. Since $L$ is $f$-big, there is a singular Hermitian metric $h$ on $L$ such that $\Theta_h(L)+f^\ast(\omega_W-\theta)\geqslant\omega_V$ as current for some K\"ahler form $\omega_W$ on $W$ (such that $\omega_W$ is still a K\"ahler form) and for some K\"ahler form $\omega_V$ on $V$. Since $N$ is nef, there are smooth Hermitian metrics $(g_\delta)_{\delta>0}$ on $N$ such that $\Theta_{g_\delta}(N)+\delta\omega_V\geqslant 0$.  Now consider the singular Hermitian metric
\[
h_\epsilon:=h_L^{1-\epsilon}\otimes h^\epsilon\otimes g_\epsilon
\]
on the line bundle $L\otimes N$. Then for $\epsilon$ sufficiently small (with respect to $h$) we have 
\[
\scrJ(h_\epsilon^{1/m}|_{V_w})=\scrJ(h_L^{(1-\epsilon)/m}\otimes h^{\epsilon/m}|_{V_w})\simeq\scrO_{V_w}
\]
for general $w\in W$. And by a direct computation we have
\[
\Theta_{h_\epsilon}(L\otimes N)\geqslant f^\ast(-\epsilon\omega_W+\theta). 
\]
Then {\hyperref[thm_positivity-direct-image]{Theorem \ref*{thm_positivity-direct-image}}} implies that $f_\ast\scrO_V(mK_{V/W}+N+L)$ is $\theta$-weakly semipositively curved. 
\end{proof}

\subsection{Numerical flatness and locally constant fibrations}
\label{ss_preliminary_local-constancy}

In this subsection we recall the notion of numerically flat vector bundles as well as its relation to the local constancy of fibre spaces; then we recall a fundamental criterion for numerical flatness.
First let us define the numerical flatness for vector bundles on compact K\"ahler manifolds (c.f. \cite[Definition 1.9 \& Definition 1.17]{DPS94}):
\begin{defn}
\label{defn_num-flat}
Let $W$ be a compact complex manifold. A holomorphic vector bundle $E$ on $W$ is said nef if the line bundle $\scrO_{\PP\!E}(1)$ on $\PP\!E$ is nef (c.f. \cite[Definition 1.2]{DPS94} for the more general definition of nefness of holomorphic line bundles on (non-necessarily algebraic) compact complex manifolds). The vector bundle $E$ is said to be numerically flat if both $E$ and its dual $E^\ast$ are nef. 
\end{defn}

As shown in \cite[Proposition 1.14 \& Proposition 1.15]{DPS94}, nefness of vector bundles is preserved by tensor products, by surjection of vector bundles, by pullbacks via surjective morphisms, and thus by symmetric and exterior products. Moreover, by \cite[Theorem 1.18]{DPS94}, \cite[Corollary 3.10]{Sim92} (c.f. also \cite[Ch.6, Theorem V]{Deng17b}) and \cite[Lemma 4.3.3]{Cao13}, we have the following structure result on numerically flat vector bundles:
\begin{thm}
\label{thm_DPS-Simpson_num-flat}
Let $W$ be a compact K\"ahler manifold and let $E$ be a numerically flat vector bundle on $W$. Then we have:
\begin{itemize}
\item[\rm(a)] $E$ admits a filtration 
\[
\{0\}=E_0\subsetneqq E_1\subsetneqq E_1\subsetneqq\cdots\subsetneqq E_k=E.
\]
where the $E_i$ are vector bundles and the quotients $E_{i+1}/E_i$ are irreducible Hermitian flat vector bundles, that is, induced by irreducible unitary representations $\pi_1(X)\to \Unitary(r_i)$.
\item[\rm(b)] $E$ is isomorphic to the underlying holomorphic vector bundle of a local system $L$, such that the natural Gauss-Manin connection $\nabla$ on $L$ is compatible with the filtration in {\rm(a)} and induces flat connections on the quotients $E_{i+1}/E_i$. In particular, every section of $\Coh^0(X,E)$ is parallel with respect to $\nabla$. 
\end{itemize}
\end{thm}

Next let us define:
\begin{defn}
\label{defn_local-const-fibration}
Let $f: V\to W$ be {\it a(n) (analytic) fibre space}, that is, a proper surjective morphism between complex analytic varieties with connected fibres. We call $f$ a {\it locally constant fibration} if $f$ is a locally trivial fibre bundle with fibre $F$ and there is a representation $\rho:\pi_1(W)\to\Aut(F)$ such that $V$ isomorphic to the quotient of $\widetilde W\times F$ by the action of $\pi_1(W)$ given by $\gamma\cdot(w,z)=(\gamma\cdot w,\rho(\gamma)z)$ where $\widetilde W$ denotes the universal cover of $W$.
\end{defn}

\begin{rmq}
\label{rmk_defn_lcf}
In the definition above, we see that $\pi_1(W)$ acts diagonally on $\widetilde W\times F$. Hence if we suppose in addition that $V$ is normal, then the natural decomposition $T_{\widetilde W\times F}\simeq\pr_1^\ast T_{\widetilde W}\oplus\pr_2^\ast T_F$ induces a splitting of the tangent sheaf of $V$ into foliations. 
\end{rmq}

As a corollary of {\hyperref[thm_DPS-Simpson_num-flat]{Theorem \ref*{thm_DPS-Simpson_num-flat}}} we have the following proposition which reveals the relation between local constancy of fibre spaces and numerical flatness of direct images (c.f. \cite[2.8.Proposition]{Cao19} and \cite[Proposition 2.8]{CCM19}; c.f. also \cite[4.1.Proposition]{CH17} and \cite[Proposition 4.3.6]{Cao13}):

\begin{prop}
\label{prop_num-flat--local-const}
Let $W$ be a compact K\"ahler manifold and let $f:V\to W$ be a flat projective morphism with connected fibres ($V$ is not necessarily smooth). Suppose that there is a $f$-very ample line bundle $L$ on $V$ such that for every $m>1$ the direct image $E_m:=f_\ast(mL)$ is a numerically flat vector bundle. Then $f$ is a locally constant fibration. 
\end{prop}
\begin{proof}
We will follow the main line of the argument in the proof of \cite[Proposition 2.8]{CCM19}. We nevertheless give some details in order to illustrate how the proof works for $V$ singular. Since $L$ is $f$-very ample, we have an embedding $i:V\to \PP\!E_1$ over $W$ with $i^\ast\scrO_{\PP\!E_1}\!(1)=L$. Let $\scrI_V$ be the ideal of $V$ in $\PP E_1$, we will show that (up to twisting with some power of $\scrO_{\PP E_1}\!(1)$) the generating polynomials of $\scrI_V$ have coefficients being constant functions over $W$.  

By relative Serre vanishing, for $m$ large enough we have a short exact sequence
\begin{equation}
\label{eq_ses-direct-image}
0\to p_\ast(\scrI_V\otimes\scrO_{\PP\!E_1}(m))\to p_\ast(\scrO_{\PP\!E_1}(m))\to E_m=f_\ast(mL)\to 0, 
\end{equation}
where $p$ denotes the natural morphism $\PP\!E_1\to W$. By hypothesis $E_1$ is a numerically flat vector bundle, then by {\hyperref[thm_DPS-Simpson_num-flat]{Theorem \ref*{thm_DPS-Simpson_num-flat}}} it is a local system, equipped with the Gauss-Manin connection $\nabla_{\!E_1}$. Take $\gamma:\widetilde W\to W$ the universal covering of $W$, then $\gamma^\ast E_1$ is trivial. And there are $r+1$ global sections $e_0,\cdots,e_r$ in $\Coh^0(\widetilde W,\gamma^\ast E_1)$ which are parallel with respect to $\nabla_{\!E_1}$ and generate $\gamma^\ast E_1$, where $r:=\rank\!E_1-1$. 

Now set $F_m:=p_\ast(\scrI_V\otimes\scrO_{\PP\!E_1}\!(m))$. The morphism $f$ being flat, $\scrI_V$ is flat over $W$, thus by the same argument as that in \cite[\S I\!I\!I.9, Proof of Theorem 9.9, pp.~261-262]{Har77}, $F_m$ is a vector bundle for $m$ sufficiently large. Then by the short exact sequence \eqref{eq_ses-direct-image} and by \cite[Proposition 1.15]{DPS94}, $F_m$ is numerically flat. Then again by {\hyperref[thm_DPS-Simpson_num-flat]{Theorem \ref*{thm_DPS-Simpson_num-flat}}} $F_m$ is a local system, equipped with the Gauss-Manin connection $\nabla_{\!F_m}$. By the same argument as above, $\gamma^\ast F_m$ is a trivial vector bundle and admits generating global sections $f_1,\cdots,f_{s_m}$ which are parallel with respect to $\nabla_{\!F_m}$, where $s_m:=\rank\!F_m$. 

Consider the inclusion 
\[
\eta:\gamma^\ast F_m\hookrightarrow\gamma^\ast p_\ast\scrO_{\PP\!E_1}\!(m)=\gamma^\ast \Sym^m \!E_1.
\]
By {\hyperref[thm_DPS-Simpson_num-flat]{Theorem \ref*{thm_DPS-Simpson_num-flat}}}, the sections $\eta(f_i)$ are all parallel with respect to the connection $\nabla_{\!\Sym^m\!E_1}$ induce by $\nabla_{\!E_1} $on $\Sym^m\!E_1$. Since $\Sym^m\!E_1$ is generated by the flat global sections 
\[
\left(e_0^{\alpha_0}\cdots e_r^{\alpha_r}\right)_{\alpha_j\in\ZZ_{>0},\alpha_0+\cdots+\alpha_r=m}
\]
we can write, for every $i=1,\cdots,s_m$ 
\[
\eta(f_i)=\sum_{\substack{\alpha=(\alpha_0,\cdots,\alpha_r)\in\ZZ_{>0}^{r+1} \\ |\alpha|=m}}c_{i,\alpha}\cdot e_0^{\alpha_0}\cdots e_r^{\alpha_r}\,,
\]
for some constants $c_{i,\alpha}\in\CC$. This then implies that the embedding of $\widetilde V:=V\underset{W}{\times}\widetilde W$ into $\widetilde W\times\PP^r$ over $\widetilde W$ is defined by polynomials whose coefficients are independent of $w\in \widetilde W$. Hence $\widetilde V$ splits into a product $\widetilde W\times F$. Since $E_1$ is a flat bundle, it is induced by a representation $\rho_1:\pi_1(W)\to\PGL(r+1)$. Let $\gamma\in\pi_1(W)$, then $\rho_1(\gamma)$ sends $V_{w}$ to $V_{\gamma(w)}$ viewed as subvarieties of $\PP^r$. But as seen before, the defining polynomial of $V_w$ in $\PP^r$ is independent of $w$ hence $\rho_1(\gamma)$ can be seen as an element of $\Aut(F)$, and hence a representation $\rho:\pi_1(W)\to\Aut(F)$. By construction $V$ is isomorphic to the quotient of $\widetilde V$ by the action of $\pi_1(W)$, and hence $f$ is locally constant fibration.  
\end{proof}

To finish this subsection let us recall the following numerical flatness criterion, which is proved in \cite[Proposition 2.7]{CCM19} when $W$ is projective and is extended to K\"ahler case by \cite[\S 1, Corollary of Main Theorem]{Wu20}:
\begin{prop}
\label{prop_criterion-num-flat}
Let $W$ be a compact K\"ahler manifold and let $\scrF$ be a reflexive sheaf on $W$. Suppose that $\scrF$ is weakly positively curved and  that $\det\!\scrF$ is numerically trivial. Then $\scrF$ is a numerically flat vector bundle on $W$.  
\end{prop}

\section{Positivity and numerical flatness of the direct images}
\label{sec_positivity}
Let $X$ be a klt projective variety with nef anticanonical divisor. In order to give a uniform treatment of the Albanese map and of the MRC fibration of $X$, we prove in this section some general results on the dominant rational mapping from $X$ to any smooth non-uniruled variety $Y$; in particular, by virtue of {\hyperref[prop_num-flat--local-const]{Proposition \ref*{prop_num-flat--local-const}}} we study the direct images of powers of a relatively very ample line bundle on $X$. Before stating these results, let us set up some general notations (see also \cite[Setting 3.1]{CCM19}): 
\begin{setting}
\label{general-setting}
Let $\psi: X\dashrightarrow Y$ be a dominant rational map between projective varieties with $Y$ smooth. Suppose that $X$ is of semi-Fano type, that is, there is an effective divisor $\Delta$ on $X$ such that the pair $(X,\Delta)$ is klt and $-(K_X+\Delta)$ is nef. Let $\phi:M\to Y$ be an elimination of indeterminacy of $\psi$ with $M$ smooth and let $\pi: M\to X$ be the induced (birational) morphism. For convenience, we further assume that the branch locus of $\phi$ is a SNC divisor on $Y$ and that its inverse image on $M$ has SNC support. Let $Y_0$ be the maximal Zariski open of $Y$ such that $\phi$ is flat over $Y_0$ and that for every prime divisor $D$ on $Y_0$ the pullback $\phi^\ast D$ is not contained in the exceptional locus of $\pi$. 
\end{setting}

\begin{center}
\begin{tikzpicture}[scale=2.5]
\node (A) at (0,0) {$Y$.};
\node (B) at (0,1) {$M$};
\node (C) at (1,1) {$X$};
\path[->,font=\scriptsize,>=angle 90]
(B) edge node[left]{$\phi$} (A)
(B) edge node[above]{$\pi$} (C);
\path[dashed, ->,font=\scriptsize,>=angle 90]
(C) edge node[below right]{$\psi$} (A);
\end{tikzpicture}
\end{center}
Write $\Exceptional(\pi)=\sum_{i\in I}E_i=:E$. Since $(X,\Delta)$ is klt, $K_X+\Delta$ is $\QQ$-Cartier and we can write:
\begin{equation}
\label{eq_canonical-pi}
K_{M}+\pi\inv_\ast\!\Delta\sim_{\QQ} \pi^\ast(K_X+\Delta)+\sum_i a_i E_i
\end{equation}
with $a_i>-1$, where $\pi\inv_\ast\Delta$ denotes the strict transform of $\Delta$ via $\pi$. We rewrite the formula above by:
\begin{equation}
\label{eq_canonical-pi-exc}
K_M+\Delta_M\sim_{\QQ}\pi^\ast(K_X+\Delta)+\sum_{i\in I_{>0}}a_iE_i
\end{equation}
where $I_{>0}$ (resp. $I_{>0}$) is the set of indices $i$ such that $a_i>0$ (resp. $a_i<0$) and  
\[
\Delta_M:=\pi\inv_\ast\!\Delta+\sum_{i\in I_{<0}}(-a_i)E_i.
\]
By the klt condition we see that the coefficients of the components in $\Delta_M$ are all $<1$ thus $(M,\Delta_M)$ is klt.

\subsection{Birational geometry of \texorpdfstring{$\psi$}{text}}
Let everything be as in the {\hyperref[general-setting]{General Setting \ref*{general-setting}}}. In this subsection we recall some general results on the birational geometry of $\psi$. They are essentially proved by Qi Zhang in \cite[Main Theorem]{Zhang05}. The following result is explicitly formulated in \cite[Theorem 3.2]{CCM19} for $X$ is smooth.
\begin{prop}
\label{prop_bir-geometry-psi}
Let everything be as in the {\hyperref[general-setting]{General Setting \ref*{general-setting}}} except that we only assume that the pair $(X,\Delta)$ is log canonical (abbr. lc). Suppose further that $Y$ is not uniruled. Then we have:   
\begin{itemize}
\item[\rm(a)] $\kappa(Y)=0$. Moreover, if $N_Y$ is an effective $\QQ$-divisor $\QQ$-linearly equivalent to $K_Y$, then $\phi^\ast N_Y$ is $\pi$-exceptional; in particular, $N_Y$ is contained in $Y\backslash Y_0$.
\item[\rm(b)]  $\Delta$ is horizontal with respect to $\psi$.
\item[\rm(c)] $\pi(\phi\inv(Y\backslash Y_0))$ is of codimension $\geqslant 2$ in $X$. In particular, every $\phi$-exceptional divisor on $M$ is also  $\pi$-exceptional.
\item[\rm(d)] $Y_0$ has the following Liouville property: every global holomorphic function on $Y_0$ is constant. 
\item[\rm(e)] $\psi$ is semistable in codimension $1$ (c.f. {\rm \cite[Definition 1]{Zhang05}}), i.e. for every prime divisor $P$ on $Y_0$\,, write $\phi^\ast P=\sum_{i}c_i P_i$ with $P_i$ being prime divisor
on $\phi\inv(Y_0)$ for every $i$, then $c_i>1$ implies that $P_i$ is $\pi$-exceptional.
\end{itemize}
\end{prop}

\begin{proof}
When $X$ is smooth, the proposition is established in \cite[Lemma 3.1, Proposition 3.2]{CH19}. In the singular case, the proof becomes a little subtle. For the convenience of the readers, we will briefly present the proof below following ideas from \cite{Zhang05} and \cite{CH19}. The same ideas are also used in the proof of {\hyperref[lemme_anticanonique-nef-plat]{Lemma \ref*{lemme_anticanonique-nef-plat}}} below. 

Up to further blowing-up $M$ and $Y$, we can assume that $\phi$ is smooth outside a SNC divisor $D_Y:=\sum_{j}D_{Y\!,j}$ (called the {\it branching divisor} of $\phi$) and that $\Supp(\phi^\ast D_Y+E)$ is SNC. In addition, let us fix a very ample line bundle $L$ on $X$.

Now take $A_M$ an ample divisor on $M$, then for any $\epsilon\in\QQ_{>0}$ the $\QQ$-divisor $-\pi^\ast(K_X+\Delta)+\epsilon A_M$ is ample since $-(K_X+\Delta)$ is nef; choose an ample $\QQ$-divisor $H_\epsilon$ on $Y$ such that $-\pi^\ast(K_X+\Delta)+\epsilon A_M-\phi^\ast H_\epsilon$ remains ample. Take
\[
\Delta_{M,\epsilon}:=\Delta_M+\frac{1}{k}\cdot\text{general member of the linear series }\left|k\left(-\pi^\ast(K_X+\Delta)+\epsilon A_M-\phi^\ast\!H_\epsilon\right)\right|,
\]
for $k$ sufficiently large and divisible. Then $\Delta_{M,\epsilon}$ is a $\QQ$-divisor with coefficients $\leqslant 1$ and has SNC support. By \cite[Corollary 2.31, pp.~53-55]{KM98} the pair $(M,\Delta_{M,\epsilon})$ is lc, thus by the weak positivity result \cite[Theorem 1.1]{Fuj17}, the direct image $\phi_\ast\scrO_M(k(K_{M/Y}+\Delta_{M,\epsilon}))$ is weakly positive; moreover, since $K_{M/Y}+\Delta_{M,\epsilon}$ is linearly equivalent to $\epsilon A+\sum_{a_i>0}a_i E_i$ over the general fibre of $\phi$, hence $K_{M/Y}+\Delta_{M,\epsilon}$ is relatively big, in particular we have
\[
\phi_\ast\scrO_M(k(K_{M/Y}+\Delta_{M,\epsilon}))\neq 0.
\]
In consequence, the $\QQ$-divisor
\[
K_{M/Y}+\Delta_{M,\epsilon}+\phi^\ast H_\epsilon\sim_{\QQ}-\phi^\ast K_Y+\sum_{i\in I_{>0}}a_i E_i +\epsilon A_M
\]
is $\QQ$-linearly equivalent to an effective $\QQ$-divisor (for details, see the proof of {\hyperref[lemme_anticanonique-nef-plat]{Lemma \ref*{lemme_anticanonique-nef-plat}}}); by letting $\epsilon\to 0$, we see that $-\phi^\ast K_Y+\sum_{a_i>0}a_i E_i$ is pseudoeffective.  

Finally take $H_1\,,\cdots, H_{\dim X-1}$ be general members of the linear series $\left|\pi^\ast L\right|$, and let
\[
C:= H_1\cap\cdots\cap H_{\dim X-1}\,,
\]
then $C$ is a movable curve on $M$, thus 
\[
(-\phi^\ast K_Y+\sum_{i\in I_{>0}}a_i E_i)\cdot C\geqslant 0.
\]
When $a_i>0$ the divisor $E_i$ is $\pi$-exceptional, then the projection formula implies that $E_i\cdot C=0$ for every $i$. Hence we have $\phi^\ast K_Y\cdot C\leqslant 0$. By our hypothesis $Y$ is not uniruled, then by \cite[Corollary 0.3]{BDPP13} $K_Y$ is pseudoeffective, since $C_Y:=\phi_\ast C$ moves in a strongly connecting family (c.f. \cite[\S 0]{BDPP13}), in particular it is movable, thus by \cite[Theorem 0.2]{BDPP13} $K_Y\cdot C_Y\geqslant 0$. But on the other hand, we have seen that $K_Y\cdot C_Y=\phi^\ast K_Y\cdot C\leqslant 0$, hence $K_Y\cdot C_Y=0$; then by \cite[9.8 Theorem]{BDPP13} we have $\kappa(Y)=0$. If $N_Y$ is an effective $\QQ$-divisor $\QQ$-linear equivalent to $K_Y$, then by the projection formula we have 
\[
\pi_\ast\phi^\ast N_Y\cdot L^{\dim X-1}=\phi^\ast N_Y\cdot C=0;
\]
but $L$ being very ample, a fortiori $\pi_\ast\phi^\ast N_Y=0$, meaning that $\phi^\ast\!N_Y$ is $\pi$-exceptional. This proves (a).

For the point (b), note that in the proof of (a), if we set 
\[
\Delta_{M,\epsilon}:=\pi\inv_\ast\!\Delta^{\horizontal}+\sum_{i\in I_{<0}}a_iE_i+\frac{1}{k}\cdot\text{general member of }\left|k\left(-\pi^\ast(K_X+\Delta)+\epsilon A_M-\phi^\ast\!H_\epsilon\right)\right|,\]
with $k$ sufficiently large and divisible, then the same argument as in (a) plus the equality $\phi^\ast\!K_Y\cdot C=0$ shows that $\Delta^{\vertical}\cdot C\leqslant 0$, but $\Delta^{\vertical}$ is effective, then a fortiori $\Delta^{\vertical}=0$, which implies that $\Delta$ is horizontal. Thus we proved (b). 

Now let us prove (c). Take a prime divisor $V$ on $M$ such that $\phi(V)\subseteq Y\backslash Y_0$. By definition of $Y_0$, if $\phi(V)$ is of codimension $1$, then $V$ is automatically $\pi$-exceptional; hence we can suppose that $\phi(V)$ is of codimension $\geqslant 2$, i.e. $V$ is $\phi$-exceptional. Let $\beta_Y: Y_1\to Y$ be a desingularization of the blow-up of $Y$ at $\phi(V)$, then $\phi(V)\subseteq\beta_Y(\Exceptional(\beta_Y))$. Since $Y$ is smooth, we have $K_{Y_1}\sim K_Y+F_Y$ with $F_Y$ effective and $\beta_Y$-exceptional, moreover we have  $\Supp(F_Y)=\Exceptional(\beta_Y)$. Take $M_1$ be a desingularization of the fibre product $M\underset{Y}{\times}Y_1$, with the induced morphisms $\beta_M:M_1\to M$ and $\phi_1:M_1\to Y_1$. And let $V_1$ be the strict transform of $V$ in $M_1$. Then $\phi_1(V_1)\subseteq\Exceptional(\beta_Y)$.

\begin{center}
\begin{tikzpicture}[scale=2.5]
\node (A) at (0,0) {$Y$.};
\node (B) at (0,1) {$M$};
\node (C) at (1,1) {$X$};
\node (A1) at (-1,0) {$Y_1$};
\node (B1) at (-1,1) {$M_1$};
\path[->,font=\scriptsize,>=angle 90]
(B) edge node[left]{$\phi$} (A)
(B) edge node[above]{$\pi$} (C)
(B1) edge node[left]{$\phi_1$} (A1)
(B1) edge node[above]{$\beta_M$} (B)
(A1) edge node[below]{$\beta_Y$} (A);
\path[dashed, ->,font=\scriptsize,>=angle 90]
(C) edge node[below right]{$\psi$} (A);
\end{tikzpicture}
\end{center}

By (b) there exists an effective $\QQ$-divisor $N_Y$ which is $\QQ$-linearly equivalent to $K_Y$. Then $\beta_Y^\ast N_Y+F_Y$ is an effective $\QQ$-divisor $\QQ$-linearly equivalent to $K_{Y_1}$. Apply (a) to the dominant rational map $X\dashrightarrow Y_1$ one sees that $\phi_1^\ast(\beta_Y^\ast N_Y+F_Y)$ is $(\pi\circ\beta_M)$-exceptional. But 
\[
\phi_1(V_1)\subseteq\Exceptional(\beta_Y)=\Supp(F_Y)\subseteq\Supp(\beta_Y^\ast N_Y+F_Y),
\]
therefore $V_1\subseteq \Supp(\phi^\ast(\beta_Y^\ast N_Y+F_Y))$ and thus $V_1$ is also $(\pi\circ\beta_Y)$-exceptional. This implies that $V=\beta_M(V_1)$ is $\pi$-exceptional. Thus we proved (c). 

Point (d) is a simple consequence of (c) by the same argument as \cite[\S 3.A, Remark 3]{CH19}. For convenience of the readers let us briefly recall the proof: let $h:Y_0\to \CC$ is a holomorphic function, then its pullback $\phi^\ast h$ induces a holomorphic function $h_1$ on $\pi(\phi\inv(Y_0)\backslash E)$. By (c) the complement of $\pi(\phi\inv(Y_0)\backslash E)$ in $X$ has codimension $\geqslant 2$. Then $h_1$ extends to a holomorphic function on $X$, which is constant by Liouville's Theorem. Hence $h$ is constant. 

It remains to prove (e). To this end it suffices to show the following statement: for every $j$ write $\phi^\ast D_{Y\!,j}=\sum_{l}m_{j,l}D_{j,l}$, if $m_{j,l}>1$ then $D_{j,l}$ is $\pi$-exceptional. By Kawamata's covering techniques (a Block-Gieseker cover followed by cyclic cover, c.f. \cite[Proposition 4.1.6, Theorem 4.1.10, Theorem 4.1.12, pp.~243-247]{Laz04}) we can construct a flat finite cover $p_Y: Y'\to Y$ such that $p_Y^\ast D_{Y\!,j}=m_{j,l}D_{Y'\!,j}$ for some smooth prime divisor $D_{Y'\!,j}$ on $Y'$ and that $Y'$ is smooth with $\sum_i p_Y^\ast E_i+\sum_{k\neq j}p_Y^\ast D_{Y\!,k}+D_{Y'\!,j}$ being a reduced SNC divisor. By \cite[Proposition 4.1.6]{Laz04} the fibre product $M\underset{Y}{\times}Y'$ is singular along the singular locus of the divisor $\phi^\ast D_{Y\!,j}$, in particular, it is singular along the preimage of $D_{j,l}$ since $m_{j,l}>1$. Take $M'$ a strong desingularization of $M\underset{Y}{\times}Y'$ with induced morphisms $p_M:M'\to M$ and $\phi':M'\to Y'$.   


\begin{center}
\begin{tikzpicture}[scale=2.5]
\node (A) at (0,0) {$Y$.};
\node (B) at (0,1) {$M$};
\node (C) at (1,1) {$X$};
\node (A1) at (-1,0) {$Y'$};
\node (B1) at (-1,1) {$M'$};
\path[->,font=\scriptsize,>=angle 90]
(B) edge node[left]{$\phi$} (A)
(B) edge node[above]{$\pi$} (C)
(B1) edge node[left]{$\phi'$} (A1)
(B1) edge node[above]{$p_M$} (B)
(A1) edge node[below]{$p_Y$} (A);
\path[dashed, ->,font=\scriptsize,>=angle 90]
(C) edge node[below right]{$\psi$} (A);
\end{tikzpicture}
\end{center}

By \cite[Proposition (9), Remark (26)(v\!i\!i)]{Kle80}, $M\underset{Y}{\times}Y'$ is Gorenstein and 
\[
K_{M\underset{Y}{\times}Y'/Y'}\sim \text{pullback of }K_{M/Y}\text{ to }M\underset{Y}{\times}Y'.
\]
over $p_Y\inv(Y_{\plat})$ where $Y_{\plat}\subseteq Y$ denotes the flat locus of $\phi$. By generic flatness and \cite[Example A.5.4, p.~416]{Ful84}, $Y\backslash Y_{\plat}$ is of codimension $\geqslant 2$, then so is $Y'\backslash p_Y\inv(Y_{\plat})$. By \cite[2.3 Proposition]{Reid94} we can write (for details, see the proof of {\hyperref[lemme_anticanonique-nef-plat]{Lemma \ref*{lemme_anticanonique-nef-plat}}})
\[
K_{M'/Y'}\sim_{\QQ} p_M^\ast\!K_{M/Y}+E_{M\underset{Y}{\times}Y'}+E_M-G
\]
where $E_{M\underset{Y}{\times}Y'}$ is a (non-necessarily effective) divisor which exceptional for $M'\to M\underset{Y}{\times}Y'$, $E_M$ is a (non-necessarily effective) divisor such that $\phi'(E_M)\subseteq Y'\backslash p_Y\inv(Y_{\plat})$ (in particular $E_M$ is $\phi'$-exceptional), and $G$ is an effective divisor supported on the preimage of the prime divisors with multiplicity $>1$ in $\phi^\ast D_{Y\!,j}$. In particular, $p_M(G)$ contains $D_{j,l}$. Combine this with the formula \eqref{eq_canonical-pi} we get
\[
K_{M'/Y'}\sim_{\QQ} p_M^\ast\pi^\ast(K_X+\Delta)-p_M^\ast\phi^\ast K_Y+\sum_{\lambda\in\Lambda} b_\lambda E'_\lambda+ E'_{M\underset{Y}{\times}Y'}+E_M-G.
\]
where 
$E'_{M\underset{Y}{\times}Y'}$ is exceptional for $M'\to M\underset{Y}{\times}Y'$ and, for every $\lambda\in\Lambda$, $E'_\lambda$ is prime divisor on $M'$ supported on the strict transform via $M'\to M\underset{Y}{\times}Y'$ of the pullback of $\sum_i E_i$ on $M\underset{Y}{\times}Y'$ with $b_\lambda:=a_{i_\lambda}\cdot\multiplicity_\lambda$ where $i_\lambda$ is the index such that $E_{i_\lambda}=p_M(E_\lambda)$ and 
\[
\multiplicity_\lambda:=\text{multiplicity of the image of }E_\lambda\text{ in the pullback of }E_i\text{ on }M\underset{Y}{\times}Y'.
\]
 By construction of $p_Y$ we see that $\multiplicity_\lambda >1$ if and only if $E_{i_\lambda}$ coincide with a divisor contained in the non-reduced part of $\phi^\ast D_{Y\!,j}$. In particular, for $\lambda\in\Lambda$ such that $\phi(E_{i_\lambda})\not\subset\Supp(D_Y)$ we have $\multiplicity_\lambda=1$ and thus $b_\lambda=a_{i_\lambda}\geqslant-1$. 

Now take $A'$ an ample divisor on $M'$. Since $-K_X$ is nef, for any $\epsilon\in\QQ_{>0}$ the $\QQ$-divisor $-p_M^\ast\pi^\ast K_X+\epsilon A'$ is ample; then choose an ample $\QQ$-divisor $H'_\epsilon$ on $Y'$ such that $-p_M^\ast\pi^\ast K_X+\epsilon A-(\phi')^\ast H_\epsilon$ remains ample. Take
\begin{align*}
\Delta_{M',\epsilon} &:=\sum_{\substack{\phi(E_{i_\lambda})\not\subset\Supp(D_Y) \\ b_\lambda<0}}(-b_\lambda)E'_\lambda \\
&+\frac{1}{k}\cdot\text{general member of the linear series }\left|k\left(-p_M^\ast\pi^\ast K_X+\epsilon A-(\phi')^\ast H_\epsilon\right)\right|,
\end{align*}
for $k$ sufficiently large and divisible. Then $(M',\Delta_{M',\epsilon})$ is a lc pair. Moreover, since the general fibre of $\phi$ and thus of $\phi'$ is smooth, $E'_{M\underset{Y}{\times}Y'}$ is $\phi'$-vertical; $E_M$ and $G$ are $\phi'$-vertical by construction. Therefore $K_{M'/Y'}+\Delta_{M',\epsilon}$ is big on the general fibre of $\phi'$. Hence by the same argument as in the proof of (a) we obtain that the $\QQ$-divisor 
\[
K_{M'/Y'}+\Delta_{M',\epsilon}+(\phi')^\ast H'_\epsilon\sim_{\QQ}-p_M^\ast\phi^\ast K_Y+\sum_{b_\lambda>0}b_\lambda E'_\lambda+E'_{M\underset{Y}{\times}Y'}+E_M-G-\sum_{\substack{b_\lambda\leqslant 0\\ \phi(E_{i_\lambda})\subset\Supp(D_Y)}} (-b_\lambda)E'_\lambda+\epsilon A_{M'}
\]
is $\QQ$-linearly equivalent to an effective $\QQ$-divisor; by letting $\epsilon\to 0$, we see that 
\[
-p_M^\ast\phi^\ast K_Y+E'_{M\underset{Y}{\times}Y'}+\sum_{b_\lambda>0}b_\lambda E'_\lambda+E_M-G-\sum_{\substack{b_\lambda\leqslant 0\\ \phi(E_{i_\lambda})\subset\Supp(D_Y)}}(-b_\lambda)E'_\lambda
\]
is pseudoeffective.  

Finally take $H'_1\,,\cdots, H'_{\dim X-1}$ be general members of the linear series $\left|p_M^\ast\pi^\ast L\right|$, and let
\[
C':= H'_1\cap\cdots\cap H'_{\dim X-1}\,,
\]
then $C'$ is a movable curve on $M'$, thus 
\[
\left(-p_M^\ast\phi^\ast K_Y+\sum_{b_\lambda>0}b_\lambda E'_\lambda+E'_{M\underset{Y}{\times}Y'}+E_M-G-\sum_{\substack{b_\lambda\leqslant 0\\ \phi(E_{i_\lambda})\subset\Supp(D_Y)}}-(b_\lambda)E'_\lambda\right)\cdot C'\geqslant 0.
\]
By construction, $E'_\lambda$ is $(\pi\circ p_M)$-exceptional for $\lambda$ such that $b_\lambda>0$, so is $E'_{M\underset{Y}{\times}Y'}$, hence $E_\lambda\cdot C'=E'_{M\underset{Y}{\times}Y'}\cdot C'=0$ for $\lambda$ such that $b_\lambda>0$. Furthermore, by construction $E_M$ is $\phi'$-exceptional, hence $p_{M\ast}E_M$ is $\phi$-exceptional, then by (c) $E_M$ is $\pi$-exceptional and $E_M\cdot C'=0$. Therefore we have 
\[
\left(\sum_{\substack{b_\lambda\leqslant 0\\ \phi(E_{i_\lambda})\subset\Supp(D_Y)}}(-b_\lambda)E'_\lambda+G\right)\cdot C'\leqslant -p_M^\ast\phi^\ast K_Y\cdot C'=-K_Y\cdot(\phi\circ p_M)_\ast C'.
\]
Since $K_Y$ is pseudoeffective (by assumption $Y$ is not uniruled), $(\phi\circ p_M)_\ast C'$ is movable, we have
\[
\left(\sum_{\substack{b_\lambda\leqslant 0\\ \phi(E_{i_\lambda})\subset\Supp(D_Y)}}(-b_\lambda)E'_\lambda+G\right)\cdot C'\leqslant -K_Y\cdot(\phi\circ p_M)_\ast C'\leqslant 0.
\]
But $G$ is effective, a fortiori $G\cdot C'=0$ and $b_\lambda=0$ or $E'_\lambda\cdot C'=0$ for every $\lambda\in\Lambda_{\vertical}$ such that $\phi(E_{i\lambda})\subset\Supp(D_Y)$. By the projection formula this implies that 
\[
(\pi\circ p_M)_\ast G\cdot L^{\dim X-1}=0.
\]
Since $L$ is very ample, we have $(\pi\circ p_M)_\ast G=0$. In particular, since $p_{M\ast}G\supseteq D_{j,l}$, this implies that $D_{j,l}$ is $\pi$-exceptional, which proves (d). By the way, the same argument shows that $b_\lambda=0$ for every $\lambda\in\Lambda_{\vertical}$ such that $\phi(E_{i_\lambda})\subset\Supp(D_Y)$.  
\end{proof}


\subsection{Positivity and numerical flatness of the direct images}
Throughout this subsection, let everything be as in the {\hyperref[general-setting]{General Setting \ref*{general-setting}}}, and suppose further that $Y$ is not uniruled. The main purpose of this subsection is to study the positivity of the $\phi$-direct images of a sufficiently ample line bundle on $M$. Before stating these results, let us fix some notations: by \eqref{eq_canonical-pi-exc} the $\QQ$-divisor
\begin{equation}
\label{eq_rel-canonical-pi-exc}
-(K_{M/Y}+\Delta_M)+E'\sim_{\QQ} -\pi^\ast(K_X+\Delta)
\end{equation}
is nef, where $E':=\sum_{i\in I_{>0}}a_iE_i-\phi^\ast N_Y$ with $N_Y$ being an effective $\QQ$-divisor $\QQ$-linearly equivalent to $K_Y$ (by {\hyperref[prop_bir-geometry-psi]{Proposition \ref*{prop_bir-geometry-psi}(b)}} such an $N_Y$ exists). By {\hyperref[prop_bir-geometry-psi]{Proposition \ref*{prop_bir-geometry-psi}(a)}} $E'$ is $\pi$-exceptional and the restriction of $E'$ to a general fibre of $\phi$ is effective. The basic result in this subsection is the following (c.f. \cite[Lemma 3.4]{CCM19}):

\begin{prop}
\label{prop_positivity-anti-nef}
Let everything be as in the {\hyperref[general-setting]{General Setting \ref*{general-setting}}} with $Y$ non-uniruled and $E'$ as above. Let $\theta$ be a smooth $(1,1)$-form on $Y$ and let $G$ be a $\phi$-big divisor on $M$ such that $\scrO_M(G)$ admits a singular Hermitian metric $h_G$ such that $\Theta_{h_G}(\scrO_M(G))\geqslant\phi^\ast\theta$. Then for any $q\in\ZZ$ the direct image sheaf $\phi_\ast\scrO_M(q(K_{M/Y}+\Delta_M)+G+pE')$ is $\theta$-weakly semipositively curved for any $p\in\ZZ_{>0}$ sufficiently large with respect to $q$.     
\end{prop}

\begin{proof}
This can be deduced immediately from the {\hyperref[cor_positivity-direct-image]{Corollary \ref*{cor_positivity-direct-image}}}. Let us briefly recall how the proof goes. Let $p\in\ZZ_{>0}$ such that $pE'$ is an integral divisor on $M$, and write
\[
q(K_{M/Y}+\Delta_M)+G+pE'=(p+q)(K_{M/Y}+\Delta_M)+G+\underbrace{(-p(K_{M/Y}+\Delta_M)+pE')}_{\text{nef}}.
\]
Let $h_{\Delta_M}$ be the canonical metric on $\Delta_M$. Since $(M,\Delta_M)$ is still klt, $\scrJ(h_{\Delta_M})\simeq\scrO_M$, thus by \cite[\S 9.5.D, Theorem 9.5.35, pp.~210-211]{Laz04} $\scrJ(h_{\Delta_M}|_{M_y})\simeq\scrO_{M_y}$ for general $y$. Hence
for $p$ sufficiently large  
\[
\scrJ((h_G^{1/(p+q)}\otimes h_{\Delta_M})|_{M_y})\simeq\scrO_{M_y}
\]
for general $y\in Y$. For such a $p$ {\hyperref[cor_positivity-direct-image]{Corollary \ref*{cor_positivity-direct-image}}} implies that $\phi_\ast\scrO_M(q(K_{M/Y}+\Delta_M)+G+pE')$ is $\theta$-weakly semipositively curved.  
\end{proof}

Then let us recall in the sequel some results in \cite{CH19}. We first remark that:
\begin{rmq}
\label{rmk_CCM19}
Let us remark that most of the results below have been essentially contained in \cite{CCM19}. We carefully state and prove them for the following reason: since $X$ is not necessarily smooth (nor $\QQ$-factorial), the pushforward of a (Cartier) $\QQ$-divisor on $M$ via $\pi$ is not necessarily $\QQ$-Cartier, thus in general it does not make sense to talk about pseudoeffectivity of them (\cite{CCM19} does not take care of this point). However, since the effectivity of a Weil divisor still makes sense, we will use this to overcome this difficulty.
\end{rmq}

\begin{prop}[{\cite[Lemma 3.5]{CCM19}}]
\label{prop_more-psef-det}
Let everything be as in the {\hyperref[general-setting]{General Settinng \ref*{general-setting}}} with $Y$ non-uniruled and let $G$ be a $\phi$-big divisor on $M$, then for any ample divisor $A_Y$ on $Y$ and for any integers $c,s\in\ZZ_{>0}$ the $\QQ$-divisor
\[
\pi_\ast\left(G-\phi^\ast\!D_{G,c,1}+\frac{1}{s}\phi^\ast A_Y\right)
\]
is $\QQ$-linearly equivalent to an effective divisor on $X$, where $D_{G,c,1}$ is the $\QQ$-divisor on $Y$ defined by 
\[
D_{G,c,1}:=\frac{1}{r}\cdot\text{the Cartier divisor on }Y\text{ associated to the line bundle }\det\!\phi_\ast\scrO_M(G+cE)
\]
with $r=\rank\!\phi_\ast\scrO_M(G+cE)$. If moreover $\pi_\ast(G-\phi^\ast\!D_{G,c,1})$ is $\QQ$-Cartier on $X$, then it is pseudoeffective; in particular, for $k\in\ZZ_{>0}$ sufficiently large $G-\phi^\ast\!D_{G,c,1}+kE$ is pseudoeffective on $M$. 
\end{prop}
\begin{proof}
The proof is essentially the same as that of \cite[Theorem 3.4]{CP17} (c.f. \cite[Theorem 3.4]{JWang19} for more details), see also \cite[Proposition 3.15]{Cao19}, \cite[Lemma 3.3]{CH19} and \cite[Lemma 3.5]{CCM19}. We give the detailed proof in order to clarify the problems pointed out in {\hyperref[rmk_CCM19]{Remark \ref*{rmk_CCM19}}}

\paragraph{(A) Construction of the fibre product and of the canonical section.}
Let $Y_{\ff}$ be the Zariski open subset of $Y$ over which $\phi$ is flat and $\phi_\ast\scrO_M(G+cE)$ is locally free. Then $\codim_Y(Y\backslash Y_{\ff})\geqslant 2$ and for every $y\in Y_{\ff}$ the fibre $M_y$ is Gorenstein (c.f. \cite[\S 23, Theorem 23.4, p.~181]{Mat89}). Over $Y_{\ff}$ we have a natural inclusion
\begin{equation}
\label{eq_incl-det-tensor}
\det\!\phi_\ast\scrO_M(G+cE)|_{Y_{\ff}}\simeq\scrO_{Y_{\ff}}(rD_{G,c,1}|_{Y_{\ff}})\hookrightarrow\bigotimes^r\phi_\ast\scrO_M(G+cE)|_{Y_{\ff}}.
\end{equation}
Now we take the $r$-fold fibre product
\[
M^r:=\underbrace{M\underset{Y}{\times}M\underset{Y}{\times}\cdots\underset{Y}{\times}M}_{r\text{ times}},
\]
equipped with natural projections $\pr_i:M^r\to M$ and the natural morphism $\phi^r:M^r\to Y$ such that $\phi\circ\pr_i=\phi^r$ for every $i$. Set
\begin{align*}
G^r &  :=\sum_{i=1}^r\pr_i^\ast G,\\
E^r &:=\sum_{i=1}^r\pr_i^\ast E,\\
\Delta_{M^r} &:=\sum_{i=1}^r\pr_i^\ast\Delta_M. 
\end{align*}
Let $\mu:M^{(r)}\to M^r$ be a strong desingularization of $M^r$ such that $\mu|_{\mu\inv(M^r_{\reg})}$ is an isomorphism, and set $p_i:=\pr_i\circ\mu$, $\phi^{(r)}:=\phi^r\circ\mu$, $G^{(r)}:=\mu^\ast G^r$, $E^{(r)}:=\mu^\ast E^r$, $\Delta_{M^{(r)}}:=\mu^\ast\Delta_{M^r}$. By the projection formula and by induction we have 
\[
\phi^{(r)}_\ast\scrO_{M^{(r)}}(G^{(r)}+cE^{(r)})|_{Y_{\ff}}\simeq\phi^r_\ast\scrO_{M^r}(G^r+cE^r)|_{Y_{\ff}}\simeq\bigotimes^r\phi_\ast\scrO_M(G+cE)|_{Y_{\ff}}\,, 
\]
Then \eqref{eq_incl-det-tensor} induces a non-zero section
\[
s_0\in\Coh^0(Y_{\ff}\,,\phi^{(r)}_\ast\scrO_{M^{(r)}}(G^{(r)}+cE^{(r)})\otimes(\det\!\phi_\ast\scrO_M(G+cE))\inv),
\]
By \cite[\S I\!I\!I.5, 5.10.Lemma, pp.~107-108]{Nak04} (c.f. \cite[Theorem 1.13]{JWang19} for more details) there is an effective divisor $B_1$ supported in $M^{(r)}\backslash(\phi^{(r)})\inv(Y_{\ff})$ such that $s_0$ extends to a non-zero section 
\[
\bar s_0\in\Coh^0(M^{(r)},\scrO_{M^{(r)}}(G^{(r)}+cE^{(r)}+B_1-r(\phi^{(r)})^\ast D_{G,c,1})),
\]
in particular $\Delta_0:=G^{(r)}+cE^{(r)}+B_1-r(\phi^{(r)})^\ast D_{G,c,1}$ is (linearly equivalent to) an effective divisor on $M^{(r)}$.   
\paragraph{(B) Comparison of the relative canonical divisors.}
By induction and the base change formula of the relative canonical sheaf \cite[Proposition (9)]{Kle80} we see that $M^r_{\ff}$ is Gorenstein and the relative dualizing sheaf
\[
\omega_{M^r_{\ff}/Y}\simeq\scrO_{M^r_{\ff}}(\sum_{i=1}^r\pr_i^\ast K_{M/Y}).
\] 
The natural morphism $\omega_{M^r_{\ff}/Y}\to \mu_\ast\scrO_{M^{(r)}}(K_{M^{(r)}/Y})|_{M^r_{\ff}}$ (from \cite[\S I\!I.8, Proposition 8.3, p.~176]{Har77}) is an isomorphism over $M_{\rat}^r$, the rational singularities locus of $M^r$. By assumption (c.f. {\hyperref[general-setting]{General Setting \ref*{general-setting}}}), the branch locus $\Branch(\phi)$ of $\phi$ is a SNC divisor on $Y$ and $f^\ast\Branch(\phi)$ has SNC support. Write
\[
f^\ast\Branch(\phi):=\sum_\lambda W_\lambda+\sum_\mu a_\mu V_\mu
\]
with $a_\mu>1$ for every $\mu$ and set 
\[
W:=\sum_\lambda W_\lambda\,,\quad\quad V:=\sum_\mu V_\mu\,,
\]
then by \cite[3.13.Lemma]{Hor10} $M^r_{\ff}$ has rational singularities along 
\[
\underbrace{(M_{\ff}\backslash (V\cup \phi\inv\Sing(\Branch(\phi))))\underset{Y_{\ff}\backslash\Sing(\Branch(\phi))}{\times}\cdots\underset{Y_{\ff}\backslash\Sing(\Branch(\phi))}{\times}(M_{\ff}\backslash(V\cup \phi\inv\Sing(\Branch(\phi)))).}_{r\text{ times}}
\]
Hence there is a divisor $B_2$ on $M^{(r)}$ supported on 
\[
E^{(r)}\cup(M^{(r)}\backslash \mu\inv(M^r_{\ff}))\cup\Supp(\sum_{i=1}^r\pr_i^\ast V)
\]
such that 
\[
-(K_{M^{(r)}/Y}+\Delta_{M^{(r)}})+B_2\sim \sum_{i=1}^r\pr_i^\ast(-(K_{M/Y}+\Delta_M)+E').
\]

\paragraph{(C) Ohsawa-Takegoshi type extension.}
For $y\in Y$ general, the general fibre
\[ 
M_y^r:=\underbrace{M_y\times\cdots\times M_y}_{r\text{ times}}
\]
of $\phi^r$ is smooth; since $\mu$ is an isomorphism over $M^r_{\reg}$, $M^r_y$ is also the general fibre of $\phi^{(r)}$. Now fix a sufficiently ample divisor $A_Y$ on $Y$ divisible by $2$, such that $\frac{1}{2}A_Y-K_Y$ separates all the $(2\dim\!Y)$-jets. For $s\in\ZZ_{>0}$, since $\Delta_0=G^{(r)}+cE^{(r)}+B_1-r(\phi^{(r)})^\ast D_{G,c,1}$ is $\phi^{(r)}$-big, there is $\epsilon\in\QQ_{>0}$ sufficiently small such that $\epsilon s\Delta_0+A_Y$ is big. Then we can write 
\[
\epsilon s\Delta_0+\frac{1}{2}A_Y\sim_{\QQ} H_{s,\epsilon}+\Delta_{s,\epsilon}.
\]
with $H_{p,\epsilon}$ an ample $\QQ$-divisor and $\Delta_{p,\epsilon}$ an effective $\QQ$-divisor. Now let $t\in\ZZ_{>0}$ sufficiently large such that 
\[
\left(\Delta_{M^{(r)}}+\frac{1}{st}\Delta_{s,\epsilon}+\frac{1-\epsilon}{t}\Delta_0\right)\big|_{M_y^r}=\left(\Delta_{M^r}+\frac{1}{st}\Delta_{s,\epsilon}+\frac{1-\epsilon}{t}\Delta_0\right)\big|_{M^r_y} 
\]
is klt. Since $st(-K_{M^{(r)}/Y}-\Delta_{M^{(r)}}+B_2)+H_{s,\epsilon}$ is ample, we can apply \cite[Theorem 2.10]{Cao19} (c.f. also \cite[Theorem 2.11]{Deng17a}) to the divisor
\begin{align*}
L: &=st(-K_{M^{(r)}/Y}+B_2)+s\Delta_0+\frac{1}{2}A_Y \\
&\sim_{\QQ}[st(-K_{M^{(r)}/Y}-\Delta_{M ^{(r)}}+B_2)+H_{s,\epsilon}]+(st\Delta_{M^{(r)}}+\Delta_{s,\epsilon}+(1-\epsilon)s\Delta_0)
\end{align*}
to obtain the surjectivity of the restriction morphism
\[
\Coh^0(M^{(r)},\scrO_{M^{(r)}}(stK_{M^{(r)}/Y}+L+\frac{1}{2}A_Y))\to\Coh^0(M^r_y, \scrO_{M_y^r}(stK_{M^{(r)}/Y}+L+\frac{1}{2}A_Y)), 
\]
which can be rewritten as
\[
\Coh^0(M^{(r)},\scrO_{M^{(r)}}(sG^{(r)}+scE^{(r)}+sB_1+stB_2+(\phi^{(r)})^\ast (A_Y-srD_{G,c,1})))\twoheadrightarrow\Coh^0(M_y^r,\scrO_{M^r_y}(sG^r+scE^r))
\]

\paragraph{(D) Restriction to the diagonal and conclusion.}
Now take a non-zero section (since $G+cE$ is $\phi$-big, such a section exists)
\[
u\in\Coh^0(M_y\,,\scrO_{M_y}(pG+pcE))
\]
for $y\in Y$ general, then
\[
u^{(r)}:=\sum_{i=1}^r \pr_i^\ast u\in\Coh^0(M_y^r\,,\scrO_{M_y^r}(sG^r+scE^r)).
\]
By {\bf Step (C)} we get a section 
\[
\sigma^{(r)}\in\Coh^0(M^{(r)}, \scrO_{M^{(r)}}(sG^{(r)}+scE^{(r)}+sB_1+stB_2+(\phi^{(r)})^\ast (A_Y-prD_{G,c,1}))))\]
such that $\sigma^{(r)}|_{M_y^r}=u^{(r)}$. Since $\mu$ is an isomorphism over $(M\backslash\Supp(V+W))^r\subseteq(M^r)_{\reg}$, then $\sigma^{(r)}|_{(M\backslash\Supp(V+W))^r}$ can be restricted to the diagonal and gives rise to a section 
\[
\sigma'\in\Coh^0(M\backslash\Supp(V+W), \scrO_M(srG+srcE+F'_{s,t}+\phi^\ast(A_Y-srD_{G,c,1})))
\]
for some $F'_{s,t}$ supported in $\Supp(E)$ (by {\hyperref[prop_bir-geometry-psi]{Proposition \ref*{prop_bir-geometry-psi}(c)}} any $\phi$-exceptional divisor is also $\pi$-exceptional thus contained in $\Supp(E)$). By construction of $B_1$ and $B_2$ we know that $\sigma'$ is bounded around a general point of $W$; moreover, by {\hyperref[prop_bir-geometry-psi]{Proposition \ref*{prop_bir-geometry-psi}(d)}} $V$ is contained in $\Supp(E)$, hence there is a $\pi$-exceptional divisor $F_{s,t}$ such that $\sigma'$ extends to a section 
\[
\sigma\in\Coh^0(M, \scrO_M(srG+srcE+F_{s,t}+\phi^\ast(A_Y-srD_{G,c,1}))).
\]
By construction $\sigma|_{M_y}=u\ptensor[r]$, hence $\sigma\neq 0$, which implies that $srG-+srcE+F_{s,t}+\phi^\ast(A_Y-srD_{G,c,1})$ is linearly equivalent to an effective divisor on $M$. But $E$ and $F_{s,t}$ are $\pi$-exceptional, hence
\[
\pi_\ast\left(G-\phi^\ast\!D_{G,c,1}+ \frac{1}{s}\phi^\ast\!A_Y\right)
\]
is $\QQ$-linearly equivalent to an effective (Weil) $\QQ$-divisor on $X$. Since this holds for any $s\in\ZZ_{>0}$\,, we can take $A_Y$ to be any ample divisor on $Y$.

If we assume moreover that $\pi_\ast(G-\phi^\ast D_{G,c,1})$ is $\QQ$-Cartier, then by taking a sufficiently ample divisor $A$ on $X$ containing $\pi_\ast\phi^\ast\!A_Y$, we see that $\pi_\ast(G-\phi^\ast\!D_{G,c,1})+\frac{1}{s}A$ is $\QQ$-linearly equivalent to an effective (Cartier) $\QQ$-divisor, hence $\pi_\ast(G-\phi^\ast\!D_{G,c,1})$ is pseudoeffective. In particular, for $k\in\ZZ_{>0}$ sufficiently large $G-\phi^\ast D_{G,c,1}+kE$ is pseudoeffective on $M$.  
\end{proof}

\begin{prop}[{\cite[Lemma 3.5]{CH19},\cite[Proposition 3.6]{CCM19}}]
\label{prop_det-stablize}
Let everything be as in the {\hyperref[general-setting]{General Settinng \ref*{general-setting}}} with $Y$ non-uniruled and let $G$ be a $\phi$-big divisor. Then there is an $c_0\in\ZZ_{>0}$ such that for every $c\geqslant c_0$ the natural inclusion $\det\!\phi_\ast\scrO_M(G+cE)\to \det\!\phi_\ast\scrO_M(G+(c+1)E)$ is an isomorphism over $Y_0$.  
\end{prop}

\begin{proof}
If $X$ is smooth, then $E$ cannot dominate $Y$ and the proposition is proved in \cite[Lemma 3.5]{CH19}. In our case, $X$ is not necessarily smooth and it takes more effort to prove the proposition. We will follow the same argument of \cite[Proposition 3.6]{CCM19} with some clarifications (c.f. {\hyperref[rmk_CCM19]{Remark \ref{rmk_CCM19}}}). The proof can be divided into two steps:
\paragraph{Step 1: Constancy of the rank of the direct images with respect to $c$.}  
Since $\rank\!\phi_\ast\scrO_M(G+cE)=\dimcoh^0(M_y,\scrO_{M_y}(G+cE))$ for $y\in Y$ general, and since $E$ is effective and $\pi$-exceptional, it suffices to prove that $\dimcoh^0(M_y, \scrO_{M_y}(G+cE))$ is bounded by a constant for all $c\in\ZZ_{>0}$. By \cite[Theorem 2.10]{Cao19} and by the argument as in {\bf Step (C)} of the proof of {\hyperref[prop_more-psef-det]{Proposition \ref*{prop_more-psef-det}}}, for $p$ sufficiently large and for $A_Y$ sufficiently ample on $Y$ divisible by $2$ and such that $\frac{1}{2}A_Y-K_Y$ separates all the $(2\dim\!Y)$-jets,  we have a surjection 
\begin{align*}
\Coh^0(M,\scrO_M(G+cE+pE'+\phi^\ast A_Y)) &= \Coh^0(M,\scrO_M(p(K_M+\Delta_M)+p(-K_M-\Delta_M+E')+G+cE+\phi^\ast\!A_Y)) \\
&\twoheadrightarrow \Coh^0(M_y, \scrO_{M_y}(G+cE+pE')),
\end{align*}
for $y\in Y$ general. Since $E'|_{M_y}$ is effective, we have 
\[
\dimcoh^0(M_y,\scrO_{M_y}(G+cE))\leqslant\dimcoh^0(M_y, \scrO_{M_y}(G+cE+pE'))\leqslant\dimcoh^0(M, \scrO_M(G+cE+pE'+\phi^\ast\!A_Y)).
\]
It remains to see the boundedness of $\dimcoh^0(M, \scrO_M(G+cE+pE'+\phi^\ast\!A_Y))$. By \cite[\S I\!I\!I.5, 5.10.Lemma, pp.~107-108]{Nak04}, for $c$ and $p$ sufficiently large, 
\[
\pi_\ast\scrO_M(G+(pk+c)E+\phi^\ast\!A_Y)\simeq\left(\pi_\ast\scrO_M(G+\phi^\ast\!A_Y)\right)^{\ast\ast}
\]
for any $k\in\ZZ_{>0}$. Hence for sufficiently large $c$ and for $p$ sufficiently large with respect to $c$ and $G$ we have
\[
\dimcoh^0(M,\scrO_M(G+cE+pE'+\phi^\ast\!A_Y))\leqslant\dimcoh^0(M,\scrO_M(G+(c+pk)E+\phi^\ast\!A_Y))=\dimcoh^0(X,\left(\phi_\ast\scrO_M(G+\phi^\ast\!A_Y)\right)^{\ast\ast}),
\]
where $k$ is a positive integer such that $E'\leqslant kE$. In consequence $\dimcoh^0(M,\scrO_M(G+cE+pE'+\phi^\ast\!A_Y))$ is bounded by a constant independent of $c$ and $p$, and so is $\rank\!\phi_\ast\scrO_M(G+cE)$. In other word, there is $c_0\in\ZZ_{>0}$ such that for any $c\geqslant c_0$, the rank of $\phi_\ast\scrO_M(G+cE)$ is independent of $c$.  

\paragraph{Step 2: Stability of the determinant sheaf over $Y_0$.}
By contradiction, let us assume that there is an increasing sequence $(c_k)_{k\in\ZZ_{>0}}$ such that $c_1\geqslant c_0$, $c_k\!\nearrow\!+\infty$ and that there is some effective divisor $B_k$ on $Y$ such that $B_k\cap Y_0\neq\varnothing$ (in particular $B_k\neq0$) and
\[
rD_{G,c_{k+1},1}-(rD_{G,c_k,1}+B_k)
\]
is linearly equivalent to an effective divisor on $M$ for every $k$, where $r:=\rank\!\phi_\ast\scrO_M(G+c_0E)$ and
\[
D_{G,c,1}:=\frac{1}{r}\cdot\text{ the Cartier divisor on }Y\text{ associated to }\det\!\phi_\ast\scrO_M(G+cE).
\]
By {\bf Step 1} for any $c\geqslant c_0$, $\rank\!\phi_\ast\scrO_M(G+cE)=r$. Then by {\hyperref[prop_more-psef-det]{Proposition \ref*{prop_more-psef-det}}} for any ample divisor $A_Y$ on $Y$ and for any $p\in\ZZ_{>0}$ the $\QQ$-divisor 
\[
\pi_\ast\left(G-\phi^\ast\!D_{G,c_k,1}+\frac{1}{s}\phi^\ast\!A_Y\right)
\]
is $\QQ$-linearly equivalent to an effective divisor. In particular, take $s=r$, we see that for every $N>0$ we have
\[
r\pi_\ast G+\pi_\ast\phi^\ast\!A_Y-\sum_{k=1}^{N}\pi_\ast\phi^\ast\!B_k
\] 
is linear equivalent to a Weil divisor on $X$. But since $B_k\cap Y_0\neq\varnothing$, $\phi^\ast\!B_k$ is not $\pi$-exceptional, hence $\pi_\ast\phi^\ast\!B_k$ is non-zero effective for every $k$. By letting $N\to+\infty$ we see that this is impossible.
\end{proof}

As an immediate corollary of {\hyperref[prop_det-stablize]{Proposition \ref*{prop_det-stablize}}} we have 
\begin{cor}
\label{cor_iso-direct-image-Y0}
Let everything be as in the {\hyperref[general-setting]{General Settinng \ref*{general-setting}}} with $Y$ non-uniruled and let $G$ be a $\phi$-big divisor. Let $c_0$ be the integer given by the {\hyperref[prop_det-stablize]{Proposition \ref*{prop_det-stablize}}} and let $c\geqslant c_0$. For every $a\in\ZZ_{>0}$ set 
\[
D_{G,c,a}:=\frac{1}{r_a}\cdot\text{the Cartier divisor on }Y\text{ associated to the line bundle }\det\!\phi_\ast\scrO_M(aG+acE)
\]
where $r_a:=\rank\!\phi_\ast\scrO_M(aG+acE)$. Then 
\begin{itemize}
\item[\rm(a)] $\phi_\ast\scrO_M(G+cE)$ is isomorphic to $\phi_\ast\scrO_M(G+kE+pE')$ over $Y_0$ for any $k\geqslant c$ and for any $p\in\ZZ_{\geqslant 0}$ rendering $pE'$ integral;
\item[\rm(b)] Suppose that $\pi_\ast G$ and $\pi_\ast\phi^\ast\!D_{G,c,b}$ are $\QQ$-Cartier on $X$ for some $b\in\ZZ_{>0}$. Then $\phi_\ast\scrO_M(G+cE)$ is $\frac{1}{b}D_{G,c,b}$-weakly semipositively curved over $Y_0$.
\end{itemize}
\end{cor}
\begin{proof}
By construction $E'=\sum_{i\in I^+}a_iE_i-\phi^\ast\!N_Y$ with $N_Y$ an effective $\QQ$-divisor on $Y$ supported out of $Y_0$, hence {\hyperref[prop_det-stablize]{Proposition \ref*{prop_det-stablize}}} implies that $\rank\!\phi_\ast\scrO_M(G+kE+pE')=r_1$ and that the natural injection
\[
\scrO_Y(rD_{G,c,1})\simeq\det\!\phi_\ast\scrO_M(G+cE)\hookrightarrow\det\!\phi_\ast\scrO_M(G+kE+pE')
\]
is an isomorphism over $Y_0$. By \cite[Lemma 1.20]{DPS94} this means that the natural inclusion 
\begin{equation}
\label{eq_incl-direct-image}
\phi_\ast\scrO_M(G+cE)\hookrightarrow\phi_\ast\scrO_M(G+kE+pE')
\end{equation}
is an isomorphism over the locally free locus of $\phi_\ast\scrO_M(G+kE+pE')\big|_{Y_0}$. Since $\phi$ is flat over $Y_0$, both $\phi_\ast\scrO_M(G+cE)$ and $\phi_\ast\scrO_M(G+kE+pE')$ are reflexive over $Y_0$, hence \eqref{eq_incl-direct-image} must be an isomorphism over $Y_0$. Thus (a) is proved.

As for (b), since by hypothesis $\pi_\ast G$ and $\pi_\ast\phi^\ast\!D_{G,c,b}$ are $\QQ$-Cartier on $X$, then by {\hyperref[prop_more-psef-det]{Proposition \ref*{prop_more-psef-det}}} we see that $b\pi_\ast G-\pi_\ast\phi^\ast\!D_{G,c,b}$ is a pseudoeffective ($\QQ$-Cartier) $\QQ$-divisor on $X$. In consequence there is an integer $k\in\ZZ_{>0}$ such that $\pi_\ast G+kE-\frac{1}{b}\phi^\ast\!D_{G,c,b}$ is pseudoeffective on $M$. Then by {\hyperref[prop_positivity-anti-nef]{Proposition \ref*{prop_positivity-anti-nef}}}, for $p_{k}$ sufficiently large $\phi_\ast\scrO_M(G+kE+p_kE')$ is $\frac{1}{b}D_{G,c,b}$-weakly semipositively curved. Combine this with (a) we see that $\phi_\ast\scrO_M(G+cE)$ is $\frac{1}{b}D_{G,c,b}$-weakly semipositively curved over $Y_0$, which proves (b).
\end{proof}

\begin{prop}[{\cite[Proposition 3.6]{CH19}}] 
\label{prop_det-direct-image-power}
Let everything be as in the {\hyperref[general-setting]{General Settinng \ref*{general-setting}}} with $Y$ non-uniruled. Suppose that $\psi$ is almost holomorphic and let $A$ be a sufficiently ample divisor on $X$ such that for general $y\in Y$ the natural morphism (noting that $\psi$ is almost holomorphic)
\begin{equation}
\label{eq_condition-surjectivity-sym}
\Sym^k\!\Coh^0(X_y\,,\scrO_{X_y}(A))\to\Coh^0(X_y\,,\scrO_{X_y}(kA))
\end{equation}
is surjective for every $k\in\ZZ_{>0}$. Let $c_0$ be the positive integer given by {\hyperref[prop_det-stablize]{Proposition \ref*{prop_det-stablize}}} and let $c$ be any integer $\geqslant c_0$. For every $a\in\ZZ_{>0}$ set 
\[
D_{A,c,a}:=\frac{1}{r_a}\cdot\text{the Cartier divisor on} Y \text{associated to }\det\!\phi_\ast\scrO_M(a\pi^\ast\!A+acE)
\]
where $r_a:=\rank\!\phi_\ast\scrO_M(a\pi^\ast\!A+acE)$, and suppose that $\pi_\ast\phi^\ast\!D_{A,c,1}$ is $\QQ$-Cartier on $X$ (e.g. when $X$ is $\QQ$-factorial). Then for any $m\in\ZZ_{>0}$ divisible by $r:=r_1$ such that $\pi_\ast\phi^\ast\!D_{A,c,m}$ is $\QQ$-Cartier, we have 
\[
\pi_\ast\phi^\ast\!D_{A,c,m}
\equiv m\pi_\ast\phi^\ast\!D_{A,c,1}
\]
where $\equiv$ denotes the numerical equivalence.
\end{prop}

Before proving the proposition, let us first prove the following auxiliary lemma:
\begin{lemme}
\label{lemma_positivity-Um-Vm}
 Let everything be as in the {\hyperref[general-setting]{General Settinng \ref*{general-setting}}} with $Y$ non-uniruled. Suppose that $\psi$ is almost holomorphic and let $A$ as in {\hyperref[prop_det-direct-image-power]{Proposition \ref*{prop_det-direct-image-power}}}. For every $m$ divisible by $r$ set
\begin{align*}
\scrU_{c,m} &:=\Sym^m\!\phi_\ast\scrO_M(\pi^\ast\!A+cE)\otimes\scrO_Y(-mD_{A,c,1}), \\
\scrV_{c,m} &:=\phi_\ast\scrO_M(m\pi^\ast\!A+mcE)\otimes\scrO_Y(-mD_{A,c,1}),
\end{align*} 
then $\scrU_{c,m}$ and $\scrV_{c,m}$ are both weakly semipositively curved on $Y_0$.
\end{lemme}
\begin{proof}
By hypothesis $\pi_\ast\phi^\ast\!D_{A,c,1}$ is $\QQ$-Cartier on $X$, hence by {\hyperref[cor_iso-direct-image-Y0]{Corollary \ref*{cor_iso-direct-image-Y0}(b)}} we see that $\phi_\ast\scrO_M(\pi^\ast\!A+cE)$ is $D_{A,c,1}$-weakly semipositively curved on $Y_0$, which implies that $\scrU_{c,m}$ is weakly semipositively curved on $Y_0$. 

By \eqref{eq_condition-surjectivity-sym} and by \cite[Lemma 7.11]{Deb01} we have a surjection
\[
\Sym^m\!\Coh^0(M_y,\scrO_{M_y}(\pi^\ast\!A+cE))\twoheadrightarrow\Coh^0(M_y,\scrO_{M_y}(m\pi^\ast\!A+mcE))
\]
for $y\in Y$ general, from which we see that the natural morphism $\scrU_{c,m}\to\scrV_{c,m}$ is generically surjective. Hence $\scrV_{c,m}$ is also weakly semipositively curved on $Y_0$. 
\end{proof}

Moreover, in the statement of {\hyperref[prop_det-direct-image-power]{Proposition \ref*{prop_det-direct-image-power}}} we presume the existence of a very ample divisor $A$ on $X$ satisfying the condition \eqref{eq_condition-surjectivity-sym}. We will next show that such divisor really exists. More generally we have:
\begin{lemme}
Let $V$ be a normal projective variety and let $H$ be a semiample divisor. Then up to multiplying $H$ the natural morphism 
\[
\Sym^k\!\Coh^0(V,\scrO_V(H))\to\Coh^0(V,\scrO_V(kH))
\]
is surjective for every $k\in\ZZ_{>0}$. 
\end{lemme}
\begin{proof}
The proof is quite similar to that of \cite[\S 7.1, Proposition 7.6, pp.~172-173]{Deb01}. First by  \cite[\S 7.1, Proposition 7.6(b), p.~172]{Deb01}, up to replacing $V$ by the image of $|rH|$ for some $r$ sufficiently large and replacing $H$ by the hyperplane divisor, we can assume that $H$ is a very ample and $V$ is embedded into $\PP\!E$ so that $\scrO_V(H)$ is equal to the pullback of $\scrO(1)$ where $E:=\Coh^0(V,\scrO_V(H))$. Then by the Serre vanishing, for $s$ sufficiently large, we have 
\[
\Coh^1(\PP E,\scrI_X(ks))=0.
\]
for any $k\in\ZZ_{>0}, $where $\scrI_X$ is the ideal of $X$ in $\PP\!E$. Hence the surjectivity of 
\[
\Coh^0(\PP\!E, \scrO(ks))\to\Coh^0(V,\scrO_V(ksH)).
\]
But
\[
\Coh^0(\PP\!E,\scrO(ks))\simeq\Sym^{ks}\!\Coh^0(\PP\!E, \scrO(1))\simeq\Sym^{ks}\!\Coh^0(V,\scrO_V(H)),
\]
thus
\[
\Sym^{ks}\!\Coh^0(V,\scrO_V(H))\twoheadrightarrow\Coh^0(V,\scrO_V).
\]
Now this map factorizes through 
\[
\Sym^{ks}\!\Coh^0(V,\scrO_V(H))\to\Sym^k\!\Coh^0(V,\scrO_V(sH)),
\]
hence $sH$ satisfies the condition, and the lemma is proved. 
\end{proof}

Now let us turn to the proof of {\hyperref[prop_det-direct-image-power]{Proposition \ref*{prop_det-direct-image-power}}}:

\begin{proof}[Proof of {\hyperref[prop_det-direct-image-power]{Proposition \ref*{prop_det-direct-image-power}}}]
By {\hyperref[prop_det-stablize]{Proposition \ref*{prop_det-stablize}}}, as soon as $c\geqslant c_0$, for any $a\in\ZZ_{>0}$ the divisor $\pi_\ast\phi^\ast\!D_{A,c,a}$ over $X$ is independent of $c$. Hence it suffices to prove the proposition for a particular choice of $c\geqslant c_0$. By Kleiman's criterion for numerical triviality \cite[Lemma 4.1]{GKP16a}, it suffices to show that for any $(\dim\!X-1)$-tuple of ample line bundles $L_1\,,\cdots,L_{\dim\!X-1}$ on $X$ the intersection number 
\[
L_1\cdot\cdots\cdot L_{\dim\!X-1}\cdot(\pi_\ast\phi^\ast\!D_{A,c,m}-m\pi_\ast\phi^\ast\!D_{A,c,1})=0.
\]
To this end, let $H_i$ be a general member of $\left|L_i\ptensor[k]\right|$ for $k$ sufficiently large and set $C=H_1\cap\cdots\cap H_{\dim\!X-1}$. By the projection formula it suffices to show that 
\[
(\phi^\ast\!D_{A,c,m}-m\phi^\ast\!D_{A,c,1})\cdot(\pi\inv\!C)=0.
\]
Since $\pi(\Exceptional(\pi))$ is of codimension $2$ in $X$, $C$ is disjoint from $\pi(\Exceptional(\pi))$, then $\pi\inv\!C$ is disjoint from $E$ and thus $C_Y:=\phi(\pi\inv C)$ is contained in $Y_0$. 
Let $\bar C_Y$ be the normalization of $C_Y$ and let $\bar i_{C_Y}:\bar C_Y\to Y$ be the natural morphism. Again by the projection formula, we are reduced to show that \[
\bar i_{C_Y}^\ast(D_{A,c,m}-mD_{A,c,1})=0.
\]
As in {\hyperref[lemma_positivity-Um-Vm]{Lemma \ref*{lemma_positivity-Um-Vm}}}, we set for any $m$ divisible by $r$
\begin{align*}
\scrU_{c,m} &:=\Sym^m\!\phi_\ast\scrO_M(\pi^\ast\!A+cE)\otimes\scrO_Y(-mD_{A,c,1}), \\
\scrV_{c,m} &:=\phi_\ast\scrO_M(m\pi^\ast\!A+mcE)\otimes\scrO_Y(-mD_{A,c,1}),
\end{align*}
Since $\scrU_{c,m}$ and $\scrV_{c,m}$ are torsion free, we can assume that $C_Y$ is contained in the locally free locus of them. 

By {\hyperref[lemma_positivity-Um-Vm]{Lemma \ref*{lemma_positivity-Um-Vm}}} $\scrV_{c,m}$ is weakly semipositively curved, since $C_Y$ is a general complete intersection curve, 
$C_Y$ is not contained in the singular locus of the $-\epsilon\omega$-semipositive metric of $\scrV_{c,m}$, $\scrV_{c,m}|_{Y_0}$ is semipositively curved on $C_Y$, in particular $\bar i_{C_Y}^\ast\det\scrV_{c,m}\geqslant 0$. But $\det\scrV_{c,m}\simeq\scrO_Y(r_mD_{A,c,m}-mr_mD_{A,c,1})$, hence we have
\[
\bar i_{C_Y}^\ast(D_{A,c,m}-mD_{A,c,1})\geqslant 0.
\]

On the other hand, $\pi_\ast\phi^\ast\!D_{A,c,m}$ is $\QQ$-Cartier on $X$, then by {\hyperref[cor_iso-direct-image-Y0]{Corollary \ref*{cor_iso-direct-image-Y0}(b)}} we see that $\phi_\ast\scrO_M(\pi^\ast\!A+cE)$ is $\frac{1}{m}D_{A,c,m}$-weakly semipositively curved over $Y_0$, in consequence 
\[
\bar i_{C_Y}^\ast\phi_\ast\scrO_M(\pi^\ast\!A+cE)
\]
is $\frac{1}{m}\bar i_{C_Y}^\ast\!D_{A,c,m}$weakly semipositively curved. Hence
\[
\bar i_{C_Y}^\ast\det\!\phi_\ast\scrO_M(\pi^\ast\!A+cE)-\frac{r}{m}
\bar i_{C_Y}^\ast\!D_{A,c,m}\geqslant 0,
\]
implying that
\[
\bar i_{C_Y}^\ast(D_{A,c,m}-mD_{A,c,1})\leqslant 0.
\]
\end{proof}

\section{Albanese map of \texorpdfstring{$X$}{text}}
\label{sec_Albanese}

In this section, we take $\psi$ in the \hyperref[general-setting]{General Setting \ref*{general-setting}} to be the Albanese map $\alb_X$ of $X$. Recall that $X$ is of semi-Fano type, i.e. $X$ admits an effective $\QQ$-divisor $\Delta$ such that $(X,\Delta)$ is klt and that the twisted anticanonical divisor $-(K_X+\Delta)$ is nef. In this situation, we can take $M$ to be any smooth model of $X$, $\phi=\alb_M$ and $Y=\Image(\alb_M)\subseteq\Alb_M$\,, by \cite[Proposition 9.12, pp.~107-108]{Ueno75} $\psi$ is independent of the choice of the smooth model $M$.

First recall the basic properties of the Albanese map (c.f. \cite[\S 9, pp.~94-115]{Ueno75}):
\begin{prop}
\label{prop_app-Alb}
Let $V$ a compact K\"ahler manifold and let $\alb_V:V\to\Alb_V$ be its Albanese map. Then we have:
\begin{itemize}
\item[\rm (a)] $\alb_V$ satisfies the following universal property: every morphism $V\to T$ with $T$ a complex torus factorizes via $\alb_V: V\to \Alb_V$; in addition $\Alb_V\to T$ is a morphism of analytic Lie groups up to a translation. C.f. {\rm \cite[Defintion 9.6, pp.~101-102]{Ueno75}}.
\item[\rm (b)] $W:=\Image(\alb_V)$ generates $\Alb_V$, i.e. there is an integer $k>0$ such that the morphism 
\begin{align*}
\underbrace{W\times \cdots \times W}_{k\text{ times}}  & \longrightarrow \Alb_V\,, \\
(w_1,\cdots,w_k) &\longmapsto w_1+\cdots +w_k\,,
\end{align*}
is surjective. C.f. {\rm \cite[Lemma 9.14, pp.~108-110]{Ueno75}}.
\end{itemize}
\end{prop}

More generally, for $V$ a compact complex variety in the Fujiki class $\mathscr{C}$ (not necessarily smooth), the Albanese map $\alb_V$ of $V$ is defined to be the meromorphic map induced by the Albanese map of a smooth model of $V$ (this definition is independent of the choice of the smooth model by \cite[Proposition 9.12, pp.~107-108]{Ueno75}). In this case, $\alb_V$ has the universal property that every meromorphic map from $V$ to a complex torus factorizes via $\alb_V$ (analogous to {\hyperref[prop_app-Alb]{Proposition \ref*{prop_app-Alb}(a)}}), c.f. \cite[Theorem-Definition 2.1]{YWang16}.

\subsection{Everywhere-definedness, surjectivity and connectedness of fibres of \texorpdfstring{$\alb_X$}{text}}
\label{ss_Albanese_surjective}

In this subsection, we briefly recall how one proves that $\psi=\alb_X$ is everywhere defined, surjective and with connected fibres:
\begin{itemize}
\item Since $(X,\Delta)$ is a klt pair, in particular $X$ has rational singularities (c.f. \cite[Theorem 5.22, pp.~161-162]{KM98}) hence by \cite[Lemma 8.1]{Kaw85}, $\psi$ is a(n) (everywhere defined) morphism $X\to\Alb_X$. 
\item By {\hyperref[prop_bir-geometry-psi]{Proposition \ref*{prop_bir-geometry-psi}(b)}}, the Kodaira dimension of $\Image(\psi)$ is equal to $0$, then \cite[Theorem 10.3]{Ueno75} implies that $\Image(\psi)$ is a translate of a subtorus of $Y$; in virtue of {\hyperref[prop_app-Alb]{Proposition \ref*{prop_app-Alb}(b)}}, a fortiori $\Image(\psi)=Y$, i.e. $\psi$ is surjective. 
\item To see that $\psi$ has connected fibres, let us take $\pi: Y'\to Y$ a Stein factorization of $\psi$ with $Y'$ a normal projective variety, then by {\hyperref[prop_bir-geometry-psi]{Proposition \ref*{prop_bir-geometry-psi}(b)}} we have $\kappa(Y')=0$, which implies, in virtue of \cite[Main Theorem]{KV80}, that $\pi$ is a finite étale cover. Then the theorem of Serre-Lang \cite[\S 18, pp.~167-168]{Mum70} implies that $Y'$ is an abelian variety with $\pi$ an isogeny. By {\hyperref[prop_app-Alb]{Proposition \ref*{prop_app-Alb}(a)}} $\pi$ is a fortiori an isomorphism. 
\end{itemize}

\subsection{Flatness of \texorpdfstring{$\alb_X$}{text}}
\label{ss_Albanese_flat}

In order to apply {\hyperref[prop_num-flat--local-const]{Proposition \ref*{prop_num-flat--local-const}}} to $\psi$ one needs to prove first that it is flat. In this subsection we will settle this by following the argument of \cite{LTZZ10}. Recall that $X$ is a 
normal projective variety of semi-Fano type, i.e. $X$ admits an effective $\QQ$-divisor $\Delta$ such that $(X,\Delta)$ is a klt pair and that the twisted anticanonical divisor $-(K_X+\Delta)$ is nef. The flatness of $\psi$ can be deduced from the following lemma. Let us remark that under the additional assumption that $V$ is smooth and $D=0$, a stronger result is obtained in \cite[Proposition 4.1]{EIM20} (they also prove the semistability of the fibre space).

\begin{lemme}
\label{lemme_anticanonique-nef-plat}
Let $f: V\to W$ a surjective morphism with connected fibres with $V$ a projective Cohen-Macaulay variety and $W$ a smooth projective variety. Suppose that there is an effective $\QQ$-divisor $D$ on $V$ such that $(V,D)$ is a log canonical pair and that the twisted relative anticanonical divisor  $-(K_{V/W}+D)$ is nef on $V$. Then $f$ is flat.  
\end{lemme}

\begin{proof}
By the miracle flatness, it suffices to show that $f$ is equi-dimensional. Suppose by contradiction that $f$ is not so, then there is a (closed) point $w_0\in W$ such that $\dim V_{w_0}>\dim F$ where $F$ denotes the general fibre of $f$. Now take $S$ to be the complete intersection of $\dim V-\dim V_{w_0}+1$ general very ample divisors passing through $w_0$. Then by Bertini $S$ is a smooth projective variety containing $w_0$ of dimension $\dim S= \dim V_{w_0}-\dim F+1$. Set $T=S\underset{W}{\times}V$ with $g:T\to S$ the induced morphism, then $\dim T=\dim V_{w_0}+1$. 

Let us remark that in \cite{LTZZ10} it is claimed that $T$ is smooth in codimension 1; but this cannot be true in general since a priori $V_{w_0}$ can be a non-reduced fibre of $f$ which is a codimension $1$ subvariety contained in $T$.  We will present below a proof avoiding the use of this claim.

By construction $T$ is a complete intersection in $X$, thus $T$ is Cohen-Macaulay by \cite[Theorem 17.3(i\!i), p.~134]{Mat89}. By adjunction formula \cite[\S II.5, Proposition 5.26, pp.~139-140]{CDGPR94} one finds that $\omega_T\simeq\omega_V|_T$. Now take a flattening morphism $p_S:S'\to S$ of $g$ (c.f. \cite[Flatenning Theorem]{Hir75}) and take $T'$ a desginularization of the principal component of $T\underset{S}{\times}S$ with $g':T'\to S'$ and $p_T: T'\to T$ the induced morphisms. Then every $g'$-exceptional divisor must be $p_T$-exceptional.  

\begin{center}
\begin{tikzpicture}[scale=1.6]
\node (A) at (0,0) {$S$};
\node (B) at (0,1.5) {$T$};
\node (BB) at (0,2.5) {$\bar T$};
\node (A1) at (-1.5,0) {$S'$};
\node (B1) at (-1.5,1.5) {$T'$};
\node (C) at (1,0) {$W$};
\node (D) at (1,1.5) {$V$};
\node (E) at (0.5,0.75) {$\square$};
\path[->,font=\scriptsize,>=angle 90]
(B) edge node[left]{$g$} (A)
(D) edge node[right]{$f$} (C)
(B1) edge node[left]{$g'$} (A1)
(BB) edge node[right]{$\nu$} (B)
(B1) edge node[above left]{$p_{\bar T}$} (BB)
(A1) edge node[below]{$p_S$} (A)
(B1) edge node[above]{$p_T$} (B);
\path[right hook->, >=angle 90]
(A) edge (C)
(B) edge (D);
\end{tikzpicture}
\end{center}

Take the normalization $\nu:\bar T\to T$ of $T$, then $p_T$ factors through $\nu$, and denote by $p_{\bar T}$ the induced morphism $T'\to\bar T$. By \cite[(2.3) Proposition and Step 3 of (3.6)]{Reid94} (noting that $T$ is Cohen-Macaulay thus $S_2$) we have
\[
\nu^{[\ast]} \omega_T\simeq \scrO_{\bar T}(K_{\bar T}+\Cond_{\bar T})
\]
where $\nu^{[\ast]}$ signifies the reflexive hull of the pullback and $\Cond_{\bar T}$ is the effective Weil divisor defined by the conductor ideal on $\bar T$. Since $K_V+D$ is $\QQ$-Cartier, we can write 
\begin{align*}
K_{T'}+(p_T)\inv_\ast\!D_T &\sim_{\QQ} p_T^\ast(K_V|_T+D_T)+E_T-G\,, \\
K_{S'} &\sim p_S^\ast K_S+E_S\,,
\end{align*}
where $D_T:=D|_T$, $E_T$ is a (non necessarily effective) $p_T$-exceptional (thus $p_{\bar T}$-exceptional) divisor, $G$ is an effective divisor consisting of the non-exceptional components of the pullback of $\Cond_{\bar T}$, and $E_S$ is an effective $p_S$-exceptional divisor (noting that $S$ is smooth). Hence 
\[
K_{T'/S'}+(p_T)\inv_\ast\!D_T\sim_{\QQ} \left.p_T^\ast(K_{V/W}+D)\right|_T+E_T-G-(g')^{\ast}E_S.
\]
Moreover, let $F$ be the general fibre of $g$, then by construction it is also the general fibre of $f$, by \cite[Lemma 5.7, pp.~158-159]{KM98} $(F,D_F)$ is a lc pair where $D_F:=D|_F$, hence the horizontal part of $E_T-G$ has coefficients $\geqslant -1$. Write
\[
(E_T-G)^{\horizontal}:=\sum_{j\in J} b_j B_j
\]
with the $B_j$'s being prime divisors, and set
\[
\Delta_0:=\sum_{j\in J, b_j<0} (-b_j) B_j\,,
\]
then every coefficient in $\Delta_0$ is $\leqslant 1$. By the construction of $\Delta_0$, we can rewrite $E_T+\Delta_0-G$ as $E'_T-G'$ with $E'_T$ being $p_T$-exceptional and $G'$ being effective whose components come from the conductor divisor of the normalization of $T$. Clearly the support of $E'_T$ (resp. $G'$) is contained in that of $E_T$ (resp. $G$). 

Since $\dim T=\dim V_{w_0}+1$, $p_T\inv(V_{w_0})$ is a  non-$p_T$-exceptional divisor in $T'$, hence is not $g'$-exceptional, consequently $g'(p_T\inv(V_{w_0}))$ contains a codimension $1$ component, which we denote by $E$. Then $p_S(E)=\{w_0\}$ hence $E\subseteq\Supp(E_S)$ (by assumption $g$ is not flat, hence $E_S\neq 0$ and $\Supp(E_S)\neq\varnothing$).    

Take an ample divisor $A$ on $T'$, since $-(K_{V/W}+D)$ is nef then for any $\epsilon\in\QQ_{>0}$ the $\QQ$-divisor $-\left.p_T^\ast (K_{V/W}+D)\right|_T+\epsilon A$ is ample. Choose an ample $\QQ$-divisor $H_\epsilon$ on $S'$ such that $\left.-p_T^\ast (K_{V/W}+D)\right|_T+\epsilon A-(g')^\ast H_\epsilon$ is still ample. Let 
\[
\Delta_\epsilon:=(p_T)\inv_\ast\!D_T+\Delta_0+\frac{1}{k}\cdot\text{general member of
}\left|k\left(-\left.p_T^\ast(K_{V/W}+D)\right|_T+\epsilon A-(g')^\ast H_\epsilon\right)\right|\,,
\]
where $k$ is a positive integer sufficiently large and divisible (so that $\epsilon\cdot k \in\ZZ$ and that $kH_\epsilon$ is an integral divisor). Then the coefficients in $\Delta_\epsilon$ are $\leqslant 1$, thus the pair $(T',\Delta_\epsilon)$ is lc. By \cite[Theorem 1.1]{Fuj17} the direct image sheaf $g'_\ast\scrO_{T'}\left(k(K_{T'/S'}+\Delta_\epsilon)\right)$ is weakly positive; $(E_T-G)^{\vertical}-(g')^\ast E_S$ being $g'$-vertical and $(E_T-G)^{\horizontal}+\Delta_0$ being effective, $K_{T'/S'}+\Delta_\epsilon$ is big on the general fibre of $g'$, in particular we have 
\[
g'_\ast\scrO_{T'}\left(k(K_{T'/S'}+\Delta_\epsilon)\right)\neq 0.
\]
Hence there is $p\in\ZZ_{>0}$ such that 
\[
\RefSym^p\!g'_\ast\scrO_{T'}\left(k(K_{T'/S'}+\Delta_\epsilon)\right)\otimes\scrO_{S'}(kpH_\epsilon)\simeq\RefSym^p\!g'_\ast\scrO_{T'}\left(k(\epsilon A+E_T+\Delta_0-G-(g')^\ast E_S)\right)
\]
is generically globally generated, that is, there is a generically surjective morphism
\[
\scrO_{S'}^{\oplus d}\to\RefSym^p\!g'_\ast\scrO_{T'}\left(k(\epsilon A+E_T+\Delta_0-G-(g')^\ast E_S)\right),
\]
where
\[
d:=\dim\Coh^0(S',\RefSym^p\!g'_\ast\scrO_{T'}\left(k(\epsilon A+E_T+\Delta_0-G-(g')^\ast E_S)\right))\in\ZZ_{>0}.
\]
Pull it back to $T'$ and combined with the natural (non-trivial) morphism
\[
(g')^\ast\RefSym^p\!g'_\ast\scrO_{T'}\left(k(\epsilon A+E_T+\Delta_0-G-(g')^\ast E_S)\right)\to\scrO_{T'}\left(kp(\epsilon A+E_T+\Delta_0-G-(g')^\ast E_S)\right)
\]
one finds that $\epsilon A+E_T+\Delta_0-G-(g')^\ast E_S$ is $\QQ$-linearly equivalent to an effective $\QQ$-divisor. Letting $\epsilon\to 0$, we obtain that $E_T+\Delta_0-G-(g')^\ast E_S=E'_T-G'-(g')^\ast E_S$ is pseudoeffective. If $(V,D)$ is klt, the pseudoeffectivity result can also be obtained by the semipositivity of the curvature current of the relative $m$-Bergman kernel metric on the twisted relative canonical bundle (c.f. \cite[Theorem 1.2]{Cao17}).

Finally, let $L$ be a very ample line bundle on $\bar T$, and let $H_1\,,\cdots,H_{\dim\! V_{w_0}}$ be general members of the linear series $\left|p_{\bar T}^\ast L\right|$. Set 
\[
C:=H_1\cap\cdots\cap H_{\dim\!V_{w_0}}\,,
\]
then $C$ is a movable curve on $T'$, hence $(E'_T-G'-(g')^\ast E_S)\cdot C\geqslant 0$ by \cite[0.2 Theorem]{BDPP13} (c.f. also \cite[vol.I\!I, Theorem 11.4.19,  p.~308]{Laz04}). The divisor $E'_T$ being $p_{\bar T}$-exceptional, we have $E'_T\cdot C=0$ by the projection formula. Thus we get
\[
(g')^\ast E\cdot C\leqslant (g')^\ast E_S\cdot C\leqslant -G'\cdot C\leqslant 0
\]
where the last inequality results from the effectivity of $G'$. On the other hand, $(g')^\ast E$ is not $p_{\bar T}$-exceptional, hence $(p_{\bar T})_\ast(g')^\ast E$ is an effective (Weil) divisor on $\bar T$ (e.g. it contains $\nu\inv(V_{w_0})$), thus again by the projection formula one gets
\[
(g')^\ast E\cdot C=(p_{\bar T})_\ast(g')^\ast E\cdot L^{\dim\!V_{w_0}}>0,
\]
which is a contradiction. Hence $f$ is flat.
\end{proof}

\subsection{Reduction to \texorpdfstring{$\QQ$}{text}-factorial case}
\label{ss_Albanese_Q-fact}
In this subsection we prove that in order to prove {\hyperref[mainthm_Albanese]{Theorem \ref*{mainthm_Albanese}}}, we can assume that $X$ is $\QQ$-factorial. The key ingredient in the proof of this reduction is the following lemma:

\begin{lemme}
\label{lemma_relative-Druel-decomp}
Let $p: S\to B$ and $f:S'\to S$ be projective surjective morphisms between normal complex varieties such that $f_\ast\scrO_{S'}\simeq\scrO_S$. Suppose that $p\circ f$ induces a decomposition of $S'$ into a product $B\times Y'$ with $q(Y')=0$. Then there is a normal projective variety $Y$ along with a projective morphism $g:Y'\to Y$ such that $p$ induces a decomposition of $S$ into a product $B\times Y$ and that under the decompositions $S'\simeq B\times Y'$ and $S\simeq B\times Y$ we have $f=\id_B\!\times\,g$.  
\end{lemme}
\begin{proof}
This is the relative version of \cite[Lemma 4.6]{Druel18a}. In fact, when $B$ is a projective variety, it is just a simple corollary of \cite[Lemma 4.6]{Druel18a}; in order to apply to our situation 
we need to treat the case that $B$ is a (non-necessarily compact) complex manifold. The proof can be divided into four parts. 

\paragraph{(A) Construction of $g$.}
Since $p:S\to B$ is a projective morphism, there is a $p$-very ample line bundle $L$ on $S$; since $q(Y')=0$ by (the analytic version of) \cite[\S I\!I\!I.12, Exercise 12.6, p.~292]{Har77} there are line bundles $L_B\in\Pic(B)$ and $L_{Y'}\in\Pic(Y')$ such that 
\[
f^\ast L\simeq\pr_1^\ast\!L_B\otimes\pr_2^\ast\!L_{Y'}
\]
with $\pr_1:=p\circ f$ and $\pr_2$ being natural projections of $S'\simeq B\times Y'$. Up to replacing $L$ by $L\otimes p^\ast\!L_B\inv$ we can assume that $f^\ast\!L\simeq\pr_2^\ast\!L_{Y'}$ for some line bundle $L_{Y'}$. Since $f^\ast\!L$ is $(p\circ f)$-relatively generated, hence $L_{Y'}$ is globally generated over $Y'$. Then by \cite[\S 2.1.B, Theorem 2.1.27, pp.~129-130, Vol.I]{Laz04} for $m$ sufficiently large, $L_{Y'}\ptensor[m]$ defines a morphism $g: Y'\to Y$ with connected fibres. In addition, by construction there is a very ample divisor $H$ on $Y$ such that $g^\ast\scrO_Y(H)\simeq L_{Y'}\ptensor[m]$. 

\paragraph{(B) Contraction of the fibres of $\id_B\!\times\,g$ by $f$.}
Set (by identifying $S'$ with $B\times Y'$) $g_B=\id_B\!\times\,g:S'\to B\times Y$. Then we have the following commutative diagram:
\begin{center}
\begin{tikzpicture}[scale=2.5]
\node (A) at (0,0) {$B$.};
\node (B) at (0,1) {$S$};
\node (C) at (0,2) {$S'$};
\node (B1) at (-1,1) {$B\times Y$};
\node (C1) at (-1,2) {$B\times Y'$};
\path[->,font=\scriptsize,>=angle 90]
(B) edge node[right]{$p$} (A)
(C) edge node[right]{$f$} (B)
(C1) edge node[left]{$g_B$} (B1)
(B1) edge node[above]{$\exists\,\bar f$} (B)
(C) edge node[above]{$\simeq$} (C1)
(B1) edge node[below left]{$\pr_1$} (A);
\end{tikzpicture}
\end{center}
In this part we will prove that every fibre of $g_B$ is contracted by $f$. Let $g_B\inv(b,z)$ be a positive dimensional fibre of $g_B$ (with $(b,z)\in B\times Y$), since $g_B\inv(b,z)\simeq g\inv(z)=:Y'_z$, it can be regarded as a subvariety of $Y'$ contracted by $g$. Let $C$ any curve contained in $g_B\inv(b,z)$, then $C\subseteq(p\circ f)\inv(b)=:S'_b$ and since $C$ is contracted by $g$ we have 
\[
\left.(f^\ast\!L)\right|_{S'_b}\cdot C=L_{Y'}\cdot C=\frac{1}{m}g^\ast\!H\cdot C=0,
\]
which means that $C$ is contracted by $f$. Hence every fibre of $g_B$ is contracted by $f$.

\paragraph{(C) Factorization of $f$ through $g_B$.}
In this step we prove that $f$ factorizes through $g_B$. This can be deduced from the following rigidity lemma, which is nothing but an analytic version of \cite[Lemma 1.15, pp.~12-13]{Deb01}:
\begin{lemme}
\label{lemma_rigidity}
Let $f_1:S'\to S_1$ and $f_2:S'\to S_2$ be proper surjective morphisms between normal complex varieties such that $f_{1\ast}\scrO_{S'}\simeq\scrO_{S_1}$. If $f_2$ contracts every fibre of $f_1$, then $f_2$ factorizes through $f_1$. 
\end{lemme}
\begin{proof}[Proof of {\hyperref[lemma_rigidity]{Lemma \ref*{lemma_rigidity}}}]
The proof is the same as the one of \cite[Lemma 1.15, pp.~12-13]{Deb01}. For the convenience of the readers we give the details below to illustrate that the argument in \cite[Proof of Lemma 1.15, pp.~12-13]{Deb01} fits into the analytic case. Consider the morphism 
\[
\phi:=(f_1,f_2):S'\to S_1\times S_2.
\]
Let $\Gamma$ be the image of $\phi$ and let $p_1:\Gamma\to S_1$ and $p_2:\Gamma\to S_2$ be the natural projections restricted to $\Gamma$, then $p_i\circ\phi=f_i$ for $i=1,2$. For any $s\in S_1$, $f_2$ contracts $f_1\inv(s)=(\phi\circ p_1)\inv(s)$, hence 
\[
p_1\inv(s)=\phi(\phi\inv(p_1\inv(s)))=\phi(f_1\inv(s))
\]
is a singleton, hence the proper surjective morphism $p_1:\Gamma\to S_1$ is a finite morphism. But $f_1$ has connected fibres, then so is $p_1$, thus by Stein factorization \cite[\S 1, Theorem 1.9, pp.~8-9]{Ueno75} $p_1$ is an isomorphism. Then we have $\phi=p_1\inv\circ f_1$ and 
\[
f_2=p_2\circ\phi=p_2\circ p_1\inv\circ f_1.
\]
\end{proof}

\paragraph{(D) Conclusion.}
By {\bf (C°)} there is a morphism $\bar f: B\times Y\to S$ such that $f=\bar f\circ g_B$. Hence 
\[
g_B^\ast\left(\bar f^\ast\!L\ptensor[m]\right)=f^\ast\!L\ptensor[m]\simeq \pr_2^\ast\!L_{Y'}\ptensor[m]\simeq\pr_2^\ast g^\ast\scrO_Y(H)=g_B^\ast(\pr_2^\ast\!\scrO_Y(H)).
\]
But $g$ has connected fibres, hence so is $g_B$, in consequence $g_B^\ast$ is an injective morphism between Picard groups, thus $\bar f^\ast\!L\ptensor[m]\simeq\pr_2^\ast\!\scrO_Y(H)$. Since $H$ is very ample and $L$ is $p$-relatively very ample, by looking at every fibre of $p$, we see that $\bar f$ is a finite morphism; but 
\[
\bar f_\ast\scrO_{B\times Y}\simeq\bar f_\ast g_{B\ast}\scrO_{S'}\simeq f_\ast\scrO_{S'}\simeq\scrO_S\,,
\]
hence $\bar f$ is an isomorphism.       
\end{proof}

Now let us return to the proof of the reduction of {\hyperref[mainthm_Albanese]{Theorem \ref*{mainthm_Albanese}}}, whose idea comes from the author's personal communications with Stéphane Druel (of course, any mistake is the author's):

\begin{proof}[{Reduction to the $\QQ$-factorial case}]
Suppose that {\hyperref[mainthm_Albanese]{Theorem \ref*{mainthm_Albanese}}} holds for $X$ $\QQ$-factorial, let us prove it for general $X$. Let $g: X^{\qf}\to X$ be a $\QQ$-factorialization of $X$, whose existence is proved in \cite[Corollary 1.37, pp.~29-30]{Kollar13}. By construction, $g$ is a small birational morphism, hence
\[
K_{X^{\qf}}+g\inv_\ast\!\Delta\sim_{\QQ} g^\ast(K_X+\Delta)
\]
then $(X^{\qf},g_\ast\inv\!\Delta)$ is a klt pair with the twisted anticanonical divisor $-(K_{X^{\qf}}+g_\ast\inv\!\Delta)$ nef. In particular $\alb_{X^{\qf}}$ is an everywhere defined morphism; and since the Albanese map is independent of the choice of the birational model, we have $\Alb_{X^{\qf}}=\Alb_X$ and
\[
\alb_{X^{\qf}}=g\circ\alb_X.
\]
Now by our assumption $\alb_{X^{\qf}}$ is a locally constant fibration whose fibre $F^{\qf}$ has vanishing irregularity, then by passing to the universal cover of $\Alb_X$ the pullback of $X^{\qf}$ splits into a product $\CC^q\times F^{\qf}$ where $q=\dimcoh^1(X,\scrO_X)$. By {\hyperref[lemma_relative-Druel-decomp]{Lemma \ref*{lemma_relative-Druel-decomp}}} we see that the pullback of $X$ splits also in to a product $\CC^q\times F$ where $F$ is the general fibre of $\alb_X$ and there is a birational morphism $g_F:F^{\qf}\to F$ that makes the following diagram commutative:
\begin{center}
\begin{tikzpicture}[scale=2.5]
\node (A) at (0,0) {$\Alb_X$.};
\node (A1) at (-1,0) {$\CC^q$};
\node (B) at (0,1) {$X$};
\node (B1) at (-1,1) {$\CC^q\times F$};
\node (C) at (0,2) {$X^{\qf}$};
\node (C1) at (-1,2) {$\CC^q\times F^{\qf}$};
\node (S1) at (-0.5,0.5) {$\square$};
\node (S2) at (-0.5,1.5) {$\square$};
\path[->,font=\scriptsize,>=angle 90]
(B) edge node[right]{$\alb_X$} (A)
(C) edge node[right]{$g$} (B)
(B1) edge node[left]{$\pr_1$} (A1)
(C1) edge node[left]{$\id_{\CC^q}\!\times g_F$} (B1)
(A1) edge (A)
(B1) edge (B)
(C1) edge (C);
\end{tikzpicture}
\end{center}
Since every square in the diagram is Cartesian, the morphism $\CC^q\times F\to X$ is also a Galois cover with Galois group $\pi_1(\Alb_X)$ and the $\pi_1(\Alb_X)$-action on $\CC^q\times F$ is compatible with $\id_{\CC^q}\!\times g_F$. A simple computation shows that the action of $\pi_1(\Alb_X)$ on $\{t\}\times F$ is independent of $t$, hence the $\pi_1(\Alb_X)$-action on $\CC^q\times F$ is diagonal, therefore $\alb_X$ is also a locally constant fibration.
\end{proof}

In the sequel of the section, we always assume that $X$ is $\QQ$-factorial (so that $X$ itself has klt singularities).

\subsection{Local constancy of \texorpdfstring{$\alb_X$}{text} as fibration}
\label{ss_Albanese_local-constant}

In this subsection let us prove that $\alb_X$ is a locally constant fibration (c.f. {\hyperref[defn_local-const-fibration]{Definition \ref*{defn_local-const-fibration}}}). In virtue of {\hyperref[prop_num-flat--local-const]{Proposition \ref*{prop_num-flat--local-const}}}, it suffices to find a $\psi$-very ample divisor $A$ on $X$ such that $\psi_\ast\scrO_X(mA)$ is numerically flat for every $m$ where $\psi=\alb_X$. 

Recall that we set $\psi=\alb_X$, $\pi: M\to X$ a smooth model of $X$ and $\phi=\alb_M$, $Y=\Alb_M=\Alb_X$, as mentioned at the beginning of {\hyperref[sec_Albanese]{\S \ref*{sec_Albanese}}}. By {\hyperref[ss_Albanese_Q-fact]{\S \ref*{ss_Albanese_Q-fact}}} we can assume that $X$ is $\QQ$-factorial, then by {\hyperref[ss_Albanese_flat]{\S \ref*{ss_Albanese_flat}}} $\psi$ is flat, we thus have $Y_0=Y$. Let $A$ be a very ample divisor on $X$. Up to multiplying $A$ we can assume that for general $y\in Y$ the natural morphism
\[
\Sym^k\!\Coh^0(X_y,\scrO_{X_y}(A))\to\Coh^0(X_y,\scrO_{X_y}(kA))
\]
is surjective for every $k$. As $\pi$ is birational, $\pi^\ast\!A$ is big and for every $k\in\ZZ_{>0}$ and for general $y\in Y$ we have a surjection:
\[
\Sym^k\!\Coh^0(M_y,\scrO_{M_y}(\pi^\ast\!A))\twoheadrightarrow\Coh^0(M_y,\scrO_{M_y}(k\pi^\ast\!A));
\]
in addition, for any $m,c\in\ZZ_{>0}$ we have
\[
\pi_\ast\scrO_M(m\pi^\ast\!A+mcE)\simeq \scrO_X(mA),
\]
hence
\[
\phi_\ast\scrO_M(m\pi^\ast\!A+mcE)\simeq\psi_\ast\scrO_X(mA).
\]
For every integer $a$ we set
\[
D_{A,b}:=\frac{1}{r_b}\cdot\text{ the Cartier divisor on }W\text{ associated to the line bundle }\det\!\psi_\ast\scrO_X(bA)
\]
where $r_b:=\rank\psi_\ast\scrO_X(bA)$. Then by \cite[\S 2.3, Proposition 2.3(c), pp.~33-34]{Ful84} we have
\[
\pi_\ast\phi^\ast\!D_{A,m}\sim\psi^\ast\!D_{A,m}
\]

Since $X$ is $\QQ$-factorial, by {\hyperref[prop_more-psef-det]{Proposition \ref*{prop_more-psef-det}}}, 
the ($\QQ$-Cartier) $\QQ$-divisor
\[
A-\psi^\ast\!D_{A,1}=\pi_\ast(\pi^\ast-\phi^\ast\!D_{A,1})
\]
is pseudoeffective. By {\hyperref[prop_det-direct-image-power]{Proposition \ref*{prop_det-direct-image-power}}}, 
up to multiplying $A$ by a integer divisible by $r$, we can assume that $\psi^\ast\!D_{A,1}$ is an integral Cartier divisor (noting that $\Pic^0(X)$ is an Abelian variety, thus divisible). In consequence, by replacing $A$ by $A-\psi^\ast\!D_{A,1}$, we get an integral Cartier divisor $A$ on $X$ such that:
\begin{itemize}
\item $A$ is pseudoeffective on $X$;
\item $A$ is $\psi$-very ample;
\item for general $y\in Y$ and for any $k\in\ZZ_{>0}$ the natural morphism 
\[
\Sym^k\!\Coh^0(X_y,\scrO_{X_y}(A))\to\Coh^0(X_y,\scrO_{X_y}(kA))
\]
is surjective;
\item $D_{A,1}$ is trivial.
\end{itemize}
In the sequel we will show that $\psi_\ast\scrO_X(mA)$ is numerically flat for every $m\in\ZZ_{>0}$. 

First, since $\pi$ is birational, $\pi^\ast\!A$ is $\phi$-big and by \cite[Lemma 7.11]{Deb01} the natural morphism 
\[
\Sym^k\!\Coh^0(M_y\,,\scrO_{M_y}(\pi^\ast\!A))\to\Coh^0(M_y,\scrO_{M_y}(k\pi^\ast\!A))
\]
is surjective for all $k\in\ZZ_{>0}$. 
Hence by {\hyperref[prop_positivity-anti-nef]{Proposition \ref*{prop_positivity-anti-nef}}} $\psi_\ast\scrO_X(mA)\simeq\phi_\ast\scrO_M(m\pi^\ast\!A+pE')$ is weakly semipositively curved for every $m\in\ZZ_{>0}$. Moreover, by {\hyperref[prop_det-direct-image-power]{Proposition \ref*{prop_det-direct-image-power}}} $D_{A,m}\equiv mD_{A,1}=0$, i.e. $\det\!\psi_\ast\scrO_X(mA)$ is numerically trivial. Since $\psi$ is flat, then {\hyperref[prop_criterion-num-flat]{Proposition \ref*{prop_criterion-num-flat}}} implies that $\psi_\ast\scrO_X(mA)$ is a numerically flat vector bundle for every $m\in\ZZ_{>0}$. In virtue of {\hyperref[prop_num-flat--local-const]{Proposition \ref*{prop_num-flat--local-const}}} we see that $\psi$ is a locally constant fibration. The proof of {\hyperref[mainthm_Albanese]{Theorem \ref*{mainthm_Albanese}}} is thus finished.    
\section{MRC fibration for \texorpdfstring{$X$}{text} with simply connected smooth locus}
\label{sec_MRC}

Throughout the section, let $X$ be a projective variety equipped with an effective $\QQ$-divisor $\Delta$ such that the pair $(X,\Delta)$ is klt and that the twisted anticanonical divisor $-(K_X+\Delta)$ is nef, and suppose that $\pi_1(X_{\reg})=\{1\}$. Take the $\psi$ in the {\hyperref[general-setting]{General Setting \ref*{general-setting}}} to be the maximally rationally connected (MRC) fibration of $X$ (c.f. \cite[\S 5.4, Theorem 5.13, pp.~128-129]{Deb01}), we will prove in this section that $\psi$ induces a product structure on $X$.

\subsection{Splitting of the tangent sheaf}
\label{ss_MRC_decomp-tangent}
In this subsection we will prove that following decomposition theorem for the tangent sheaf of $X$:
\begin{thm}
\label{thm_decomp-tangent}
Let $X$ be a normal projective variety of semi-Fano type whose smooth locus $X_{\reg}$ is simply connected. Then the tangent sheaf of  $X$ admits a splitting
\[
T_X\simeq\scrF\oplus\scrG
\]
with $\scrF$ and $\scrG$ being algebraically integrable foliations. Moreover, the closure of the general leaf of $\scrF$ is rationally connected and $\det\!\scrG\simeq\scrO_X$.  
\end{thm}

The proof of this result can be divided into four steps:

\paragraph{Step 1: Reduction to the terminal case.}
To prove the theorem, we can assume that the pair $(X,\Delta)$ is terminal and $\QQ$-factorial. In fact, by \cite[Corollary 1.4.3]{BCHM10} we can take a ($\QQ$-factorial) terminal model $g:X^{\terminal}\to X$ of $X$, with an effective $\QQ$-divisor $\Delta^{\terminal}$ on $X^{\terminal}$ such that
\[
K_{X^{\terminal}}+\Delta^{\terminal}\sim_{\QQ}g^\ast(K_X+\Delta).
\]
Then $-(K_{X^{\terminal}}+\Delta^{\terminal})$ is nef. Suppose that $T_{X^{\terminal}}$ admits a decomposition into algebraically integrable foliations
\[
T_{X^{\terminal}}\simeq\scrF^{\terminal}\oplus\scrG^{\terminal}
\]
with $\det\!\scrG^\terminal\simeq\scrO_{X^\terminal}$ and the closure of the general leaf of $\scrF^\terminal$ is rationally connected. Then we get a decomposition $T_X\simeq\scrF\oplus\scrG$ on $X$ with 
\[
\scrF:=\left(g_\ast\scrF^\terminal\right)^{\ast\ast}\text{  and}\quad\scrG:=\left(g_\ast\scrG^\terminal\right)^{\ast\ast}.
\]
By \cite[Proposition 3.3.(3), p.~200]{Kollar96}, the closure of the general leaf of $\scrF$ is also rationally connected. Since $g$ is an isomorphism out of a codimension $2$ subscheme of $X$, then $\det\!\scrG^\terminal\simeq\scrO_{X^\terminal}$ implies that $\det\!\scrG\simeq\scrO_X$. It remains to prove that $(X^{\terminal})_{\reg}$ is simply connected. Since $X_{\reg}\backslash g(\Exceptional(g))$ can be regarded as an Zariski open in $(X^\terminal)_{\reg}$, by \cite[\S 0.7 (B)]{FL81} it suffices to show that $X_{\reg}\backslash g(\Exceptional(g))$ is simply connected. This can be obtained easily by the following topological result:
\begin{lemme}
\label{lemma_fundamental-group-complement}
Let $W$ be a complex manifold and let $Z$ be an analytic subspace of $V$ of codimension $\geqslant 2$. Then the natural morphism $\pi_1(W\backslash Z)\to \pi_1(W)$ induced by the embedding $W\backslash Z\hookrightarrow W$ is an isomorphism.
\end{lemme}
\begin{proof}
This result is of course well known to experts, we nevertheless give the proof for the convenience of the readers. The argument is taken from \cite{Pol16}. Let us argue by induction on $\dim Z$. If $\dim Z=0$, then $\dim W\geqslant 2$, and the lemma results from \cite[\S X.2, Theorem 2.3, p.~146]{God81}. In general, by the induction hypothesis, $\pi_1(W\backslash Z_{\sing})\to\pi_1(W)$ is an isomorphism; then we apply \cite[\S X.2, Theorem 2.3, p.~146]{God81} to $Z_{\reg}\subset W\backslash Z_{\sing}$ to obtain an isomorphism $\pi_1(W\backslash Z)\to \pi_1(W\backslash Z_{\sing})$, hence we have $\pi_1(W\backslash Z)\xrightarrow{\simeq}\pi_1(W)$.     
\end{proof}

\paragraph{Step 2: Triviality of the direct image sheaves.} 
We will prove in this step the following lemma:
\begin{lemme}
\label{lemma_triviality-direct-image}
Let $X$ be a $\QQ$-factorial projective variety and suppose that there is an effective $\QQ$-divisor $\Delta$ on $X$ such that the pair $(X,\Delta)$ is terminal and $-(K_X+\Delta)$ is nef and let everything as in the {\hyperref[general-setting]{General Setting \ref*{general-setting}}} with $\psi$ being the MRC fibration of $X$. Let $A$ be a sufficiently ample divisor on $X$ such that for every $k\in\ZZ_{>0}$ and for general $y\in Y$ the natural morphism
\[
\Sym^k\!\Coh^0(X_y,\scrO_{X_y}(A))\to\Coh^0(X_y,\scrO_{X_y}(kA))
\]
is surjective. Then the following two torsion free sheaves 
\begin{align*}
\scrU_{c,m} &:=\Sym^m\!\phi_\ast\scrO_M(\pi^\ast\!A+cE)\otimes\det\!\phi_\ast\scrO_M(\pi^\ast\!A+cE)^{\otimes-\frac{m}{r}} \\
\scrV_{c,m} &:=\phi_\ast\scrO_M(m\pi^\ast\!A+mcE)\otimes\det\!\phi_\ast\scrO_M(\pi^\ast\!A+cE)^{\otimes-\frac{m}{r}} 
\end{align*}
are trivial on $Y_0$ for every $m\in\ZZ_{>0}$ divisible by $r$. 
\end{lemme}
\begin{proof}
When $X$ is smooth, the theorem is proved in \cite[Proposition 3.7]{CH19}; for the singular case, the proof is much more subtle but the main idea remains the same: take a general complete intersection surface in $X$ and prove the triviality of $\scrU_{c,m}$ and $\scrV_{c,m}$ on this surface, then try to extend the trivializing sections to $Y_0$. For the convenience of the readers, we give the details below. Furthermore, for sake of clarity we divide the proof into five parts:

\subparagraph{1. General settings:}
If $\dim X=1$ then everything is clear, so in the sequel we assume that $\dim X\geqslant 2$. We will only give the proof of triviality on $Y_0$ for $\scrV_{c,m}$, for $\scrU_{c,m}$ the argument is exactly the same (and simpler since $\det\!\scrU_{c,m}\simeq \scrO_Y$). Since $\phi$ is flat over $Y_0$, $\scrV_{c,m}$ is reflexive on $Y_0$, hence in order to prove the triviality of $\scrV_{c,m}$ on $Y_0$, 
it suffices to show that $\scrV_{c,m}$ is trivial on $Y_0\cap Y_{\scrV_{c,m}}$ where $Y_{\scrV_{c,m}}$ is the locally free locus of $\scrV_{c,m}$. For every $a\in\ZZ_{>0}$ set 
\[
D_{A,c,a}:=\frac{1}{r_a}\cdot\text{the Cariter divisor on }Y\text{ associated to the line bundle }\det\!\phi_\ast\scrO_M(a\pi^\ast\!A+acE)
\]
where $r_a:=\rank\!\phi_\ast\scrO_M(a\pi^\ast\!A+acE)$. Then we have
\[
\det\!\scrV_{c,m}\simeq\scrO_Y(r_mD_{A,c,m}-mr_mD_{A,c,1}). 
\]

Since $X$ is not necessarily smooth, the exceptional divisor $E=\Exceptional(\pi)$ can dominate $Y$, which will render the arguments in \cite{CH19} invalid. In order to overcome this difficulty, we set $\Gamma$ to be the normalization of the graph of the rational mapping $\psi$, up to further blow up $M$ we can assume that $\phi:M\to Y$ and $\pi:M\to X$ both factorize through $\Gamma$ and denote by $\bar\phi:\Gamma\to Y$ and $\bar\pi:\Gamma\to X$ the corresponding morphisms. By construction, $\psi$ is almost holomorphic (c.f. \cite[\S 5.4, Definition 5.12, p.~128]{Deb01} and \cite[Definition 2.3]{BCEKPRSW02}), hence $\Exceptional(\bar\pi)$ does not dominate $Y$.  

\begin{center}
\begin{tikzpicture}[scale=2.5]
\node (A) at (0,0) {$Y$.};
\node (B) at (0,1) {$\Gamma$};
\node (C) at (1,1) {$X$};
\node (D) at (-0.5,1.5) {$M$};
\path[->,font=\scriptsize,>=angle 90]
(B) edge node[left]{$\bar\phi$} (A)
(B) edge node[above]{$\bar\pi$} (C)
(D) edge[bend left] node[above]{$\pi$} (C)
(D) edge[bend right] node[below left]{$\phi$} (A)
(D) edge (B);
\path[dashed, ->,font=\scriptsize,>=angle 90]
(C) edge node[below right]{$\psi$} (A);
\end{tikzpicture}
\end{center}

\subparagraph{2. Simple connectedness of a general complete intersection surface in $X$:}
Let $A$ be a very ample divisor on $X$ and take $H_1\,,\cdots,H_{n-1}$ be general hypersurfaces in $|A|$. Set $n:=\dim X$ and let $S=H_1\cap\cdots\cap H_{n-2}$ be the complete intersection surface cut out by $H_1\,,\cdots,H_{n-2}$ (if $n=2$ then we simply take $S=X$). Since terminal singularities are smooth in codimension $2$ (c.f. \cite[Corollary 5.18, p.~159]{KM98}), $S$ is smooth (see also \cite[Theorem 4.5, p.~113]{KM98}). Since $X$ is normal, by \cite[\S 0.7 (B), p.~33]{FL81} we have a surjection between fundamental groups $\pi_1(X_{\reg})\twoheadrightarrow\pi_1(X)$, then $\pi_1(X_{\reg})=\{1\}$ implies that $\pi_1(X)=\{1\}$. We claim that $S$ is also simply connected: 
\begin{itemize}
\item If $n=2$, then $S=X$ is simply connected. 
\item If $n\geqslant 3$, then by \cite[Theorem 1.1.3]{HL85} $X_{\reg}$ has the same homotopy type of the space obtained from $H_1\cap X_{\reg}$ by attaching cells of dimension $\geqslant \dim X$, but the fundamental group of a CW complex only depends on its $2$-skeleton, so that we get an isomorphism 
\[
\pi_1(H_1\cap X_{\reg})\xrightarrow{\simeq}\pi_1(X_{\reg}),
\]
hence $\pi_1(H_1\cap X_{\reg})=\{1\}$. By iterating the argument, we see that $\pi_1(S\cap X_{\reg})=\{1\}$; but since $X$ is smooth in codimension $2$ we have $S\cap X_{\reg}=S$, hence $S$ is simply connected.
\end{itemize}

\subparagraph{3. Triviality of the pullback of $\scrU_{c,m}$ and $\scrV_{c,m}$ to a general complete intersection surface in $X$:}
Now set $\bar E:=\Exceptional(\bar\pi)$, then $\bar\pi|_{\Gamma\backslash\bar E}:\Gamma\backslash\bar E\to X\backslash\bar\pi(\bar E)$ is an isomorphism and $S\cap\bar\pi(\bar E)$ is of dimension $0$. In particular, $\bar\pi\inv(S\backslash\bar\pi(\bar E))=\bar\pi\inv(S)\backslash\bar E$ is smooth. By {\hyperref[lemma_positivity-Um-Vm]{Lemma \ref*{lemma_positivity-Um-Vm}}} $\scrV_{c,m}$ is weakly positively curved on $Y_0$ , in consequence $\left.\bar\phi^\ast\scrV_{c,m}\right|_{\bar\pi\inv(S)\backslash\bar E}$ is also weakly positively curved by {\hyperref[ss_preliminary_direct-image]{\S \ref*{ss_preliminary_direct-image}}}. By viewing $\Gamma\backslash\bar E$ as a Zariski open of $X$ via the isomorphism $\bar\pi|_{\Gamma\backslash\bar E}:\Gamma\backslash\bar E\to X\backslash\bar\pi(\bar E)$, $\left.\bar\phi^\ast\scrV_{c,m}\right|_{\Gamma\backslash\bar E}$ extends to a reflexive sheaf on $X$, denoted by $\scrW_{c,m}$. By the projection formula we have:
\[
\det\!\scrW_{c,m}\simeq(\bar\pi_\ast\bar\phi^\ast\det\!\scrV_{c,m})^{\ast\ast}\simeq(\pi_\ast\phi^\ast\det\!\scrV_{c,m})^{\ast\ast},
\]
hence the (unique up to linear equivalence) Weil divisor associated to $\det\!\scrW_{c,m}$ is equal to 
\[
\pi_\ast\phi^\ast(r_mD_{A,c,m}-mr_mD_{A,c,1}).
\]
Since $X$ is $\QQ$-factorial, {\hyperref[prop_det-direct-image-power]{Proposition \ref*{prop_det-direct-image-power}}} implies that $\det\!\scrW_{c,m}$ is a numerically trivial $\QQ$-line bundle.
Hence by {\hyperref[prop_criterion-num-flat]{Proposition \ref*{prop_criterion-num-flat}}} $\scrW_{c,m}|_S$ is a numerically flat vector bundle on $S$; but $S$ is simply connected, then $\scrW_{c,m}|_S$ is a trivial vector bundle. 

\subparagraph{4. Surjectivity of the restriction morphism:}
Since $\bar\pi(\bar E)$ is of codimension $\geqslant 2$ in $X$, then we have an isomorphism
\[
\Coh^0(\Gamma\backslash\bar E,\bar\phi^\ast\scrV_{c,m})\xrightarrow{\simeq}\Coh^0(X,\scrW_{c,m}).
\]
Since $(X,\Delta)$ is terminal, $X$ has rational singularities and in particular $X$ is Cohen-Macaulay. For $A$ sufficiently ample we have
\[
\Coh^1(X,\scrW_{c,m}\otimes\scrO_X(-H_1))\simeq\Coh^{n-1}(X,\omega_X\otimes\scrW^\ast_m\otimes\scrO_X(H_1))=0,
\]
where $\omega_X$ denotes the dualizing sheaf of $X$. Then the canonical exact sequence $0\to\scrO_X(-H_1)\to\scrO_X\to\scrO_{H_1}\to 0$ induces a surjection 
\[
\Coh^0(\Gamma\backslash\bar E,\bar\phi^\ast\scrV_{c,m})\simeq\Coh^0(X,\scrW_{c,m})\twoheadrightarrow\Coh^0(H_1\,,\scrW_{c,m}|_{H_1}).
\]
By iterating this argument we see that for $C:=S\cap H_{n-1}$ the restriction morphism (since $C$ is disjoint from $\bar\pi(\bar E)$, we can identify $\bar\pi\inv(C)$ and $C$)
\[
\Coh^0(\Gamma\backslash\bar E,\bar\phi^\ast\scrV_{c,m})
\to\Coh^0(C,\scrW_{c,m}|_C)
\]
is surjective. 

\subparagraph{5. Construction of the trivializing sections and Conclusion:}
But $\scrW_{c,m}|_S$ is a trivial vector bundle of rank $r_m$, we get $r_m$ sections $\sigma_1\,,\cdots,\sigma_{r_m}$ in $\Coh^0(\Gamma\backslash\bar E,\bar\phi^\ast\scrV_{c,m})$ whose restrictions to $C$ are everywhere linearly independent. Then $\sigma_1\wedge\cdots\wedge\sigma_{r_m}$ is a non-zero section in $\Coh^0(\Gamma\backslash\bar E,\bar\phi^\ast\!\det\!\scrV_{c,m})$, which extends, via the isomorphism $\pi|_{\Gamma\backslash\bar E}:\Gamma\backslash\bar E\to X\backslash\bar\pi(\bar E)$, to a non-zero section of $\Coh^0(X,\det\!\scrW_{c,m})$; but $\det\!\scrW_{c,m}$ is a numerically trivial $\QQ$-line bundle, then this section must be constant, which implies that $\sigma_1\wedge\cdots\wedge\sigma_{r_m}$ is a non-zero constant.

We claim that for every $i$ there is a section $\tau_i\in\Coh^0(Y_0,\scrV_{c,m})$ such that $\phi^\ast\tau_i=\sigma_i|_{\bar\phi\inv(Y_0)\backslash\bar E}$. The argument is the same as in \cite[Proof of Proposition 3.7]{CH19}. In fact, since $\psi$ is the MRC fibration of $X$, $\bar E$ does not dominate $Y$, then $\sigma_i$ induces a section $\bar\tau_i\in\Coh^0(Y_0\backslash\bar\phi(\bar E),\scrV_{c,m})$. It remains to show that $\bar\tau_i$ extends to $Y_0$. Since $\scrV_{c,m}$ is reflexive on $Y_0$, it suffices to show that $\bar\tau_i$ extends to a general point of any divisor $P$ in $Y_0$. By {\hyperref[prop_bir-geometry-psi]{Proposition \ref*{prop_bir-geometry-psi}(d)}} $\bar\phi^\ast P$ contains at least a reduced component, hence locally around a general point of $P$, $\bar\phi|_{\Gamma\backslash\bar E}\Gamma\backslash\bar E\to Y$ admits a local section, which implies that $\bar\tau_i$ is locally bounded (with respect to any Hermitian metric) around a general point of $P$. Hence by Riemann extension $\bar\tau_i$ extends to $Y_{\scrV_{c,m}}\cap Y_0$ and thus to $Y_0$ by the reflexivity of $\scrV_{c,m}|_{Y_0}$, in this way for every $i$ we obtain a section $\tau_i\in\Coh^0(Y_0,\scrV_{c,m})$ such that $\sigma_i|_{\bar\phi\inv(Y_0)\backslash\bar E}=\bar\phi^\ast\tau_i$.

Now
\[
\tau_1\wedge\cdots\wedge\tau_{r_m}=\left.\bar\phi^\ast(\sigma_1\wedge\cdots\wedge\sigma_{r_m})\right|_{\bar\phi\inv(Y_0)\backslash\bar E}
\]
is a non-zero constant, this implies that the sections $\tau_1\,,\cdots,\tau_{r_m}$ are everywhere linearly independent on $Y_0$. Hence the $\tau_i$'s give a trivialization of $\scrV_{c,m}|_{Y_0}$. 
\end{proof}

\paragraph{Step 3: Birational version of the decomposition.}
In the sequel of the proof of {\hyperref[thm_decomp-tangent]{Theorem \ref*{thm_decomp-tangent}}}, let us fix a very ample divisor $A$ on $X$, such that
\begin{equation}
\label{eq_cond-sym-surj}
\Sym^k\!\Coh^0(X,\scrO_X(A))\to\Coh^0(X,\scrO_X(kA))
\end{equation}
is surjective for every $k\in\ZZ_{>0}$. In this step we will prove that $\phi\inv(Y_0)$ is birational to a product, which can be seen as a birational version of the decomposition theorem for $X$. Let $c_0$ be as in {\hyperref[prop_det-stablize]{Proposition \ref*{prop_det-stablize}}} and let $c$ be any integer $\geqslant c_0$. Set $G:=\pi^\ast\!A+cE$, and for every $a\in\ZZ_{>0}$ set
\[
D_{A,c,a}:=\frac{1}{r_a}\cdot\text{the Cartier divisor on }Y\text{ associated to the line bundle }\det\!\phi_\ast\scrO_M(aG). 
\]
where $r_a:=\rank\!\phi_\ast\scrO_M(aG)$. Then by {\hyperref[lemma_triviality-direct-image]{Lemma \ref*{lemma_triviality-direct-image}}} for every $m\in\ZZ_{>0}$ divisible by $r:=r_1$ the torsion free sheaves
\begin{align*}
\scrU_{c,m} &:=\Sym^m\!\phi_\ast\scrO_M(G)\otimes\scrO_Y(-mD_{A,c,1}) \\
\scrV_{c,m} &:=\phi_\ast\scrO_M(mG)\otimes\scrO_Y(-mD_{A,c,1})
\end{align*}
are trivial on $Y_0$. Up to blowing-up $M$, we can assume that the $\phi$-relative base locus of $G$, i.e. the subscheme of $M$ defined by the ideal sheaf $\Image(\phi^\ast\phi_\ast\scrO_M(G)\otimes\scrO_M(-G)\to\scrO_M)$, is a divisor. Then we can write
\[
G=G_{\base}+G_{\free}
\]
where $G_{\base}$ is the $\phi$-relative fixed part of the linear series $|G|$ and $G_{\free}:=G-G_{\base}$ is $\phi$-relatively generated. Now the adjunction morphism admits a factorization
\[
\phi^\ast\phi_\ast\scrO_M(G)\twoheadrightarrow\scrO_M(G_{\free})\hookrightarrow\scrO_M(G),
\]
that can be pushed down to $Y$ and give morphisms
\[
\phi_\ast\scrO_M(G)\to\phi_\ast\scrO_M(G_{\free})\hookrightarrow\phi_\ast\scrO_M(G).
\]
By construction the composition morphism is the identity, hence the inclusion $\phi_\ast\scrO_M(G_{\free})\hookrightarrow\phi_\ast\scrO_M(G)$ is an isomorphism. Then the surjection $\phi^\ast\phi_\ast\scrO_M(G_{\free})\twoheadrightarrow\scrO_M(G_{\free})$ induces a morphism $\pi_G:M\to\PP(\phi_\ast\scrO_M(G_{\free}))$ such that $\scrO_M(G_{\free})=\pi_G^\ast\scrO_{\PP(\phi_\ast\scrO_M(G_{\free})}(1)$. Set $X_G$ be the image of $\pi_G$ with induced morphism $\psi_G:X_G\to Y$, then we have the following commutative diagram:

\begin{center}
\begin{tikzpicture}[scale=2.5]
\node (A) at (0,0) {$Y$.};
\node (B) at (0,1) {$M$};
\node (C) at (-1,1) {$X$};
\node (D) at (1,1) {$X_G$};
\node (E) at (2,1) {$\PP(\phi_\ast\scrO_M(G_{\free}))$};
\path[->,font=\scriptsize,>=angle 90]
(B) edge node[left]{$\phi$} (A)
(B) edge node[above]{$\pi$} (C)
(B) edge node[above]{$\pi_G$} (D)
(D) edge node[below right]{$\psi_G$} (A)
(E) edge [bend left=20] node[below right]{$p$} (A);
\path[dashed, ->,font=\scriptsize,>=angle 90]
(C) edge node[below left]{$\psi$} (A);
\path[right hook->,font=\scriptsize,>=angle 90]
(D) edge (E);
\end{tikzpicture}
\end{center}

The main purpose of this step is to prove the following lemma
\begin{lemme}
\label{lemma_birational-decomp}
In the above setting, we have $\psi_G\inv(Y_0)\simeq Y_0\times F$, where $F$ denotes the general fibre of $\psi$ (the MRC fibration $\psi$ is almost holomorphic, hence it makes sense to talk about its general fibre).  
\end{lemme}

Before entering into the proof of the above lemma let us first prove the following auxiliary result: 
\begin{lemme}
\label{lemma_general-fibre}
Let everything be as above. Then for general $y\in Y$ we have $G_{\free}|_{M_y}\sim\pi^\ast\!A|_{M_y}$ and $G_{\base}|_{M_y}\sim cE|_{M_y}$. In particular, the general fibre of $\psi_G$ is isomorphic to $F$. 
\end{lemme}
\begin{proof}
Let us first point out that a major difference between the singular case that we consider in this article and the smooth case treated in \cite{CH19,CCM19} is that if $X$ is singular the exceptional divisor $E$ can dominate $Y$, in particular $E|_{M_y}\nsim 0$. For general $y\in Y$ consider the morphism $\pi|_{M_y}:M_y\to X_y$, it is a birational morphism with the exceptional divisor being $E|_{M_y}$. By the projection formula (c.f. \cite[Lemma 7.11]{Deb01}) we have 
\begin{equation}
\label{eq_fibre-proj-formula}
\Coh^0(M_y,\scrO_{M_y}(G)) \simeq \Coh^0(X_y,\scrO_{X_y}(A))
\end{equation}
but $\pi^\ast\!A|_{M_y}$ is globally generated, hence $\pi^\ast\!A|_{M_y}$ is a fortiori the mobile part of $G|_{M_y}$, that is, $G_{\free}|_{M_y}=\pi^\ast\!A|_{M_y}$; then $G_{\base}|_{M_y}=cE|_{M_y}$. Consequently the morphism $\pi_G|_{M_y}:M_y\to (X_G)_y$ is given by the linear series $|\pi^\ast\!A|_{M_y}|$. But $A$ is very ample on $X$, hence for general $y\in Y$ the morphism $\pi_G|_{M_y}$ factors through $X_y$, and its image is isomorphic to $X_y\simeq F$.   
\end{proof}

Now let us turn to the proof of {\hyperref[lemma_birational-decomp]{Lemma \ref*{lemma_birational-decomp}}}:
\begin{proof}[Proof of {\hyperref[lemma_birational-decomp]{Lemma \ref*{lemma_birational-decomp}}}]
The idea of the proof is the same as that of \cite[\S 3.C. Proof of Theorem 1.2, Step 1]{CH19}, we nevertheless give the proof for the convenience of the readers. By \eqref{eq_cond-sym-surj} and \eqref{eq_fibre-proj-formula} the morphism \begin{equation}
\label{eq_surjection-Um-Vm}
\Sym^m\!\phi_\ast\scrO_M(G)\to\phi_\ast\scrO_M(mG) 
\end{equation}
is generically surjective. Twisting with $\scrO_M(-mD_{A,c,1})$ we get a generically surjective morphism $\scrU_{c,m}\to\scrV_{c,m}$, which gives rise to a global section $s\in\Coh^0(Y,\scrU_{c,m}^\ast\otimes\scrV_{c,m})$. By {\hyperref[lemma_triviality-direct-image]{Lemma \ref*{lemma_triviality-direct-image}}}, $\scrU_{c,m}|_{Y_0}$ and $\scrV_{c,m}|_{Y_0}$ are trivial vector bundles, hence $s|_{Y_0}$ is constant by {\hyperref[prop_bir-geometry-psi]{Proposition \ref*{prop_bir-geometry-psi}(d)}}, in particular the morphism $\scrU_{c,m}\to\scrV_{c,m}$ has constant rank over $Y_0$. Consequently the morphism \eqref{eq_surjection-Um-Vm} is surjective over $Y_0$. Now consider the inclusion $\phi_\ast\scrO_M(mG_{\free})\hookrightarrow\phi_\ast\scrO_M(mG)$ we get the following commutative diagram
\begin{center}
\begin{tikzpicture}[scale=2.0]
\node (A) at (0,0) {$\phi_\ast\scrO_M(mG_{\free})$};
\node (B) at (2,0) {$\phi_\ast\scrO_M(mG)$};
\node (C) at (0,1) {$\Sym^m\!\phi_\ast\scrO_M(G_{\free})$};
\node (D) at (2,1) {$\Sym^m\!\phi_\ast\scrO_M(G)$};
\path[->,font=\scriptsize,>=angle 90]
(C) edge node[above]{$\simeq$} (D)
(C) edge (A)
(D) edge (B);
\path[right hook->,font=\scriptsize,>=angle 90]
(A) edge (B);
\end{tikzpicture}
\end{center}
Since right column is the morphism \eqref{eq_surjection-Um-Vm}, which is shown to be surjective over $Y_0$, hence by the Snake Lemma the left column is also surjective over $Y_0$. Again apply the Snake Lemma but exchange the role of rows and of columns, then we find that the bottom row is an isomorphism over $Y_0$. In particular, $\phi_\ast\scrO_M(mG_{\free})\otimes\scrO_Y(-mD_{A,c,1})$ is trivial over $Y_0$. 

Let $\scrI_{X_G}$ be the ideal sheaf of $X_G$ in $\PP(\phi_\ast\scrO_M(G_{\free}))$. Twisting the exact sequence 
\[
0\to \scrI_{X_G}\to\scrO_{\PP(\phi_\ast\scrO_M(G_{\free}))}\to\scrO_{X_G}\to 0
\]
with $\scrO_{\PP(\phi_\ast\scrO_M(G_{\free}))}(m)$ for $m$ sufficiently large and divisible by $r$ and pushing down to $Y$ we get (by relative Serre vanishing):
\begin{equation}
\label{eq_ses-psiG}
0\to p_\ast\scrI_{X_G}\!(m)\to p_\ast\scrO_{\PP(\phi_\ast\scrO_M(G_{\free}))}\!(m)\simeq\Sym^m\!\phi_\ast\scrO_M(G_{\free})\to\psi_{G\ast}\scrO_{X_G}\!(m)\to 0.
\end{equation}
where we adapt the notation that for any coherent sheaf $\scrF$ on $\PP(\phi_\ast\scrO_M(G_{\free}))$ and for any integer $k$ we set $\scrF\!(k):=\scrF\!\otimes\scrO_{\PP(\phi_\ast\scrO_M(G_{\free}))}\!(k)$. Since $\pi_G:M\to X_G$ is birational (because it birational on the general fibre of $\phi$), the natural morphism $\scrO_{X_G}\to\pi_{G\ast}\scrO_M$ is injective, hence by the projection formula we have an injection
\[
\psi_{G\ast}\scrO_{X_G}\!(m)\hookrightarrow\psi_\ast\scrO_M(mG_{\free}).
\]
Now we consider the composition morphism
\[
\Sym^m\!\phi_\ast\scrO_M(G_{\free})\twoheadrightarrow\psi_{G\ast}\scrO_{X_G}\!(m)\hookrightarrow\phi_\ast\scrO_M(mG_{\free}),
\]
which is shown to be surjective over $Y_0$ (the left column of the diagram above), hence the inclusion $\psi_{G\ast}\scrO_{X_G}\!(m)\hookrightarrow\phi_\ast\scrO_M(mG_{\free})$ is an isomorphism over $Y_0$, and in consequence its twisting
\[
\psi_{G\ast}\scrO_{X_G}\!(m)\otimes\scrO_Y(-mD_{A,c,1})
\]
is trivial over $Y_0$. By the exact sequence \eqref{eq_ses-psiG} we see that $p_\ast\scrI_{X_G}\!(m)\otimes\scrO_Y(-mD_{A,c,1})$ is also trivial over $Y_0$. 
By {\hyperref[prop_bir-geometry-psi]{Proposition \ref*{prop_bir-geometry-psi}(d)}} this means that the defining equations of $\psi_G\inv(Y_0)$ in $\PP(\phi_\ast\scrO_M(G_{\free})|_{Y_0})\simeq Y_0\times \PP^{r-1}$ are constant over $Y_0$, hence $\psi_G\inv(Y_0)$ is isomorphic to the product $Y_0\times F$ by {\hyperref[lemma_general-fibre]{Lemma \ref*{lemma_general-fibre}}}. 
\end{proof}

\paragraph{Step 4: Proof of the splitting theorem.}
In this step we will apply {\hyperref[lemma_birational-decomp]{Lemma \ref*{lemma_birational-decomp}}} to conclude. The proof relies on the following auxiliary result:

\begin{lemme}
\label{lemma_exceptional-psiG}
Let everything be as in Step 3, then  every codimension $1$ component of the exceptional locus of $\psi_G|_{\phi\inv(Y_0)}:\phi\inv(Y_0)\to\psi_G\inv(Y_0)$ is contained in $E$. 
\end{lemme}
\begin{proof}
The proof is similar to \cite[\S 3.C. Proof of Theorem 1.2, Step 2]{CH19}, nevertheless in our case $X$ is possibly singular, then $E$ can dominate $Y$ and this renders the argument a little subtle. For the convenience of the readers, we give the proof below. First notice that we have the following observation:
\begin{itemize}
\item[] Since $\pi^\ast\!A$ is $\phi$-relatively generated, hence $G_{\base}\leqslant cE$. Let $\Gamma$ be a component of any fibre of $\phi$ not contained in $E$, then every component of $E$ restricts to an effective divisor on $\Gamma$, hence
\[
G_{\free}|_{\Gamma}=\pi^\ast\!A|_{\Gamma}+(cE-G_{\base})|_{\Gamma}     
\]
is big, and thus $\Gamma$ is not contracted by $\psi_G$.
\end{itemize}
Now let us turn to the proof of the lemma. Let $D\subset\phi\inv(Y_0)$ be an irreducible Weil divisor contained in the exceptional locus of $\psi_G|_{\phi\inv(Y_0)}$. Consider the two cases separately:
\begin{itemize}
\item If $D$ is $\phi$-horizontal. Then for general $y\in Y_0$, $D|_{M_y}$ is $\psi_G|_{M_y}$-exceptional. But $\psi_G|_{M_y}:M_y\to(X_G)_y\simeq F=X_y$ is induced by the divisor $\pi^\ast\!A|_{M_y}$, hence $D|_{M_y}$ is contained in $E|_{M_y}$ and thus $D$ is contained in $E$.  
\item If $D$ is $\phi$-vertical. Since $\phi$ is flat over $Y_0$\,, $\phi(D)$ is also a divisor. For the general fibre of $\phi|_D:D\to\phi(D)$, it is contracted by $\psi_G$, then by the observation above it is contained in $E$. Therefore $D$ is contained in $E$.  
\end{itemize}
\end{proof}

By {\hyperref[lemma_birational-decomp]{Lemma \ref*{lemma_birational-decomp}}} we have $\psi_G\inv(Y_0)\simeq Y_0\times F$, then 
\begin{equation}
\label{eq_decomp-XG}
T_{\psi_G\inv(Y_0)}\simeq\pr_1^\ast\!T_{Y_0}\oplus\pr_2^\ast\!T_F.
\end{equation}
Set $X_0:=\phi\inv(Y_0)\backslash E$, which can be regarded as a Zariski open of $X$ via the embedding $\pi|_{M\backslash E}:M\backslash E\hookrightarrow X$. By {\hyperref[lemma_exceptional-psiG]{Lemma \ref*{lemma_exceptional-psiG}}}, $\psi_G|_{X_0}:X_0\to \psi_G\inv(Y_0)\simeq Y_0\times F$ is an embedding out of a codimension $\geqslant 2$ subscheme. Hence the decomposition \eqref{eq_decomp-XG} induces a decomposition 
\begin{equation}
\label{eq_decomp-X0}
T_{X_0}\simeq \scrF^\circ\oplus \scrG^\circ, 
\end{equation}
with $\scrF^\circ$ (resp. $\scrG^\circ$) corresponding to $\pr_2^\ast\!T_F$ (resp. $\pr_1^\ast\!T_{Y_0}$). By construction, $\scrF^\circ$ and $\scrG^\circ$ are algebraically integrable foliations over $X_0$, with the closure of a general leaf of $\scrF^\circ$ equal to a Zariski open of $F$ and 
\[
K_{\scrG^\circ}\sim\pr_1^\ast\!K_{Y_0}=\pr_1^\ast\!K_Y|_{Y_0}\sim_{\QQ}0.
\]
since by {\hyperref[prop_bir-geometry-psi]{Proposition \ref*{prop_bir-geometry-psi}(a)}} any effective $\QQ$-divisor $\QQ$-linearly equivalent to $K_Y$ is supported out of $Y_0$. By {\hyperref[prop_bir-geometry-psi]{Proposition \ref*{prop_bir-geometry-psi}(c)}}, $X\backslash X_0$ has codimension $\geqslant 2$, hence \eqref{eq_decomp-X0} gives rise to a decomposition  
\[
T_X\simeq\scrF\oplus\scrG.
\]
with $\scrF$ (resp. $\scrG$) being the reflexive hull of the extension of $\scrF^\circ$ (resp. of $\scrG^\circ$) to $X$. By {\hyperref[lemma_involutive]{Lemma \ref*{lemma_involutive}}} $\scrF$ and $\scrG$ are algebraically integrable foliations; moreover, the Zariski closure of a general leaf of $\scrF$ is rationally connected (in fact equal to $F$) and $K_{\scrG}\sim_{\QQ} 0$. This means that $\det\!\scrG|_{X_{\reg}}$ is a torsion line bundle on $X_{\reg}$, but $\pi_1(X_{\reg})=\{1\}$, then $\det\!\scrG|_{X_{\reg}}$ and thus $\det\!\scrG$ must be trivial. As a byproduct we get additional information on the splitting:
\begin{lemme}
\label{lemma_open-product}
Let everything be as in the {\hyperref[general-setting]{General Setting \ref{general-setting}}} with $\psi$ being the MRC fibration of $X$ and suppose that the smooth locus $X_{\reg}$ of $X$ is simply connected.  
Then there is a Zariski open subset $X_0$ of $X$ such that $X_0$ is embedded into the product space $Y_0\times F$.
\end{lemme}
\begin{proof}
We have proved this for $(X,\Delta)$ terminal. For the klt case, let us take a terminal model $g:(X^{\terminal},\Delta^{\terminal})\to (X,\Delta)$. 
Then there is a Zariski open $(X^{\terminal})_0$ such that $(X^{\terminal})_0$ can be embedded into $Y_0\times F$. Then $(X^{\terminal})_0\backslash\Exceptional(g)$ can be regarded as a Zariski open $X_0$ of $X$, whose complement is of codimension $\geqslant 2$ in $X$. Clearly $X_0$ can be embedded into $Y_0\times F$. 
\end{proof}

\begin{rmq}
To end this subsection let us make a remark about how to show that $\phi_\ast\scrO_M(mG_{\free})\hookrightarrow\phi_\ast\scrO_M(mG)$ is an isomorphism over $Y_0$ in the proof of {\hyperref[lemma_birational-decomp]{Lemma \ref*{lemma_birational-decomp}}} in Step 3 above. By taking 
\begin{align*}
\scrA_1 &:= \left.\Sym^m\!\phi_\ast\scrO_M(G_{\free})\right|_{Y_0}, & \scrB_1 &:= \left.\Sym^m\!\phi_\ast\scrO_M(G)\right|_{Y_0}, \\
\scrA_2 &:= \phi_\ast\scrO_M(mG_{\free})|_{Y_0}, & \scrB_2 &:= \phi_\ast\scrO_M(mG)|_{Y_0},
\end{align*}
we have the following commutative square:
\begin{center}
\begin{tikzpicture}[scale=2.5]
\node (A) at (0,0) {$\scrA_2$};
\node (B) at (1,0) {$\scrB_2$,};
\node (C) at (0,1) {$\scrA_1$};
\node (D) at (1,1) {$\scrB_1$};
\path[->,font=\scriptsize,>=angle 90]
(C) edge node[above]{$c_1$} node[below]{$\simeq$} (D)
(C) edge node[left]{$a$} (A)
(D) edge node[right]{$b$}(B);
\path[right hook->,font=\scriptsize,>=angle 90]
(A) edge node[below]{$c_2$} (B);
\end{tikzpicture}
\end{center}
with $b$ being surjective. By completing the two row into short exact sequences we get
\begin{center}
\begin{tikzpicture}[scale=2.5]
\node (A) at (0,0) {$\scrA_2$};
\node (B) at (1,0) {$\scrB_2$};
\node (C) at (0,1) {$\scrA_1$};
\node (D) at (1,1) {$\scrB_1$};
\node (E) at (-1,0) {$0$};
\node (F) at (2,0) {$\Coker(c_2)$};
\node (G) at (3,0) {$0$.};
\node (E') at (-1,1) {$0$};
\node (F') at (2,1) {$0$};
\node (G') at (3,1) {$0$};
\path[->,font=\scriptsize,>=angle 90]
(C) edge node[above]{$c_1$} (D)
(C) edge node[left]{$a$} (A)
(D) edge node[right]{$b$}(B)
(F') edge (F)
(E) edge (A)
(B) edge (F)
(F) edge (G)
(E') edge (C)
(D) edge (F')
(F') edge (G')
(A) edge node[below]{$c_2$} (B);
\end{tikzpicture}
\end{center}
Since $b$ is surjective, the Snake Lemma implies that $\Coker(a)\simeq\Coker(b)=0$, hence $a$ is surjective. Then exchange the role of rows and of columns we get 
\begin{center}
\begin{tikzpicture}[scale=2.5]
\node (A) at (0,0) {$\scrB_1$};
\node (B) at (1,0) {$\scrB_2$};
\node (C) at (0,1) {$\scrA_1$};
\node (D) at (1,1) {$\scrA_2$};
\node (E) at (-2,0) {$0$};
\node (F) at (-1,0) {$\Ker(b)$};
\node (G) at (2,0) {$0$.};
\node (E') at (-2,1) {$0$};
\node (F') at (-1,1) {$\Ker(a)$};
\node (G') at (2,1) {$0$};
\path[->,font=\scriptsize,>=angle 90]
(C) edge node[above]{$a$} (D)
(C) edge node[left]{$c_1$} node[right]{$\simeq$} (A)
(F') edge (F)
(E) edge (F)
(F) edge (A)
(B) edge (G)
(E') edge (F')
(F') edge (C)
(D) edge (G')
(A) edge node[below]{$b$} (B);
\path[right hook->,font=\scriptsize,>=angle 90]
(D) edge node[right]{$c_2$}(B);
\end{tikzpicture}
\end{center}
Again by the Snake Lemma we have $\Coker(c_2)\simeq\Coker(c_1)=0$, hence $c_2$ is also surjective. Clearly this argument works in any Abelian category.
\end{rmq}

\subsection{Decomposition theorem for  \texorpdfstring{$X$}{text}}
\label{ss_MRC_decomp-thm}
In this subsection, let us prove {\hyperref[mainthm_MRC]{Theorem \ref*{mainthm_MRC}}}. Let $X$ be a projective variety of semi-Fano type with simply connected smooth locus $X_{\reg}$. Then there is an effective $\QQ$-divisor on $X$ such that $(X,\Delta)$ is klt and that the twisted anticanonical divisor $-(K_X+\Delta)$ is nef. By {\hyperref[ss_MRC_decomp-tangent]{\S \ref*{ss_MRC_decomp-tangent}}} we have a direct decomposition of the tangent sheaf into reflexive subsheaves:
\[
T_X\simeq\scrF\oplus\scrG.
\]
with $\scrF$ and $\scrG$ algebraically integrable foliations. Moreover, the Zariski closure of a general leaf of $\scrF$ is rationally connected and $\det\!\scrG\simeq\scrO_X$. Set $F$ (resp. $Z$) the Zariski closure of the general leaf of $\scrF$ (resp. of $\scrG$) and we will prove in the sequel that $X\simeq Z\times F$. In fact, if $\Delta=0$, this can be immediately deduced from the more general result of St\'ephane Druel \cite[Theorem 1.5]{Druel18b} on the foliations with numerically trivial canonical class, as will be discussed in  {\hyperref[sec_foliation-num-trivial]{\S \ref*{sec_foliation-num-trivial}}}. Nevertheless, we will present here a more elementary proof of the decomposition {\hyperref[mainthm_MRC]{Theorem \ref*{mainthm_MRC}}}, since the argument can be also be applied to the more general case without assumption on the fundamental group, and we hope that it can be used to give a proof of {\hyperref[conj_klt-anti-nef]{Conjecture \ref*{conj_klt-anti-nef}}} without proving {\hyperref[conj_pi1-anti-nef]{Conjecture \ref*{conj_pi1-anti-nef}}} or at least reducing it to a much weaker result on the fundamental group than {\hyperref[conj_pi1-anti-nef]{Conjecture \ref*{conj_pi1-anti-nef}}}. The key observation is that the decomposition $T_X\simeq\scrF\oplus\scrG$ implies that $\scrF$ and $\scrG$ are weakly regular foliations by \cite[Lemma 5.8]{Druel18b} (c.f. {\hyperref[defn_sing-foliation]{Definition \ref*{defn_sing-foliation}}} or \cite[Definition 5.4]{Druel18b} for the definition of the weak regularity).

\paragraph{Step 1: Simple connectdeness of the general leaf.}
In this first step, let us prove the following preparatory result on the topology of the general leaves of the foliations $\scrF$ and $\scrG$:
\begin{lemme}
\label{lemma_simply-conn-leaf}
As above let $F$ (resp. $Z$) be the Zariski closure of a general leaf of $\scrF$ (resp. of $\scrG$). Then both $F_{\reg}$ and $Z_{\reg}$ are simply connected.
\end{lemme}
\begin{proof}
This follows easily from {\hyperref[lemma_open-product]{Lemma \ref*{lemma_open-product}}}. In fact, by {\hyperref[lemma_open-product]{Lemma \ref*{lemma_open-product}}}, there is a Zariski open $X_0$ of $X$ which can be embedded into $Y_0\times F$ such that $\codim_X(X\backslash X_0)\geqslant 2$. Up to shrinking $Y$ we can assume that $X_0\subseteq X_{\reg}$, then we have $\codim_{X_{\reg}}(X_{\reg}\backslash X_0)\geqslant 2$. But $\pi_1(X_{\reg})\simeq\{1\}$, then {\hyperref[lemma_fundamental-group-complement]{Lemma \ref*{lemma_fundamental-group-complement}}} implies that $\pi_1(X_0)\simeq\{1\}$. Since $X_0$ is smooth, it can be regarded as a Zariski open in $Y_0\times F_{\reg}$. Then by \cite[\S 0.7 (B), p.~33]{FL81}, we have $\pi_1(Y_0\times F_{\reg})\simeq\{1\}$, which implies that $\pi_1(Y_0)\simeq\pi_1(F_{\reg})\simeq\{1\}$. Again by  {\hyperref[lemma_open-product]{Lemma \ref*{lemma_open-product}}}, we see that $Y_0$ can be regarded as a Zariski open of $Z$ (and thus of $Z_{\reg}$ since $Y_0$ is smooth). Then by \cite[\S 0.7 (B), p.~33]{FL81} $\pi_1(Z_{\reg})\simeq\{1\}$.    
\end{proof}

\paragraph{Step 2: Reduction to the $\QQ$-factorial terminal case.}
As in the {\hyperref[ss_Albanese_Q-fact]{\S \ref*{ss_Albanese_Q-fact}}}, in this step we will reduce the proof of {\hyperref[mainthm_MRC]{Theorem \ref*{mainthm_MRC}}} to the $\QQ$-factorial case. Assume that {\hyperref[mainthm_MRC]{Theorem \ref*{mainthm_MRC}}} for $X$ with terminal $\QQ$-factorial singularities, let us prove that it holds for general $X$. To this end, we take a ($\QQ$-factorial) terminal model $g: X^{\terminal}\to X$ of $X$ (by \cite[Corollary 1.4.3]{BCHM10}). By construction $X^{\terminal}$ is equipped with an effective $\QQ$-divisor $\Delta^{\terminal}$ on $X^{\terminal}$ such that
\[
K_{X^{\terminal}}+\Delta^{\terminal}\sim_{\QQ}g^\ast(K_X+\Delta).
\] 
hence the twisted anticanonical $-(K_{X^{\terminal}}+\Delta^{\terminal})$ is nef. By our assumption, the MRC fibration of $X^{\terminal}$ induces a decomposition $X^{\terminal}\simeq Z^{\terminal}\times F^{\terminal}$ with $K_{Z^{\terminal}}\sim 0$ and $F^{\terminal}$ rationally connected. But by  {\hyperref[lemma_simply-conn-leaf]{Lemma \ref*{lemma_simply-conn-leaf}}} the irregularity of $F^{\terminal}$ is zero, hence by  \cite[Lemma 4.6]{Druel18a} we get a decomposition $X\simeq Z\times F$, and we have $K_Z\sim 0$ and $F$ rationally connected.   

\paragraph{Step 3: Weak Regularity of the foliations and everywhere-definedness of the MRC fibration.}
In the sequel we always assume that $X$ has $\QQ$-factorial terminal singularities. As pointed above, $\scrF$ and $\scrG$ are weakly regular foliations. By construction $\scrF$ is an algebraically integrable foliation, we intend to apply {\hyperref[thm_weak-reg-foliation]{Theorem \ref*{thm_weak-reg-foliation}}} (\cite[Theorem 6.1]{Druel18b}) to prove that $\scrF$ is induced by an equidimensional fibre space. To this end, we need to show:

\begin{lemme}
\label{lemma_F-canonical-sing}
Let everything as above, then the foliation $\scrF$ has canonical singularities (c.f. \cite[Definition 4.1]{Druel18b} or {\hyperref[defn_foliation-singularities]{Definition \ref*{defn_foliation-singularities}}} below).
\end{lemme}
\begin{proof}   
If $K_{\scrF}$ is Cartier, then the lemma follows immediately from {\hyperref[lemma_foliation-weak-reg-canonical]{Lemma \ref*{lemma_foliation-weak-reg-canonical}}} below (\cite[Lemma 5.9]{Druel18b}). In the general case, $K_\scrF\sim K_X$ is only $\QQ$-Cartier, in order to prove the lemma we will make use of the fact that $-(K_X+\Delta)$ is nef and apply \cite[Proposition 5.5]{Druel17}; in fact, we will prove more generally that $(\scrF,\Delta)$ is canonical (c.f. \cite[\S 5.1]{Druel17} or \cite[Definition 2.9]{Spicer20}; by {\hyperref[prop_bir-geometry-psi]{Proposition \ref*{prop_bir-geometry-psi}}}, $\Delta$ is horizontal with respect to the MRC fibration, hence any component of $\Delta$ is not invariant by $\scrF$). Let $f:V\to W$ be the family of leaves of $\scrF$, with the natural morphism $\beta: V\to X$. Then by {\hyperref[prop_canonical-family-leaves]{Proposition \ref*{prop_canonical-family-leaves}}} and {\hyperref[rmk_example-alg-foliation]{Remark \ref*{rmk_example-alg-foliation}}}, there is an effective $\beta$-exceptional divisor $B$ on $V$ such that
\begin{equation}
\label{eq_canonical-family-leaves-F}
K_{\beta\inv\!\scrF}+B\sim K_{V/W}-\Ramification(f)+B\sim_{\QQ}\beta^\ast\!K_{\scrF},
\end{equation}
But since $K_{\scrF}\sim K_X$ and since $(X,\Delta)$ is terminal (thus $X$ is terminal by $\QQ$-factoriality), we must have $f^\ast\!K_W+\Ramification(f)-B\geqslant 0$. In particular, $B$ is $f$-vertical. By \cite[Remark 3.12]{AD14} or \cite[2.13]{Druel17} the general log leaf of $\scrF$ is $(V_w,B|_{V_w})$ for $w\in W$ general; since $B$ is $f$-vertical, $B|_{V_w}=0$. Moreover, by \eqref{eq_canonical-family-leaves-F} $(V,\beta^\ast\!\Delta+B-f^\ast\!K_W-\Ramification(f))$ is terminal (c.f. \cite[3.5 Definition]{Kollar97}), then so is $(V_w,\beta^\ast\!\Delta|_{V_w})$ for general $w\in W$ by \cite[Lemma 5.17, pp.~158-159]{KM98}. Finally, by writing 
\[
-K_{\scrF}\sim -K_X= -(K_X+\Delta)+\Delta
\]
with $-(K_X+\Delta)$ nef and $\Delta$ effective, we see that the foliated pair $(X,\Delta,\scrF)$ satisfies the condition of \cite[Proposition 5.5]{Druel17} and hence $(\scrF,\Delta)$ is canonical. 

\end{proof}

By virtue of {\hyperref[lemma_F-canonical-sing]{Lemma \ref*{lemma_F-canonical-sing}}} above, we can apply {\hyperref[thm_weak-reg-foliation]{Theorem \ref*{thm_weak-reg-foliation}}} (\cite[Theorem 6.1]{Druel18b}) to conclude that $\scrF$ is induced by a surjective equidimensional fibre space $f: X\to W$ onto a normal projective variety $W$. By construction, $W$ is not uniruled. Moreover we have
\begin{lemme}
\label{lemma_pi1-base-RC-foliation}
Let everything be as above, then $W_{\reg}$ is simply connected.
\end{lemme}
\begin{proof}
Since $X$ has terminal singularities, by \cite[Theorem 5.22, pp.~161-162]{KM98} or \cite[Théorème 1]{Elkik81} it has rational singularities and in particular $X$ is Cohen-Macaulay, hence by the miracle flatness \cite[Theorem 23.1, p.~179]{Mat89} the projective morphism $f|_{f\inv\!W_{\reg}}:f\inv\!W_{\reg}\to W_{\reg}$ is flat. By \cite[Theorem 23.7, p.~182]{Mat89} we see that $X_{\reg}\subseteq f\inv\!W_{\reg}$ and $X$ is smooth at $x\in f\inv\!W_{\reg}$ if and only if the fibre $X_{f(x)}$ is smooth at $x$. Hence $\scrF$ is locally free over $X_{\reg}$ and consequently $\scrF|_{X_{\reg}}$ and $\scrG|_{X_{\reg}}$ are both regular foliations on $X_{\reg}$. Then the tangent bundle sequence of the smooth morphism $f|_{X_{\reg}}:X_{\reg}\to W_{\reg}$ gives rise to an isomorphism $\scrG|_{X_{\reg}}\simeq f^\ast T_{W_{\reg}}$; and this means that the restricted morphism $f|_{Z_{\reg}}:Z_{\reg}\to W_{\reg}$ is an \'etale cover, but $f|_{Z_{\reg}}$ is also projective, hence it is a finite \'etale cover. By {\hyperref[lemma_simply-conn-leaf]{Lemma \ref*{lemma_simply-conn-leaf}}} $Z_{\reg}$ is simply connected, hence $f|_{Z_{\reg}}$ is the universal cover of $W_{\reg}$ and thus $\pi_1(W_{\reg})\simeq\pi_1^{\text{\'et}}(W_{\reg})$ is finite. Since $f$ is a fibre space, by \cite[\S X.4, Corollary 1.4, p.~263]{SGA1} we have an exact sequence of \'etale fundamental groups
\[
\pi_1^{\text{\'et}}(F)\to\pi_1^{\text{\'et}}(f\inv W_{\reg})\to\pi_1^{\text{\'et}}(W_{\reg})\to 1
\]
But since $X_{\reg}$ is simply connected, by \cite[\S 0.7 (B), p.~33]{FL81} we have $\pi_1^{\text{\'et}}(f\inv W_{\reg})= \{1\}$ and thus $\pi_1(W_{\reg})\simeq\pi_1^{\text{\'et}}(W_{\reg})\simeq\{1\}$.  
\end{proof}

\paragraph{Step 4: Decomposition of $X$.}
As shown in the preceding step, the MRC fibration is everywhere defined, then the sequel of the proof is quite similar to the argument in {\hyperref[ss_Albanese_local-constant]{\S \ref*{ss_Albanese_local-constant}}}. Take a desingularization $\mu:W'\to W$ of $W$, and let $X':=X\underset{W}{\times}W'$ be fibre product, equipped with the natural morphisms $\mu_X:X'\to X$ and $f':X'\to W'$. Up to further blowing up $M$ and $Y$ in the {\hyperref[general-setting]{General Setting \ref*{general-setting}}}, we can assume that $\pi$ factorizes through $\mu_X$ and $W'=Y$, and let $\pi':M\to X'$ be the induced morphism. Since $W$ is not uniruled, then so is $Y=W'$.
\begin{center}
\begin{tikzpicture}[scale=2.5]
\node (A) at (0,0) {$W$.};
\node (B) at (0,1) {$X$};
\node (A1) at (-1.2,0) {$W'$};
\node (B1) at (-1.2,1) {$X':=X\underset{W}{\times}W'$};
\node (C) at (-2,2) {M};
\path[->,font=\scriptsize,>=angle 90]
(B) edge node[right]{$f$} (A)
(B1) edge node[left]{$f'$} (A1)
(B1) edge node[above]{$\mu_X$} (B)
(A1) edge node[below]{$\mu$} (A)
(C) edge node[above right]{$\pi'$} (B1)
(C) edge[bend left] node[above]{$\pi$} (B)
(C) edge[bend right] node[left]{$\phi$} (A1);
\end{tikzpicture}
\end{center}

By {\hyperref[prop_bir-geometry-psi]{Proposition \ref*{prop_bir-geometry-psi}}} $f$ is semistable in codimension $1$, hence the ramification divisor of $f$ is $0$ (c.f. \cite[Definition 2.16]{CKT16}), then by \cite[Lemma 2.31]{CKT16} we have $K_{X/W}\sim K_{\scrF}\sim K_X$, which implies in particular that $K_W\sim 0$. Since $f$ is equidimensional and since $W'$ is smooth, by \cite[Proposition (9)]{Kle80} we have $K_{X'/W'}\sim \mu_X^\ast\!K_{X/W}\sim\mu_X^\ast\!K_X$.  Since $\Delta$ is horizontal with respect to $f$ by {\hyperref[prop_bir-geometry-psi]{Proposition \ref*{prop_bir-geometry-psi}(b)}}, the pullback $\mu_X^\ast\Delta$ is horizontal with respect to $f'$ by the {\hyperref[prop_horizon-base-change]{Proposition \ref*{prop_horizon-base-change}}}, hence a fortiori we have $\mu_X^\ast\Delta=(\mu_X)\inv_\ast\!\Delta$ (noting that every $\mu_X$-exceptional divisor is $f'$-vertical) and thus we can rewrite \eqref{eq_rel-canonical-pi-exc} as 
\[
-(K_{M/W'}+\Delta_M)+E'\sim_{\QQ}-(\pi')^\ast(K_{X'/W'}+(\mu_X)\inv_\ast\Delta)\sim_{\QQ}-\pi^\ast(K_X+\Delta),
\]
with $E'$ being $\pi'$-exceptional.

Take a very ample divisor $A$ on $X$, such that for general $w\in W$ the natural morphism
\[
\Sym^k\!\Coh^0(X_w,\scrO_{X_w}(A))\to\Coh^0(X_w,\scrO_{X_w}(kA))
\]
is surjective for every $k$. For every integer $b$ set
\begin{align*}
D_{A,b} &:= \frac{1}{r_b}\cdot\text{the Weil divisor on }W\text{ associated to the rank }1\text{ reflexive sheaf }\det\!f_\ast\scrO_X(bA), \\
D'_{A,b} &:= \frac{1}{r_b}\cdot\text{the Cartier divisor on }W'\text{ associated to the line bundle }\det\!f'_\ast\scrO_{X'}(b\mu_X^\ast A),
\end{align*}
where $r_b:=\rank\!f_\ast\scrO_X(bA)$. Then by construction we have $\mu_\ast D'_{A,b}=D_{A,b}$ and 
\[
\pi_\ast\phi^\ast\!D'_{A,b}\sim\mu_{X\ast}(f')^\ast\!D'_{A,b}\sim f^\ast\!\mu_\ast D'_{A,b}=f^\ast D_{A,b}.
\]
Notice that since $f$ is equidimensional and $W$ is normal, the pullback of Weil divisors via $f$ is defined, c.f. \cite[Construction 2.13]{CKT16}. Since $X$ is $\QQ$-factorial and since $\pi^\ast\!A$ is big, by {\hyperref[prop_more-psef-det]{Proposition \ref*{prop_more-psef-det}}} the ($\QQ$-Cartier) $\QQ$-divisor
\[
A-f^\ast\!D_{A,1}\sim\pi_\ast(\pi^\ast\!A-\phi^\ast\!D'_{A,1})
\]
is pseudoeffective.

By {\hyperref[prop_det-direct-image-power]{Proposition \ref*{prop_det-direct-image-power}}}, 
up to multiplying $A$ by a integer divisible by $r$, we can assume that $f^\ast\!D_{A,1}$ is an integral Cartier divisor (noting that $\Pic^0(X)$ is an Abelian variety, thus divisible). In consequence, by replacing $A$ by $A-f^\ast\!D_{A,1}$, we get an integral Cartier divisor $A$ on $X$ such that:
\begin{itemize}
\item $A$ is pseudoeffective on $X$;
\item $A$ is $f$-very ample;
\item for general $w\in W$ and for any $k\in\ZZ_{>0}$ the natural morphism 
\[
\Sym^k\!\Coh^0(X_w,\scrO_{X_w}(A))\to\Coh^0(X_w,\scrO_{X_w}(kA))
\]
is surjective;
\item $D_{A,1}$ is trivial.
\end{itemize}

Since $\pi$ is birational, $\pi^\ast\!A$ is $\phi$-big and by \cite[Lemma 7.11]{Deb01} the natural morphism 
\[
\Sym^k\!\Coh^0(M_y\,,\scrO_{M_y}(\pi^\ast\!A))\to\Coh^0(M_y,\scrO_{M_y}(k\pi^\ast\!A))
\]
is surjective for all $k\in\ZZ_{>0}$. Then by {\hyperref[prop_positivity-anti-nef]{Proposition \ref*{prop_positivity-anti-nef}}} we have that (noting that $E'$ is $\pi'$-exceptional)
\[
\mu^\ast\!f_\ast\scrO_X(mA)\simeq f'_\ast\scrO_{X'}(m\mu_X^\ast\!A)\simeq \phi_\ast\scrO_M(m\pi^\ast A)\simeq\phi_\ast\scrO_M(m\pi^\ast\!A+pE')
\]
is weakly semipositively curved for every $m\in\ZZ_{>0}$. Moreover, by {\hyperref[prop_det-direct-image-power]{Proposition \ref*{prop_det-direct-image-power}}} $f^\ast\!D_{A,m}\equiv mf^\ast\!D_{A,1}=0$, i.e. $\det\!f_\ast\scrO_X(mA)$ is numerically trivial, and so is $\det\!f'_\ast\scrO_{X'}(m\mu_X^\ast\!A)$. Since $f'$ is equidimensional, $f'_\ast\scrO_{X'}(m\mu_X^\ast\!A)$ is reflexive for every $m$, then {\hyperref[prop_criterion-num-flat]{Proposition \ref*{prop_criterion-num-flat}}} implies that $f'_\ast\scrO_{X'}(m\mu_X^\ast\!A)$ is numerically flat for every $m\in\ZZ_{>0}$. By \cite[\S I\!I\!I.9, Proof of Proposition 9.9, pp.~261-262]{Har77} (c.f. also \cite[\S I\!X.2, Proposition (2.5), p.~8]{ACG11}) the local freeness of $f'_\ast\scrO_{X'}(m\mu_X^\ast\!A)$ implies that $f'$ is flat. Then by virtue of {\hyperref[prop_num-flat--local-const]{Proposition \ref*{prop_num-flat--local-const}}} we see that $f'$ is a locally constant fibration. Since $W_{\reg}$ is simply connected by {\hyperref[lemma_pi1-base-RC-foliation]{Lemma \ref*{lemma_pi1-base-RC-foliation}}}, then by \cite[\S 0.7 (B), p.~33]{FL81} so is $Y=W'$. Hence $f'$ induces a decomposition $X'\simeq F\times W'$. The decomposition of $X$ then follows from \cite[Lemma 4.6]{Druel18a}. In addition, the decomposition is induced by $f$, hence a fortiori $W\simeq Z$ and hence $X\simeq F\times Z$. Thus we have just proved {\hyperref[mainthm_MRC]{Theorem \ref*{mainthm_MRC}}}.

\section{Foliations with numerically trivial canonical class}
\label{sec_foliation-num-trivial}

As mentioned at the beginning of {\hyperref[ss_MRC_decomp-thm]{\S \ref*{ss_MRC_decomp-thm}}}, {\hyperref[mainthm_MRC]{Theorem \ref*{mainthm_MRC}}} can be deduced directly by combining {\hyperref[thm_decomp-tangent]{Theorem \ref*{thm_decomp-tangent}}} with the following theorem, which is a variant of \cite[Theorem 1.5]{Druel18b}. The proof presented in this section is suggested by Stéphane Druel to the author through personal communications (any mistake, of course, is the author's).
\begin{thm}
\label{thm_Druel_foliation-num-triv}
Let $X$ be a normal projective variety admitting an effective $\QQ$-divisor $\Delta$ on $X$ such that $(X,\Delta)$ is klt and let $\scrG$ be an algebraically integrable foliation with canonical singularities. Suppose that the canonical class of $\scrG$ is numerically trivial. Then there are projective varieties $Z$ and $F$ and a finite quasi-\'etale cover $f:Z\times F\to X$, such that $f\inv\scrG\simeq\pr_1^\ast T_Z$.    
\end{thm}

Before entering into the proof of {\hyperref[thm_Druel_foliation-num-triv]{Theorem \ref*{thm_Druel_foliation-num-triv}}}, let us first recall the notion of singularities of foliations:
\begin{defn}[{\cite[Defintion 4.1]{Druel18b}}; see also {\cite[\S I.5]{McQ08},\cite[Section 3]{LPT18},}]
\label{defn_foliation-singularities}
Let $\scrG$ be a $\QQ$-Gorenstein foliation on a normal complex variety $X$. For any projective bimeromorphic morphism $\beta: V\to X$ with $V$ smooth, there are uniquely determined (c.f. \cite[Remark 3.2]{LPT18}) rational numbers $a(E,X,\scrG)$ such that 
\[
\beta^\ast\det\!\scrG\simeq\det\!\beta\inv\!\scrG+\sum_E a(E,X,\scrG)E\,,
\]
as $\QQ$-line bundles. where $E$ runs over all the exceptional prime divisors of $\beta$. The number $a(E,X,\scrG)$ does not depend on $\beta$ but only depends on the valuation defined by $E$ on the function filed of $X$. We say that $\scrG$ has {\it canonical} (resp. {\it terminal}) singularities if for every $E$ exceptional over $X$, $a(E,X,\scrG)\geqslant 0$ (resp. $a(E,X,\scrG)>0$).
\end{defn}

In particular, weakly regular foliations (c.f. {\hyperref[defn_sing-foliation]{Definition \ref*{defn_sing-foliation}}}) on klt varieties have canonical singularities. Indeed we have:
\begin{lemme}[{\cite[Lemma 5.9]{Druel18b}}]
\label{lemma_foliation-weak-reg-canonical}
Let $X$ be a normal complex variety admitting an effective $\QQ$-divisor $\Delta$ such that $(X,\Delta)$ is klt, and let $\scrG$ be a foliation on $X$ such that $\det\!\scrG$ is a line bundle. Suppose that $\scrG$ is weakly regular. Then $\scrG$ has canonical singularities.
\end{lemme}

For foliations with numerically trivial canonical class, the converse of {\hyperref[lemma_foliation-weak-reg-canonical]{Lemma \ref*{lemma_foliation-weak-reg-canonical}}} also holds:
\begin{lemme}[{\cite[Corollary 5.23]{Druel18b}}]
\label{lemma_num-triv-canonical-weak-reg}
Let $X$ be a normal complex variety admitting an effective $\QQ$-divisor $\Delta$ such that $(X,\Delta)$ is klt, and let $\scrG$ be a foliation on $X$ with canonical singularities. Suppose that $\det\!\scrG$ is a line bundle and is numerically trivial, then $\scrG$ is weakly regular and there is a decomposition $T_X\simeq \scrG\oplus \scrE$ of $T_X$ into involutive subsheaves. 
\end{lemme}

\begin{rmq}
\label{rmk_lemma_num-triv-canonical-weak-reg}
Let us remark that {\hyperref[lemma_num-triv-canonical-weak-reg]{Lemma \ref*{lemma_num-triv-canonical-weak-reg}}} is a key ingredient in the proof of \cite[Theorem 1.5]{Druel18b}. In fact, let $X$ be a klt projective variety and let $\scrG$ be an algebraically integrable foliation on $X$ with numerically trivial canonical class, let us briefly explain the strategy of the proof of \cite[Theorem 1.5]{Druel18b}: First by {\hyperref[lemma_num-triv-canonical-weak-reg]{Lemma \ref*{lemma_num-triv-canonical-weak-reg}}} $\scrG$ is weakly regular, hence by {\hyperref[thm_weak-reg-foliation]{Theorem \ref*{thm_weak-reg-foliation}}} (\cite[Theorem 6.1]{Druel18b}), up to replacing $X$ by a $\QQ$-factorialization one can assume that $\scrG$ is induced by an equidimensional fibre space. Then by separating the Abelian variety factor we can reduce the proof to the case that the leaf of $\scrG$ has vanishing irregularity and then \cite[Lemma 4.6]{Druel18a} permits to conclude.
\end{rmq}

\begin{rmq}
\label{rmk_weak-reg-canonical}
In \cite{Druel18b} the above two lemmas are stated for normal variety $X$ with klt singularities. But since the control on the singularities of $X$ is only used to ensure the existence and the universal property of the pullback maps of reflexive differentials (\cite[\S 2.6]{Druel18b}) and since this in fact holds for any "klt space" in the sense of Kebekus (that is, a normal complex variety $X$ admitting an effective $\QQ$-divisor $\Delta$ such that the pair $(X,\Delta)$ is klt) by \cite[Theorem 3.1, Proposition 6.1]{Keb13}, we see immediately that the two lemmas holds for klt spaces.    
\end{rmq}

Now let us recall the following important characterization of having canonical singularities for foliations with numerically trivial canonical class over {\it projective} varieties in terms of uniruledness, which first appears in \cite[Corollary 3.8]{LPT18} for $X$ smooth and is generalized to singular case in \cite[Proposition 4.22]{Druel18b}:
\begin{prop}[{\cite[Proposition 4.22]{Druel18b}}]
\label{prop_foliation-num-triv-canonical}
Let $X$ be a normal projective variety and let $\scrG$ be a $\QQ$-Gorenstein foliation on $X$ such that $K_{\scrG}\equiv 0$. Then $\scrG$ has canonical singularities if and only if $\scrG$ is not uniruled. 
\end{prop}
Recall that a foliation $\scrG$ on the normal variety $X$ is called {\it uniruled} if through a general point of $X$ there is a rational curve which is everywhere tangent to $\scrG$. 

Let us turn to the proof of {\hyperref[thm_Druel_foliation-num-triv]{Theorem \ref*{thm_Druel_foliation-num-triv}}}. The proof is suggested to the author by St\'ephane Druel through personal communications (of course, any mistake is the author's), and is very similar to Step 2 of Proof of \cite[Theorem 8.1]{Druel18b}. The main idea is to take a $\QQ$-factorialization of $X$, which enables us to apply \cite[Theorem 1.5]{Druel18b}. In order to descend the splitting to $X$ we intend to use \cite[Lemma 4.6]{Druel18a}, to this end we need the following:
\begin{lemme}[{\cite[Proposition 8.2]{Druel18b}}]
\label{lemma_foliation-ab-var}
Let $X$ be a normal projective variety and let $\scrE$ be an algebraically integrable foliation with canonical singularities on $X$. Suppose that $\scrE\simeq\scrO_X^{\oplus\rank\scrE}$. Then there exist an Abelian variety $A$, a normal projective variety $V$ and a finite \'etale cover $f:A\times V\to X$ such that $f\inv\scrE\simeq\pr_1^\ast T_A$.  
\end{lemme}

Now we can prove {\hyperref[thm_Druel_foliation-num-triv]{Theorem \ref*{thm_Druel_foliation-num-triv}}}:
\begin{proof}[Proof of {\hyperref[thm_Druel_foliation-num-triv]{Theorem \ref*{thm_Druel_foliation-num-triv}}}]
If $\Delta=0$ this is nothing but \cite[Theorem 1.5]{Druel18b}. For the general case, let $\beta: X^{\qf}\to X$ be a $\QQ$-factorialization of $X$, whose existence is proved by \cite[Corollary 1.37, pp.~29-30]{Kollar13}, and let $\scrG^{\qf}:=\beta\inv\!\scrG$. By construction, $\beta$ is a small birational morphism, then 
\[
K_{X^{\qf}}+\beta\inv_\ast\!\Delta\sim_{\QQ}\beta^\ast(K_X+\Delta),
\] 
so that $(X^{\qf},\beta_\ast\inv\!\Delta)$ remains a klt pair, but $X^{\qf}$ is $\QQ$-factorial hence $X^{\qf}$ itself is klt by \cite[Corollary 2.35(3), pp.~57-58]{KM98}. Moreover, since $\beta$ is small birational, we have 
\[
K_{\scrG^{\qf}}\sim_{\QQ}\beta^\ast K_{\scrG}\equiv 0,
\] 
hence by \cite[Lemma 4.2(2)]{Druel18b} $\scrG^{\qf}$ also has canonical singularities. Then we can apply \cite[Theorem 1.5]{Druel18b} to $(X^{\qf},\scrG^{\qf})$ to obtain projective varieties $Z^{\qf}$ and $F^{\qf}$ with klt singularities and a quasi-\'etale cover $g^{\qf}:Z^{\qf}\times F^{\qf}\to X^{\qf}$ such that $(g^{\qf})\inv\!\scrG^{\qf}\simeq\pr_1^\ast T_{Z^{\qf}}$. And we have
\[
\pr_1^\ast K_{Z^{\qf}}\sim(g^{\qf})^\ast K_{\scrG^{\qf}}\equiv 0,
\]
implying that $K_{Z^{\qf}}\equiv0$ ($\pr_1^\ast$ is an injective morphism between Picard groups). 
By \cite[1.5.Theorem]{HP19}, up to a quasi-\'etale cover, we can assume that $Z^{\qf}\simeq A^{\qf}\times B^{\qf}$ with $A^{\qf}$ being an Abelian variety and $B^{\qf}$ a normal projective variety with trivial canonical class and vanishing augmented irregularity. Now let $X_1$ be the normalization of $X$ in the function field of $Z^{\qf}\times F^{\qf}$, and let $\beta_1: Z^{\qf}\times F^{\qf}\to X_1$ and $g:X_1\to X$ be the induced morphism. Set $\Delta_1:=g^\ast\Delta$ be the pullback of $\Delta$ as Weil divisor (c.f. \cite[Construction 2.13]{CKT16}), then $(X_1,\Delta_1)$ is klt by \cite[Proposition 5.20, p.~160]{KM98}. We have the following commutative diagram
\begin{center}
\begin{tikzpicture}[scale=2.0]
\node (A) at (0,0) {$X$};
\node (A1) at (0,1) {$X_1$};
\node (B) at (-2,0) {$X^{\qf}$};
\node (B1) at (-2,1) {$Z^{\qf}\times F^{\qf}$};
\path[->, font=\scriptsize,>=angle 90]
(A1) edge node[right]{$g$, quasi-\'etale} (A)
(B1) edge node[left]{$g^{\qf}$, quasi-\'etale} (B)
(B) edge node[below]{$\beta$, small birational} (A)
(B1) edge node[above]{$\beta_1$, small birational} (A1);
\end{tikzpicture}
\end{center}
Then $\pr_1^\ast T_{A^{\qf}}$ is a direct summand of $T_{A^{\qf}\times B^{\qf}\times F^{\qf}}\simeq T_{Z^{\qf}\times F^{\qf}}$, and pushes down via $\beta_1$ to an algebraically integrable foliation $\scrE_1$ on $X_1$. Similarly, $\pr_2^\ast T_{B^{\qf}}$ induces an algebraically integrable foliation $\scrG_1$ on $X_1$. By construction $\scrE_1\oplus\scrG_1\simeq g\inv\!\scrG$ and $\scrE_1\simeq\scrO_{X_1}^{\oplus\rank\scrE_1}$. Since $\scrE_1$ is a direct summand of $T_{X_1}$, $\scrE_1$ is weakly regular (c.f. \cite[Lemma 5.8]{Druel18b}), and thus has canonical singularities by {\hyperref[lemma_foliation-weak-reg-canonical]{Lemma \ref*{lemma_foliation-weak-reg-canonical}}}. By applying {\hyperref[lemma_foliation-ab-var]{Lemma \ref*{lemma_foliation-ab-var}}} to $\scrE_1$ we see that there exist an Abelian variety $A_1$, a normal projective variety $X_2$ and a finite \'etale cover $g_1: A_1\times X_2\to X_1$ such that $g_1\inv\!\scrE_1\simeq\pr_1^\ast T_{A_1}$. Since $g_1$ is a finite \'etale cover, $(A_1\times X_2, g_1^\ast\Delta_1)$ is klt, and hence for general $a\in A_1$, the pair $(X_2,(g_1^\ast\Delta_1)|_{\pr_1\inv(a)})$ is klt (by identifying $X_2$ with $\pr_1\inv(a)$) by \cite[Lemma 5.17, pp.~158-159]{KM98}. Since $g_1$ is a finite \'etale cover, we have 
\[
g_1\inv\scrE_1\oplus g_1\inv\scrG_1\simeq g_1\inv\!g\inv\!\scrG,
\]
hence $g_1\inv\!\scrG_1$ is a direct summand of $\pr_2^\ast T_{X_2}$. In consequence, $g_1\inv\!\scrG_1$ descends to a a(n) (algebraically integrable) foliation $\scrG_2$ on $X_2$ via $\pr_2$, i.e. there is a foliation $\scrG_2$ on $X_2$ such that $\pr_2\inv\scrG_2\simeq g_1\inv\!\scrG_1$. Moreover, by construction $\scrG_2$ is a direct summand of $T_{X_2}$, hence $\scrG_2$ is weakly regular. 

By construction we have 
\[
\beta_1^\ast K_{\scrG_1}\sim \beta_1(K_{\scrE_1}+K_{\scrG_1})\sim\beta_1^\ast K_{g\inv\!\scrG}\sim K_{Z^{\qf}}\sim 0,
\]
hence $K_{\scrG_1}\sim 0$, which implies that $K_{\scrG_2}\sim 0$ and in particular $K_{\scrG_2}$ is a Cartier divisor. By {\hyperref[lemma_foliation-weak-reg-canonical]{Lemma \ref*{lemma_foliation-weak-reg-canonical}}}, $\scrG_2$ has canonical singularities. Clearly, in order to prove the theorem for $X$ and $\scrG$, it suffices to prove this for $X_2$ and $\scrG_2$. If $\dim A^{\qf}=0$, then $Z^{\qf}\simeq B^{\qf}$ has vanishing augmented irregularity, in this case \cite[Lemma 4.6]{Druel18a} permits us to conclude; otherwise, we have 
\[
\dim X_2=\dim X-\dim A_1=\dim X-\rank\scrE_1=\dim X-\dim A^{\qf}<\dim X,
\]
then since $X_2$ admits an effective divisor $\Delta_2$ such that $(X_2,\Delta_2)$ is klt, the proof is done by an induction on the dimension. 
\end{proof}

Next let us give an alternative proof of {\hyperref[mainthm_MRC]{Theorem \ref*{mainthm_MRC}}} by using {\hyperref[thm_Druel_foliation-num-triv]{Theorem \ref*{thm_Druel_foliation-num-triv}}}:
\begin{proof}[Alternative proof of {\hyperref[mainthm_MRC]{Theorem \ref*{mainthm_MRC}}}]
Let everything as in the {\hyperref[general-setting]{General Setting \ref*{general-setting}}} with $\psi:X\dashrightarrow Y$ being the MRC fibration of $X$. By {\hyperref[thm_decomp-tangent]{Theorem \ref*{thm_decomp-tangent}}} the tangent sheaf admits a splitting $T_X\simeq\scrF\oplus\scrG$ into algebraically integrable foliations with $K_{\scrG}\sim 0$. Set $F$ (resp. $Z$) to be the Zariski closure of the general leaf of $\scrF$ (resp. $\scrG$), then $F$ is rationally connected. By {\hyperref[lemma_open-product]{Lemma \ref*{lemma_open-product}}}, $Y_0$ can be regarded as a Zariski open of $Z$, hence $Z$ is birational to $Y$; but $\psi$ is the MRC fibration of $X$, $Y$ is not uniruled, then so is $Z$. This means that $\scrG$ is not uniruled, and by {\hyperref[prop_foliation-num-triv-canonical]{Proposition \ref*{prop_foliation-num-triv-canonical}}}, $\scrG$ has canonical singularities. By {\hyperref[thm_Druel_foliation-num-triv]{Theorem \ref*{thm_Druel_foliation-num-triv}}} there are projective varieties $Z_1$ and $F_1$ and a quasi-\'etale cover $f:Z_1\times F_1\to X$ such that $f\inv\!\scrG\simeq\pr_1^\ast T_{Z_1}$. Since $\pi_1(X_{\reg})\simeq\{1\}$, $f$ must be an isomorphism, then we have $\pr_1^\ast T_{Z_1}\simeq\scrG$ and $\pr_1^\ast T_{F_1}\simeq\scrF$. In particular, we have $Z_1\simeq Z$ and $F\simeq F_1$, hence $X\simeq Z\times F$ with $K_Z\sim 0$ and $F$ rationally connected.  
\end{proof}

\section{Fundamental group of \texorpdfstring{$X_{\reg}$}{text}}
\label{sec_fundamental-group}

Let $X$ be a klt projective variety with nef anticanonical divisor $-K_X$. In this section we study the fundamental group of $X_{\reg}$, especially the relation of $\pi_1(X_{\reg})$ to the decomposition theorem and to other folklore conjectures (c.f. {\hyperref[conj_pi1-Fano-CY]{Conjecture \ref*{conj_pi1-Fano-CY}}}).

\subsection{Albanese map of \texorpdfstring{$X_{\reg}$}{text} and nilpotent completion of \texorpdfstring{$\pi_1(X_{\reg})$}{text}}
\label{ss_fundamental-group_q-Albanese}

In this subsection we will study the Albanese map of $X_{\reg}$ and deduce from this the nilpotent completion of $\pi_1(X_{\reg})$ by using the same argument as in \cite[\S 2]{Cam95}. The principal result of this subsection is the following:
\begin{thm}
\label{thm_q-Alb-pi1-anti-nef}
Let $X$ be a normal projective variety of semi-Fano type, i.e. there is an effective $\QQ$-divisor $\Delta$ on $X$ such that $(X,\Delta)$ is klt and that the twisted anticanonical divisor $-(K_X+\Delta)$ is nef. Then 
\begin{itemize}
\item[\rm(a)] The Albanese map $\widetilde{\alb}_{X_{\reg}}:X_{\reg}\to\widetilde{\Alb}_{X_{\reg}}$ of $X_{\reg}$ is dominant. 
\item[\rm(b)] Let $j:X_{\reg}\hookrightarrow X$ be the open immersion. Then the morphism between fundamental groups induced by $\alb_X\circ j$ gives rise to an isomorphism
\[
\pi_1(X_{\reg})^{\nilpotent}\xrightarrow{\simeq}\pi_1(\Alb_X).
\]  
\end{itemize}
\end{thm}

Before turning to the proof of the theorem, let us first recall the definition of the nilpotent completion of a group (c.f. \cite[Appendice A]{Cam95}). Let $G$ be a group, define the descending central series of $G$ by $G_1:=G$ and $G_{k+1}=[G,G_k]$ for any $k\in\ZZ_{>0}$ and set 
\[
G_{\infty}:=\bigcap_{k\in\ZZ_{>0}}G_k. 
\]
Put
\[
G'_k=\sqrt{G_k}:=\left\{\,g\in G\,\big|\,g^m\in G_k\text{ for some }m\in\ZZ_{>0}\,\right\}.
\]
for $1\leqslant k\leqslant \infty$. Then the {\it torsion-free nilpotent completion} of $G$ is defined to be 
\[
G^{\nilpotent}:=G/G'_{\infty}\,.
\]
Let $f:G\to H$ be a group morphism, \cite[3.4.Theorem]{Sta65} gives the following criterion for the induced morphism between nilpotent completion to be injective or isomorphism (c.f. also \cite[A.2.Th\'eor\`eme]{Cam95}): 
\begin{prop}[{\cite[3.4.Theorem]{Sta65}}]
\label{prop_completion-nilpotent}
Let $f: G\to H$ be a group morphism, and for $1\leqslant k\leqslant \infty$ let $G'_k$ (resp. $H'_k$) be the radical of the $k$-th member in the descending central series of $G$ (resp. of $H$), as defined above. Suppose that the induced morphism $\Coh_i(f): \Coh_i(G,\QQ)\to\Coh_i(H,\QQ)$ is an isomorphism for $i=1$ and surjective for $i=2$. Then the morphism $f'_k: G'_k\to H'_k$ induced by $f$ is injective for every $1\leqslant k\leqslant\infty$, and is of finite index if $k<\infty$. Moreover, if $f$ is surjective then $f'_k$ is an isomorphism for every $1\leqslant k\leqslant\infty$. 
\end{prop}

Now let us turn to the proof of {\hyperref[thm_q-Alb-pi1-anti-nef]{Theorem \ref*{thm_q-Alb-pi1-anti-nef}}}: 
\begin{proof}[Proof of {\hyperref[thm_q-Alb-pi1-anti-nef]{Theorem \ref*{thm_q-Alb-pi1-anti-nef}}}]
Let us first prove (a), i.e. the Albanese map of $X_{\reg}$ is dominant. Let $Y^\circ$ be the Zariski closure in $\widetilde{\Alb}_{X_{\reg}}$ of the image of $\widetilde{\alb}_{X_{\reg}}$ and let $Y$ be a smooth compactification of $Y^\circ$ such that $D_Y:=Y\backslash Y^\circ$ is a SNC divisor, then we get a dominant rational map $\psi: X\dashrightarrow Y$. Take $M$ to be a strong desingularization of the graph of $\psi$, then the induced morphism $\pi: M\to X$ is a birational morphism which is an isomorphism over $X_{\reg}$. Let $E=\Exceptional(\pi)$ be the exceptional divisor of $\pi$ and let $\phi:M\to Y$ be the natural morphism, then by construction $M\backslash E\simeq X_{\reg}$ and thus $\Supp(\phi^\ast D_Y)\subseteq E$. Now we are in the same situation as in {\hyperref[general-setting]{General Setting \ref*{general-setting}}}, hence by the proof of {\hyperref[prop_bir-geometry-psi]{Proposition \ref*{prop_bir-geometry-psi}(a)}}, for a very ample line bundle $L$ on $X$ and for general members $H_1\,,\cdots,H_{\dim\!X-1}$ in the linear series $\left|\pi^\ast L\right|$, we have 
\[
\phi^\ast K_Y\cdot C\leqslant 0
\]
where $C:=H_1\cap\cdots\cap H_{\dim\!X-1}$. Since $\phi^\ast D_Y$ is $\pi$-exceptional, we have $\phi^\ast D_Y\cdot C=0$ hence by the projection formula we get
\[
(K_Y+D_Y)\cdot C_Y\leqslant 0
\]
where $C_Y:=\phi_\ast C$. By {\hyperref[prop_subvar-q-Ab]{Proposition \ref*{prop_subvar-q-Ab}}} we know that $\bar\kappa(Y^\circ):=\kappa(Y,K_Y+D_Y)\geqslant 0$ (c.f. {\hyperref[defn_log-Kod-dim]{Definition \ref*{defn_log-Kod-dim}}} for the definition of logarithmic Kodaira dimension), hence we must have
\[
(K_Y+D_Y)\cdot C_Y=0,
\]
but by construction $C_Y$ is moves in a strong connecting family of curves (c.f. \cite[\S 0]{BDPP13}) on $Y$, hence by \cite[0.5.Theorem]{BDPP13} the numerical dimension $\nu(Y,K_Y+D_Y)=0$, this implies that $\kappa(Y,K_Y+D_Y)\leqslant\nu(Y,K_Y+D_Y)=0$. Therefore we must have $\bar\kappa(Y^\circ)=\kappa(Y,K_Y+D_Y)=0$ . Again by {\hyperref[prop_subvar-q-Ab]{Proposition \ref*{prop_subvar-q-Ab}}} we have that $Y^\circ$ is a semi-Abelian subvariety of $\widetilde{\Alb}_{X_{\reg}}$; but by {\hyperref[prop_q-Albanese-image]{Propositionn \ref*{prop_q-Albanese-image}}} $Y^\circ$ generates $\widetilde{\Alb}_{X_{\reg}}$, hence we must have $Y^\circ=\widetilde{\Alb}_{X_{\reg}}$, and this proves (a).

Now let us prove (b). It can be deduced by \cite[2.2.Th\'eor\`eme]{Cam95} and by the more general {\hyperref[thm_klt-H2-pi1]{Theorem \ref*{thm_klt-H2-pi1}}} below. This theorem, as well as its proof,
is pointed out to the author by Benoît Claudon (any mistake, is of course, the author's).
\end{proof}

\begin{thm}
\label{thm_klt-H2-pi1}
Let $X$ be a normal projective variety which admits an effective $\QQ$-divisor $\Delta$ such that the pair $(X,\Delta)$ is klt and let $j:X_{\reg}\hookrightarrow X$ be the open immersion. Then
\[
\Coh^1(j,\CC):\Coh^1(X,\CC)\to\Coh^1(X_{\reg},\CC)
\]
is an isomorphism and 
\[
\Coh^2(j,\CC):\Coh^2(X,\CC)\to\Coh^2(X_{\reg},\CC)
\]
is injective. In particular, $j$ induces an isomorphism between the nilpotent completion of fundamental groups
\[
\pi_1(X_{\reg})^{\nilpotent}\xrightarrow{\simeq}\pi_1(X)^{\nilpotent}.
\]
\end{thm}



\begin{proof}
The klt condition is used to guarantee the vanishing of $\RDer^1\!j_\ast\CC$. In fact, by \cite[Theorem 1]{Braun20}, for any point $x\in X$, there is an open neighbourhood $U$ of $x$ such that $\pi_1(U_{\reg})$ is finite, in particular $\Coh^1(U_{\reg},\CC)=0$, hence we get $\RDer^1\!j_\ast\CC=0$. Then consider the Leray spectral sequence associated to $j$ which gives the exact sequence 
\[
0\to \Coh^1(X,\CC)\to\Coh^1(X_{\reg},\CC)\to\Coh^0(X,\RDer^1\!j_\ast\CC)\to\Coh^2(X,\CC)\to\Coh^2(X_{\reg},\CC),
\]
by the vanishing of $\RDer^1\!j_\ast\CC$ we 
get the isomorphism $\Coh^1(X,\CC)\simeq\Coh^1(X_{\reg},\CC)$ and the injectivity of $\Coh^2(X,\CC)\to\Coh^2(X_{\reg},\CC)$. 

It remains to show that $\pi_1(X_{\reg})^{\nilpotent}\simeq\pi_1(X)^{\nilpotent}$. By \cite[\S 0.7 (B), p.~33]{FL81} the fundamental group morphism $\pi_1(j):\pi_1(X_{\reg})\to\pi_1(X)$ is surjective, hence by {\hyperref[prop_completion-nilpotent]{Proposition \ref*{prop_completion-nilpotent}}} it suffices to show that 
\[
\Coh_1(\pi_1(j),\QQ):\Coh_1(\pi_1(X_{\reg}),\QQ)\to\Coh_1(\pi_1(X),\QQ)
\]
is an isomorphism and
\[
\Coh_2(\pi_1(j),\QQ):\Coh_2(\pi_1(X_{\reg}),\QQ)\to\Coh_2(\pi_1(X),\QQ)
\]
is surjective. We have shown that $\Coh^1(X,\CC)\simeq\Coh^1(X_{\reg},\CC)$, hence $\Coh_1(\pi_1(j),\QQ)$ is an isomorphism by \cite[\S 5]{Sta65}. On the other hand, the surjectivity $\Coh_2(\pi_1(j),\QQ)$ can be deduced from   \cite[2.3.Lemma]{Cam95} and from the injectivity of $\Coh^2(X,\CC)\to\Coh^2(X_{\reg},\CC)$. 
\end{proof}

\subsection{From fundamental group to decomposition theorem}
\label{ss_fundamental-group_decomposition}

In this subsection, we show that with the help of {\hyperref[mainthm_Albanese]{Theorem \ref*{mainthm_Albanese}}} and {\hyperref[mainthm_MRC]{Theorem \ref*{mainthm_MRC}}} the proof of {\hyperref[conj_klt-anti-nef]{Conjecture \ref*{conj_klt-anti-nef}}} can be reduced to the study of the fundamental group of $X$. Precisely speaking, we will prove that {\hyperref[conj_pi1-anti-nef]{Conjecture \ref*{conj_pi1-anti-nef}}} implies {\hyperref[conj_klt-anti-nef]{Conjecture \ref*{conj_klt-anti-nef}}}. Let us remark that when $X$ is smooth, the {\hyperref[conj_pi1-anti-nef]{Conjecture \ref*{conj_pi1-anti-nef}}} is proved by M.P\u aun in \cite{Paun97} by improving the arguments in the previous work of \cite[\S 1]{DPS93} and by applying the famous theorem of Cheeger-Colding \cite[Theorem 8.7]{CC96}. 

\begin{thm}
\label{thm_conj-pi1-klt-anti-nef}
Let $X$ be a normal projective variety of semi-Fano type. Suppose that {\hyperref[conj_pi1-anti-nef]{Conjecture \ref*{conj_pi1-anti-nef}}} holds for $X$, i.e. $\pi_1(X_{\reg})$ is of polynomial growth, then (the log version of) {\hyperref[conj_klt-anti-nef]{Conjecture \ref*{conj_klt-anti-nef}}} holds for $X$, i.e. up to replacing $X$ by a finite quasi-étale cover, the universal cover $\tilde X$ of $X$ can be decomposed into a product 
\begin{equation}
\label{eq_decomp-univ-cover}
\tilde X\simeq \CC^q\times Z\times F,
\end{equation}
with $q$ being the augmented irregularity of $X$, $Z$ being a klt projective variety with trivial canonical divisor and $F$ being rationally connected. 
\end{thm}

\begin{proof}
By \cite[Proposition 5.20, pp.~160-161]{KM98}, any quasi-étale cover of $X$ is still of semi-Fano type. By hypothesis $\pi_1(X_{\reg})$ is of polynomial growth (by \cite[\S 0.7 (B), p.~33]{FL81} so is $\pi_1(X)$), hence by \cite[Main Theorem]{Gro81} $\pi_1(X_{\reg})$ is virtually nilpotent, therefore, up to replacing $X$ by a finite \'etale cover we can assume that $\pi_1(X_{\reg})$ is torsion free nilpotent.  

By {\hyperref[mainthm_Albanese]{Theorem \ref*{mainthm_Albanese}}}, the Albanese map $\alb_X: X\to\Alb_X$ is a locally constant fibration. Let $V$ denotes the fibre of $\alb_X$, then $\alb_X|_{X_{\reg}}:X_{\reg}\to\Alb_X$ is a locally trivial fibration whose fibre is isomorphic to $V_{\reg}$. Apply \cite[\S 17, p.~209]{BT82} to $\alb_X|_{X_{\reg}}$ (viewed as a topological fibre bundle) we get a homotopy sequence
\[
\cdots\to\pi_2(\Alb_X)\to\pi_1(V_{\reg})\to\pi_1(X_{\reg})\to\pi_1(\Alb_X)\to 1.
\]
But $\pi_1(X_{\reg})$ is torsion-free nilpotent, by {\hyperref[thm_q-Alb-pi1-anti-nef]{Theorem \ref*{thm_q-Alb-pi1-anti-nef}}} the morphism $\pi_1(X_{\reg})\to\pi_1(\Alb_X)$ is an isomorphism. Moreover, since $\Alb_X$ is an Abelian variety, we have $\pi_2(\Alb_X)\simeq\{0\}$, hence $\pi_1(V_{\reg})\simeq\{1\}$. Then the conclusion follows from {\hyperref[mainthm_MRC]{Theorem \ref*{mainthm_MRC}}}. 
\end{proof}

\subsection{From \texorpdfstring{{\hyperref[conj_pi1-anti-nef]{Conjecture \ref*{conj_pi1-anti-nef}}}}{text} to \texorpdfstring{{\hyperref[conj_pi1-Fano-CY]{Conjecture \ref*{conj_pi1-Fano-CY}}}}{text}}
\label{ss_fundamental-group_CY-Fano}

In this subsection we will show that {\hyperref[conj_pi1-anti-nef]{Conjecture \ref*{conj_pi1-anti-nef}}} implies the Gurjar-Zhang conjecture on the finiteness of the fundamental group of the smooth locus of varieties of Fano type and {\hyperref[conj_pi1-Fano-CY]{Conjecture \ref*{conj_pi1-Fano-CY}}}. In fact, we can prove the following more general result:
\begin{prop}
\label{prop_conj-pi1-anti-nef-Fano-CY}
Let $X$ be a normal projective variety of semi-Fano type with vanishing augmented irregularity. Suppose that $\pi_1(X_{\reg})$ is of polynomial growth, then $\pi_1(X_{\reg})$ is finite.
\end{prop}
\begin{proof}
First note that, as in the proof of {\hyperref[thm_conj-pi1-klt-anti-nef]{Theorem \ref*{thm_conj-pi1-klt-anti-nef}}}, in order to prove the finiteness of $\pi_1(X_{\reg})$ we can replace $X$ by any finite quasi-\'etale cover; in particular we can assume that $\pi_1(X_{\reg})$ is a torsion-free nilpotent group (by \cite[Main Theorem]{Gro81}), so that we have $\pi_1(X_{\reg})\simeq\pi_1(\Alb_X)$. But the augmented irregularity of $X$ is zero, its Albanese variety $\Alb_X$ is trivial, then a fortiori $\pi_1(X_{\reg})\simeq\{1\}$, in particular $\pi_1(X_{\reg})$ is finite. Thus we prove the proposition.
\end{proof}

By the proposition above, we see that {\hyperref[conj_pi1-anti-nef]{Conjecture \ref*{conj_pi1-anti-nef}}} implies {\hyperref[conj_pi1-Fano-CY]{Conjectures \ref*{conj_pi1-Fano-CY}}}; moreover, since varieties of Fano type have vanishing augmented irregularity (every quasi-\'etale cover of a projective variety of Fano type remains Fano type; by \cite[Corollary 1.1]{Zhang06} and \cite{Tak00} varieties of Fano type are simply connected), {\hyperref[conj_pi1-anti-nef]{Conjecture \ref*{conj_pi1-anti-nef}}} implies the Gurjar-Zhang conjecture which states that for any projective variety of Fano type $X$ the fundamental group of $X_{\reg}$ is finite and which has recently been confirmed in \cite{Braun20}. 

Finally, let us make some remarks on the history of the Gurjar-Zhang conjecture and {\hyperref[conj_pi1-Fano-CY]{Conjecture \ref*{conj_pi1-Fano-CY}}}:
\begin{rmq}
\label{rmk_conj_pi1-Fano-CY}
The Gurjar-Zhang conjecture is first proved for del Pezzo surfaces in \cite{GZ94,GZ95} (c.f. \cite[last Remark]{GZ95} for weak Fano surfaces) and the question is explicitly raised in \cite[Introduction]{Zhang95} for log Fano varieties (c.f. also \cite[Question 0.11]{Sch07}) and in \cite{Zhang95} the conjecture is proved for canonical (klt) Fano threefolds under some additional assumption that $X$ has isolated singularities (\cite[Theorem 1]{Zhang95}) or that the index of $X$ is $\geqslant\dim\!X-2$ (\cite[Theorem 2]{Zhang95}). The three-dimensional Fano case is fully confirmed by \cite[Theorem 1.6]{TX17}. Then it is proved in \cite[Theorem 2]{Xu14} and \cite[Theorem 1.13]{GKP16a} that the profinite completion of $\pi_1(X_{\reg})$ (which is, isomorphic to the \'etale fundamental group of $X_{\reg}$) is finite for $X$ weak log Fano. Recently this conjecture has been settled in \cite{Braun20}.

As for {\hyperref[conj_pi1-Fano-CY]{Conjecture \ref*{conj_pi1-Fano-CY}}}, the question is raised in \cite{GGK19} and it is proved therein that for $X$ klt projective with trivial canonical divisor and vanishing augmented irregularity the fundamental group of $X_{\reg}$ has only finitely many $k$-dimensional complex representations for every $k\in\ZZ_{>0}$, and that the image of each finite dimensional representation of $\pi_1(X_{\reg})$ is finite. It is also proved that the \'etale fundamental group of $X_{\reg}$ is finite for $X$ an irreducible holomorphic symplectic variety or an even-dimensional Calabi-Yau varieties, c.f. \cite[\S 13.1]{GGK19}.
\end{rmq}

\begin{appendices}
\section{Albanese map of quasi-projective varieties}
\label{sec_Q-Albanese}
In this section, we recall some general results about the Albanese maps of smooth quasi-projective varieties. Our main reference is \cite[\S 3]{Fuj15}, c.f. also \cite[\S 5]{Kaw81}. First recall the definition of semi-Abelian varieties. Let us remark that they are called "quasi-Abelian varieties" by Iitaka and  \cite{Kaw81,Fuj15} (which is different from the notion of "quasi-Abelian varieties" in \cite{AK01}); we choose to use the name "semi-Abelian variety", which seems to be more commonly used in algebraic geometry.

\begin{defn}[{\cite[\S 5, Definition, p.~271]{Kaw81}; \cite[Defniition 2.8]{Fuj15}}]
\label{defn_q-ab-var}
Let $G$ be a connected algebraic group and let 
\[
1\to H\to G\to A\to 1
\]
be the Chevalley decomposition (c.f. \cite[Theorem 1.1]{Con02}) of $G$, where $H$ is a linear algebraic group and $A$ is an Abelian variety. $G$ is called a {\it semi-Abelian variety} 
if $H\simeq\GG_m^{\dim\!H}$ where $\GG_m$ denotes the multiplicative group $\CC^{\ast}$. 
\end{defn}

We collect some elementary properties of semi-Abelian varieties as following:
\begin{prop}[{\cite[Lemma 2.11, Lemma 2.13]{Fuj15}}]
\label{prop_property-q-Ab-var}
Let $G$ be a semi-Abelian variety and let 
\[
1\to \GG_m^d\to G\to A\to 1
\]
be its Chevalley decomposition with $A$ an Abelian variety. Then 
\begin{itemize}
\item[\rm(a)] $G$ is a principal $\GG_m^d$-bundle over $A$\,;
\item[\rm(b)] $G$ is a commutative group\,;
\item[\rm(c)] the universal cover of $G$ is $\CC^{\dim\!G}$ and $G\simeq\CC^{\dim\!G}/\pi_1(G)$ with $\pi_1(G)$ viewed as a lattice in $\CC^{\dim\!G}$.
\end{itemize}
\end{prop}

Analogous to the case of Abelian varieties (or even complex tori, c.f. \cite[Lemma 10.1, Theorem 10.3, pp.~116-119]{Ueno75}), the closed subvarieties of semi-Abelian varieties have the following rigidity property:
\begin{prop}[{\cite[Theorem 4.4]{Fuj15}}]
\label{prop_subvar-q-Ab}
Let $G$ be a semi-Abelian variety and let $W$ be a closed subvariety of $G$. Then the logarithmic Kodaira dimension $\bar\kappa(W)\geqslant 0$ and $\bar\kappa(W)=0$ if and only if it is a translate of a semi-Abelian subvariety of $G$.  
\end{prop}
Let us recall the notion of logarithmic Kodaira dimension as mentioned in the proposition above (c.f. \cite[\S 5, Definition, p.~271]{Kaw81} and \cite[Definition 2.2]{Fuj15}):
\begin{defn}
\label{defn_log-Kod-dim}
Let $V^\circ$ be a smooth quasi-projective variety (or more generally an algebraic variety), and take $V$ to be a smooth compactification of $V^\circ$ such that $D_V:=V\backslash V^\circ$ is a (reduced) SNC divisor (the existence of such $V$ is ensured by Nagata's compactification theorem, Chow's lemma and Hironaka's resolution of singularities). Then the {\it logarithmic Kodaira dimension} of $V^\circ$, denoted by $\bar\kappa(V^\circ)$ is defined to be the Iitaka-Kodaira dimension of $K_V+D_V$, that is
\[
\bar\kappa(V^\circ):=\kappa(V,K_V+D_V).
\]
\end{defn}

Now let us turn to the Albanese maps of smooth quasi-projective varieties:
\begin{prop-defn}[{\cite[Theorem 3.16]{Fuj15}}]
\label{defn-prop_q-Albanese}
Let $U$ be a smooth quasi-projective variety and let $u$ be a fixed point of $U$. Then there is a semi-Abelian variety $\widetilde{\Alb}_U$ and an algebraic morphism $\widetilde{\alb}_U: U\to\widetilde{\Alb}_U$ such that $\widetilde{\alb}_U(u)=0$ and that for any algebraic morphism $\alpha: U\to G$ to a semi-Abelian variety $G$ satisfying $\alpha(u)=0_G$, there is a unique morphism of algebraic groups $f:\widetilde{\Alb}_U\to G$ such that $\alpha=f\circ\widetilde{\alb}_U$\,; and $\widetilde{\alb}_U$ is uniquely determined by this universal property. $\widetilde{\alb}_U$ is called the {\it Albanese map} of $U$ and $\widetilde{\Alb}_U$ is called the {\it Albanese variety} of $U$. Moreover, if $U$ is compact, then $\widetilde{\alb}_U$ coincide with the Albanese map of $U$. 
\end{prop-defn}
See \cite[\S 3]{Fuj15} for the construction of $\widetilde{\Alb}_U$ and $\widetilde{\alb}_U$ and be careful that in \cite{Kaw81,Fuj15} this is called the "quasi-Albanese map". Nevertheless, we call it simply the Albanese map, because this is the only reasonable one (there is no other way to define it and hence "quasi-" is a little bit redundant). Now let us recall some basic properties of the Albanese map.

\begin{prop}[{\cite[Lemma 3.11]{Fuj15}}]
\label{prop_q-Albanese-H1}
Let $U$ be a smooth quasi-projective variety and let $\widetilde{\alb}_U: U\to\widetilde{\Alb}_U$ be its Albanese map. Then the induced morphism
\[
(\widetilde{\alb}_U)_\ast: \Coh_1(U,\ZZ)\to\Coh_1(\widetilde{\Alb}_U,\ZZ)
\]
is surjective and the kernel of $(\widetilde{\alb}_U)_\ast$ is equal to the torsion part of $\Coh_1(U\,,\ZZ)$.
\end{prop}

\begin{prop}
\label{prop_q-Albanese-image}
Let $U$ be a smooth quasi-projective variety and let $\widetilde{\alb}_U:U\to\widetilde{\Alb}_U$ be its Albanese map. Take $V$ be a compactification of $U$ such that $V\backslash U$ is SNC divisor. Let
\[
1\to\GG_m^d\to\widetilde{\Alb}_U\xrightarrow{p} A_U\to 1
\]
be the Chevalley decomposition of $\widetilde{\Alb}_U$. Set $Z$ be the closure of $\Image(\widetilde{\alb}_U)$. 
Then
\begin{itemize}
\item[\rm(a)] $A_U$ is isomorphic to the Albanese variety $\Alb_V$ of $V$ such that $\alb_V|_U=p\circ\widetilde{\alb}_U$.
\item[\rm(b)] $Z$ generates $\widetilde{\Alb}_U$.
\end{itemize}
\end{prop}
\begin{proof}
(a) simply results from the construction in \cite[Lemma 3.7-3.8]{Fuj15}. As for (b), let $G$ be the algebraic subgroup of $\tilde{\Alb}_U$ generated by $Z$, and set $W$ be the image of $Z$ in $A_U$; then $W$ is the image of $\alb_V$, and by {\hyperref[prop_app-Alb]{Proposition \ref*{prop_app-Alb}(b)}} (\cite[Lemma 9.14, pp.~108-110]{Ueno75}) $W$ generates $A_U=\Alb_V$, hence the morphism $G\hookrightarrow \widetilde{\Alb}_U\to A_U$ is surjective, therefore we can write the Chevalley decomposition of $G$ as
\[
1\to H\to G\to A_U\to 1. 
\]
with $H\subseteq \GG_m^d$. Since $H$ is diagonalizable, by \cite[3.2.7.Corollary, p.~45]{Spr98} $H$ is a direct product of a finite Abelian group with an algebraic torus; but $G$ is connected then so is $H$, hence $H$ is an algebraic torus and thus by definition $G$ is a semi-Abelian variety. Then the morphism $U\to G$ satisfies the universal property of the Albanese map, hence a fortiori $G=\widetilde{\Alb}_U$.   
\end{proof}

\section{General account on singular foliations on normal varieties}
\label{sec_foliation}
In this appendix, we attempt to recollect some results coming from different literatures in order to give a somewhat general account on singular foliations on normal varieties, with some emphasis on the algebro-geometric aspect, for the convenience of the readers. Some results will be used in {\hyperref[sec_MRC]{\S \ref*{sec_MRC}}}.  

\subsection{General results on holomorphic foliations}
\label{ss_foliation_general}
First recall some definitions:
\begin{defn}
\label{defn_foliation}
Let $X$ be a normal complex variety and let $T_X:=(\Omega_X^1)^\ast$ denotes the tangent sheaf of $X$ (then it is a coherent reflexive sheaf on $X$). A (singular) foliation on $X$ is a subsheaf $\scrF$ of $T_X$ satisfying the following two conditions:
\begin{itemize}
\item[(i)] $[\scrF,\scrF]\subseteq\scrF$, i.e. $\scrF$ is stable under the Lie bracket (we call such $\scrF$ {\it involutive});
\item[(i\!i)] $\scrF$ is saturated in $T_X$, i.e. the quotient $T_X/\scrF$ is torsion free (which implies that $\scrF$ is reflexive\footnote{The reciprocal is not true, e.g. consider the natural inclusion $\scrO_{\AFF^1}(-D)\hookrightarrow\scrO_{\AFF^1}\simeq T_{\AFF^1}$ for $D$ effective divisor, $\scrO_{\AFF^1}(-D)$ is reflexive (locally free) but it is not saturated in $\scrO_{\AFF^1}$.}).
\end{itemize}
The {\it codimension} of $\scrF$ is defined to be $n-\rank\scrF$. The {\it normal sheaf} of $\scrF$ is defined to be $N_{\scrF}:=(T_X/\scrF)^{\ast\ast}$. A {\it leaf} $L$ of $\scrF$ is a maximal connected and immersed holomorphic submanifold of $X^\circ$ such that $T_L=\scrF|_L$ where $X^\circ$ denotes the Zariski open subset of $X_{\reg}$ on which $\scrF|_{X_{\reg}}$ is a subbundle of $T_{X_{\reg}}$ (by \cite[Corollary 5.5.15, p.~147]{Kob87}, $\codim(X\backslash X^\circ)\geqslant 2$). If $X$ is projective, the {\it canonical divisor} $K_{\scrF}$ of $\scrF$ is defined to be a Weil divisor on $X$ satisfying $\det\!\scrF\simeq\scrO_X(-K_{\scrF})$ (defined up to linear equivalence). 
\end{defn}

The following lemma says that the involutivity of a saturated subsheaf of $T_X$ can be checked over any Zariski open of $X$:
\begin{lemme}
\label{lemma_involutive}
Let $X$ be a normal complex variety and $\scrF$ a saturated subsheaf of $T_X$. Then $\scrF$ is involutive if and only if $\scrF|_{X_0}$ is involutive for some Zariski open $X_0\subseteq X$.
\end{lemme}
\begin{proof}
The "only if" part is obvious, we will prove the "if" part as following: First notice that the problem is local, hence we can assume $X$ is a Stein variety, so that every coherent sheaf on $X$ is globally generated (so is $\scrF$, and $\scrH$, $\scrH_1$ below). 
Consider the hom sheaves 
\[
\scrH:=\SheafHom[\scrO_X](\bigwedge^2\!\scrF,\scrF)\subseteq\scrH_1:=\SheafHom[\scrO_X](\bigwedge^2\!\scrF,T_X),
\]
since $\scrF$ and $T_X$ are reflexive, by the tensor-hom adjunction and \cite[Corollary 1.2]{Har80}, so are $\scrH$ and $\scrH_1$. Moreover, $\scrH_1/\scrH$ is contained in $\SheafHom[\scrO_X](\bigwedge^2\!\scrF,T_X/\scrF)$, which is torsion free since $T_X/\scrF$ is torsion free (again by the tensor-hom adjunction and \cite[Corollary 1.2]{Har80}), hence $\scrH_1/\scrH$ is also torsion free ($\scrH$ is saturated in $\scrH_1$). Now consider $\sigma=[\;\cdot\;,\,\cdot\;]\big|_{\bigwedge^2\!\scrF}$, regarded as a (global) section of  $\scrH_1$. Then the involutivity of $\scrF$ is equivalent to $\sigma\in\scrH$. This amounts to show that the image $\bar\sigma$ of $\sigma$ in $\scrH_1/\scrH$ is zero. By our assumption, $\bar\sigma|_{X_0}$ is zero; but $\scrH_1/\scrH$ is torsion free, a fortiori $\bar\sigma=0$, this completes the proof.  
\end{proof}

An important observation is that giving a foliation on $X$ is equivalent to giving a meromorphic differential form. In fact we have:
\begin{prop-defn}[c.f.{\cite[\S 3.1]{AD14}}]
\label{defn_foliation-diff-form}
Let $X$ be a normal complex variety of dimension $n$. Then we have the following two reciprocal constructions:
\begin{itemize}
\item Let $\scrF$ be a codimension $q$ foliation on $X$, then the surjection $T_X\to T_X/\scrF$ induces an inclusion $N_{\scrF}^\ast\hookrightarrow \Omega_X^{[1]}:=\left(\Omega_X^1\right)^{\ast\ast}$, which gives rise to a $\det\!N_{\scrF}$-valued reflexive $q$-form $\omega\in\Coh^0(X,\Omega_X^q[\otimes]\det\!N_{\scrF})$ (where `$[\otimes]$' denotes the reflexive hull of the tensor product), which satisfies the following three properties:
\begin{itemize}
\item[(a)] The vanishing locus of $\omega$ is of codimension $2$;
\item[(b)] $\omega$ is locally decomposable (around a general point of $X$), that is, in a neighbourhood of a general point of $X$, we can write $\omega=\omega_1\wedge\cdots\wedge\omega_q$ with $\omega_i$'s local $1$-forms;  
\item[(c)] $\omega$ is integrable, that is, for the local decomposition $\omega=\omega_1\wedge\cdots\wedge\omega_q$ as in (b), one has $d\omega_i\wedge\omega=0$ for every $i=1,\cdots,q$.
\end{itemize}
\item Let $\scrL$ be a reflexive sheaf of rank $1$ on $X$, and let $\omega\in\Coh^0(X,\Omega_X^q[\otimes]\scrL)$ satisfying the above three conditions (a)(b)(c)\footnote{The condition (a) is not essential in the construction. In fact, $\omega$ can also be regarded as a $\scrL(D)$-valued $q$-form for any effective (Weil) divisor $D$ on $X$, as section of $\Omega_X^q[\otimes]\scrL(D)$ $\omega$ vanishes along $D$; but this does not change the kernel of the contraction morphism, in fact $T_X\to \Omega_X^{q-1}[\otimes]\scrL(D)$ is nothing other then composition of $T_X\to \Omega_X^{q-1}[\otimes]\scrL$ and the inclusion $\scrL\hookrightarrow\scrL(D)$. Nevertheless, the condition (a) guarantees that the construction is reciprocal to the first one.}, consider the morphism $T_X\to\Omega_X^{q-1}[\otimes]\scrL$ given by the contraction with $\omega$, then the kernel of this morphism is a codimension $q$ foliation on $X$.
\end{itemize}
\end{prop-defn}
\begin{proof}
This is surely well known to experts and is formulated in another way in the literatures when $\scrF$ is regular and $X$ is smooth (see e.g. \cite[\S 1.2, pp.~10-11]{MM03}). However, due to lack of references treating the singular case, we will give a proof here for the convenience of the readers. The idea of the proof is borrowed from \cite[\S 1.2, pp.~10-11]{MM03}. 
\begin{itemize}
\item Let $\scrF$ be a foliation on $X$, then there is a Zariski open subset $X^\circ\subseteq X_{\reg}$ such that $\codim(X\backslash X^\circ)\geqslant 2$ and that $\scrF|_{X^\circ}\subseteq T_{X^\circ}$ is a subbundle. Then $\omega$ is nowhere vanishing on $X^\circ$, hence the vanishing locus of $\omega$ is of codimension $\geqslant 2$. Locally in $X^\circ$, we can take $v_1\,,\cdots,v_n$ trivializing sections (local vector fields) of $T_{X^\circ}$, among which $v_1\,,\cdots,v_{n-q}$ generate $\scrF|_{X^\circ}$. Let $\alpha_1\,,\cdots,\alpha_n$ be the dual sections of the $v_i$'s. Then locally $N_{\scrF}^\ast=(T_X/\scrF)^\ast$ is generated by $\alpha_{n-q+1}\,,\cdots,\alpha_n$\,, and hence $\det\!N_{\scrF}^\ast$ is generated by $\alpha_{n-q+1}\wedge\cdots\wedge\alpha_n$. Since $\omega$ is induced by the inclusion $\det\!N_{\scrF}\hookrightarrow\Omega_X^{[q]}$, hence under the local trivialisation of $\scrL|_{X^\circ}$ given by $\alpha_{n-q+1}\wedge\cdots\wedge\alpha_n$ the differential $q$-form $\omega$ is equal to the tautological section $\alpha_{n-q+1}\wedge\cdots\wedge\alpha_n$\,, hence $\omega$ is locally decomposable. And the kernel of the morphism $T_X\to \Omega_X^{q-1}[\otimes]\det\!N_{\scrF}$ can be locally expressed as 
\[
\left\{\,v\text{ local holomorphic vector field}\,\big|\,\alpha_i(v)=0,\forall i=n-q+1\,,\cdots,n\,\right\},
\]
which is then equal to $\scrF$ on $X^\circ$; both of them are reflexive, then $\scrF=\Ker(T_X\to\Omega_X^{q-1}[\otimes]\det\!N_{\scrF})$. Finally let us check that $\omega$ is integrable. To this end, take any two local sections $v$ and $w$ of $\scrF$, then by the formula (definition) of exterior derivative we get
\begin{equation}
\label{eq_formula-ext-derivative}
d\alpha_i(v,w)=\frac{1}{2}\left[v(\alpha_i(w))-w(\alpha_i(v))-\alpha_i([v,w])\right]=0,\;\forall i=n-q+1\,,\cdots,n
\end{equation}
since $[v,w]$ is still a local section of $\scrF$ as a result of the involutivity of $\scrF$. Therefore, for every $i=n-q+1,\cdots,n$\,, we can write
\[
d\alpha_i=\sum_{j=n-q+1}^n\eta_{ij}\wedge\alpha_j\,,
\]
for some local differential $1$-forms $\eta_{ij}$. Hence $d\alpha_i\wedge\alpha_{n-q+1}\wedge\cdots\wedge\alpha_n=0$.
\item Reciprocally, let $\scrL$ be a reflexive sheaf of rank $1$ on $X$ and $\omega\in\Coh^0(X,\Omega^q_X[\otimes]\scrL)$ satisfying the conditions (a)(b)(c) as above. Now consider $\scrF=\Ker(T_X\to\Omega_X^{q-1}[\otimes]\scrL)$. Since $\Omega_X^{q-1}[\otimes]\scrL$ is torsion free (being reflexive), $\scrF$ is saturated in $T_X$ (thus reflexive). We check that $\scrF$ is involutive. In virtue of {\hyperref[lemma_involutive]{Lemma \ref*{lemma_involutive}}}, up to replacing $X$ by a Zariski open whose complement is of codimension $\geqslant 2$, we can assume that $X$ is smooth, that $\omega$ is nowhere vanishing on $X$ and that $\omega$ is locally decomposable around every point of $X$. Locally we can write $\omega=\omega_1\wedge\cdots\wedge\omega_q$, since $\omega$ is nowhere vanishing, the $\omega_i$'s are everywhere linearly independent. Then we can complete $\{\omega_i\}_{i=1,\cdots,q}$ into a family of trivializing local sections $\omega_1,\cdots,\omega_n$ of $\Omega_X^1$. Then locally $\scrF$ is equal to
\begin{equation}
\label{eq_local-expression}
\left\{\,v\text{ local holomorphic vector field}\,\big|\,\omega_i(v)=0,\;\forall i=1,\cdots,q\,\right\}.
\end{equation}
Since $\omega$ is integrable, thus $d\omega_i\wedge\omega_1\wedge\cdots\wedge\omega_q=0$ for every $i=1,\cdots,q$. Write 
\[
d\omega_i=\sum_{1\leqslant j<k\leqslant n}a_{ijk}\omega_j\wedge\omega_k\,,
\]
since $\omega_j\wedge\omega_k\wedge\omega_1\wedge\cdots\wedge\omega_q=0$ for $j\leqslant q$, we get from the integrability condition:
\[
\sum_{q+1\leqslant j<k\leqslant n}a_{ijk}\omega_j\wedge\omega_k\wedge\omega_1\wedge\cdots\wedge\omega_q=0,
\]
which implies that $a_{ijk}=0$ if $j\leqslant q+1$. Hence we can write 
\[
d\omega_i=\sum_{j=1}^q\eta_{ij}\wedge\omega_j
\]
for some local $1$-forms $\eta_{ij}$; in particular, for every $i=1,\cdots,q$, $d\omega_i$ annihilates $\scrF$ . Then by the formula  \eqref{eq_formula-ext-derivative}, we see that $[\scrF,\scrF]$ is annihilated by every $\omega_i$, $i=1,\cdots,q$, which in turn implies, by the local characterization \eqref{eq_local-expression} of $\scrF$ above, that $[\scrF,\scrF]\subseteq\scrF$. Hence $\scrF$ is a foliation on $X$. Moreover, by the the local expression \eqref{eq_local-expression} we see that $(\omega_i)_{i=1,\cdots,q}$ is a family of local trivializing sections of $(T_X/\scrF)^\ast$. In consequence, $\omega\overset{\text{loc}}{=}\omega_1\wedge\cdots\wedge\omega_q$ is equal to the the rational $q$-form induced by $\scrF$.  
\end{itemize}
\end{proof}

With the help of the construction above, we can define the pullback of a foliation:
\begin{prop-defn}[c.f.{\cite[\S 3.4]{Druel18a}}]
\label{defn_pullback-foliation}
Let $\mu:X\dashrightarrow Y$ a dominant meromorphic mapping between normal complex varieties, which restricts to a surjective morphism $\mu^\circ:X^\circ\to Y^\circ$ with $X^\circ$ and $Y^\circ$ smooth Zariski open subsets of $X$ and of $Y$ respectively. Let $\scrG$ be a foliation on $Y$, then it induces a foliation $\mu\inv\!\scrG$ on $X$ as following: as in the {\hyperref[defn_foliation-diff-form]{Proposition-Definition \ref*{defn_foliation-diff-form}}}, $\scrG$ gives rise to a meromorphic differential $q$-form $\omega\in\Coh^0(Y,\Omega_Y^q[\otimes]\det\!N_{\scrG})$, then $(\mu^\circ)^\ast(\omega|_{Y^\circ})$ extends to a meromorphic $q$-form $\tau$ on $X_{\reg}$. By well choosing a rank $1$ reflexive sheaf $\scrL$ on $X$, $\tau$ can be regarded as a section in $\Coh^0(X,\Omega^q_X[\otimes]\scrL)$ whose vanishing locus has codimension $2$ in $X_{\reg}$ (thus in $X$). And by construction it is clear that $\tau$ is locally decomposable and integrable, then by {\hyperref[defn_foliation-diff-form]{Proposition-Definition \ref*{defn_foliation-diff-form}}} $\tau$ induces a foliation on $X$, which we denote by $\mu\inv\!\scrG$. By construction $\mu\inv\!\scrG$ is the unique foliation on $X$ whose associated differential $q$-form coincides with $(\mu^\circ)^\ast(\omega|_{Y^\circ})$ on $X^\circ$.     
\end{prop-defn}

In the proof of our main theorem, we will treat the situation where the tangent sheaf admits a direct sum decomposition into foliations, and we expect that under certain condition this decomposition can be retained via pullback. When the morphism is bimeromorphic, the following lemma provides a criterion to ensure this.

\begin{lemme}
\label{lemma_pullback-decomposition}
Let $\mu:X\to Y$ be a bimeromorphic morphism between normal complex varieties and let $\scrG_1$ and $\scrG_2$ be foliations on $Y$. Suppose that we have a direct sum decomposition $T_Y\simeq\scrG_1\oplus\scrG_2$\,, and suppose that the natural morphism $\det(\mu\inv\!\scrG_1)\rightarrow\det\left(T_X/\mu\inv\!\scrG_2\right)$ is an isomorphism. Then the decomposition of $T_Y$ pulls back to $X$:
\[
T_X\simeq\mu\inv\!\scrG_1\oplus\mu\inv\!\scrG_2
\]
\end{lemme}
\begin{rmq}
The lemma does not holds in general without the assumption on the natural morphism between determinant line bundles even for regular foliations on smooth varieties. For example, consider $Y=\PP^1\times\PP^1$ and $\mu: X\to Y$ be the blow-up of a general point on $\PP^1\times\PP^1$. Then $T_Y$ admits a natural decomposition
\[
T_Y\simeq \pr_1^\ast T_{\PP^1}\oplus\pr_2^\ast T_{\PP^1}
\]
into regular (algebraically integrable) foliations. This decomposition cannot pullback via $\mu$ to $X$. Otherwise, if it were the case then we have a decomposition 
\[
T_X\simeq \scrF_1\oplus \scrF_2
\]
with $\scrF_i$ the pullback foliation of $\pr_i^\ast T_{\PP^1}$, (the Zariski closure of) whose general leaf is rationally connected. By semicontinuity $\scrF_1$ and $\scrF_2$ are locally free, hence are regular foliations. Therefore, by \cite[2.11.Corollary]{Hor07}, $\scrF_1$ induces a smooth holomorphic submersion, whose fibres are transverse to the leaves of $\scrF_2$; then by the {\hyperref[thm_classical-Ehresmann]{classical Ehresmann Theorem \ref*{thm_classical-Ehresmann}}} (c.f. \cite[\S V.4, Theorem 1 and Theorem 3, pp.~95-99]{CLN85}), $X$ splits into a product of two curves. 
But $X$ is rationally connected, hence simply connected, then it must be isomorphic to $\PP^1\times\PP^1$, which is absurd. 
\end{rmq}

{\hyperref[lemma_pullback-decomposition]{Lemma \ref*{lemma_pullback-decomposition}}} follows immediately from the following general fact:

\begin{prop}
\label{prop_extend-decomposition}
Let $X$ be a normal complex variety and let $E$ be a reflexive sheaf on $X$ and $E_1\,,E_2$ saturated subsheaves of $E$. Suppose that there is a Zariski open $X_0$ of $X$ such that 
\begin{equation}
\label{eq_direct-decomp-X0}
E_1|_{X_0}\oplus E_2|_{X_0}\simeq E,
\end{equation}
and suppose that the natural morphism $\det\! E_1\rightarrow\det(E/E_2)$ is an isomorphism. Then the direct sum decomposition extends globally:
\[
E\simeq E_1\oplus E_2.
\]
\end{prop}
\begin{proof}
Since $X$ is normal, $E$, $E_1$ and $E_2$ are reflexive, up to replace $X$ by a Zariski open whose complementary is of codimension $\geqslant 2$, we can assume that $X$ is smooth, $E$ is a vector bundle and $E_1$ and $E_2$ are subbundles of $E$. 
Now consider the natural morphism 
\[
\sigma: E_1\hookrightarrow E\twoheadrightarrow E/E_2,
\]
By \eqref{eq_direct-decomp-X0} $\sigma$ is an isomorphism over $X_0$, then it must be injective ($\Ker(\sigma)$ is torsion free and generically $0$ hence must be $0$). Hence $E_1$ is a locally free (thus reflexive) subsheaf of the vector bundle $E/E_2$. In addition, the morphism $\det\sigma:\det\!E_1\to\det(E/E_2)$ is an isomorphism by the hypothesis. Then by \cite[Lemma 1.20]{DPS94} $E_1$ is a subbundle of $E/E_2$, hence they must be isomorphic. In particular this means that the short exact sequence 
\[
0\to E_2\to E\to E/E_2\to 0
\]
splits, thus we get the desired direct decomposition.
\end{proof}

\begin{rmq}
The proposition does not hold in general without the assumption even on the natural morphism between determinant bundles. For example (pointed out by Junyan Cao), consider $X=\AFF^2$ and $E=T_{\AFF^2}\simeq\scrO^{\oplus 2}_{\AFF^2}$ with $E_1$ the foliation generated by the global vector field
\[
v_1=z_1\cdot\frac{\partial}{\partial z_1}+z_2\cdot\frac{\partial}{\partial z_2},
\]
and $E_2$ the foliation generated by the global vector field
\[
v_2=z_2\cdot\frac{\partial}{\partial z_1}+z_1\cdot\frac{\partial}{\partial z_2}.
\]
Then $E_1$ and $E_2$ are locally free subsheaves of $E=T_{\AFF^2}$, and generically (out of the line $(z_1=z_2)$) $E\simeq E_1\oplus E_2$. But the decomposition cannot extend globally. In fact, the natural morphism $\det\!E_1\to\det\left(E/E_2\right)$ is zero along the line $(z_1=z_2)$. 
\end{rmq}

\subsection{Pfaff fields and invariant subvarieties}
\label{ss_foliation_Pfaff}

\begin{defn}
\label{defn_Pfaff-field}
Let $X$ be a normal complex variety. A Pfaff vector field of rank $r$ on $X$ is a non-trivial morphism $\eta:\Omega_X^r\to\scrL$ where $\scrL$ is a reflexive sheaf of rank $1$.
The {\it singular locus} $\Sing(\eta)$ of $\eta$ is the closed analytic subspace of $X$ defined by the ideal sheaf $\Image(\Omega_X^r[\otimes]\scrL^\ast\xrightarrow{\eta[\otimes]\scrL^\ast}\scrO_X)$. If $\scrL$ is invertible, then set-theoretically $\Sing(\eta)$ consists of the points at which $\eta$ is not surjective.
\end{defn}

\begin{defn}[{\cite[Definition 5.4]{Druel18b}}]
\label{defn_sing-foliation}
Let $\scrF$ be a foliation on $X$,
then $\scrF$ induces a Pfaff field of rank $r=\rank\scrF$ on $X$
\[
\eta_{\scrF}:\Omega_X^r=\bigwedge^r\!\Omega_X^1\to\bigwedge^r\!\scrF^\ast\to\det\!\scrF^\ast.
\]
The singular locus $\Sing(\scrF)$ of the foliation $\scrF$ is defined to be the singular locus of the Pfaff field $\eta_{\scrF}$. And $\scrF$ is called {\it weakly regular} if $\Sing(\scrF)=\varnothing$. 
\end{defn}
\begin{rmq}
\label{rmk_sing-foliation}
 If $X$ is smooth, then one deduces easily from \cite[Lemma 1.20]{DPS94} that (set-theoretically)
\begin{align*}
\Sing(\scrF)_{\red} &= \left\{\,x\in X\,\big|\,\scrF\to T_X\text{ is a injective bundle map at }x\,\right\} \\
&= \left\{\,x\in X\,\big|\,\scrF\text{ is a subbundle of }T_X\text{ at }x\,\right\}.
\end{align*}
\end{rmq}

\begin{defn}
\label{defn_inv-subvar}
Let $X$ be a normal complex variety and let $\eta:\Omega_X^r\to\scrL$ be a Pfaff field of rank $r$ on $X$. Suppose that some reflexive power of $\scrL$ is invertible. A closed analytic subspace $Y$ of $X$ is called invariant under $\eta$ if
\begin{itemize}
\item none of the irreducible components of $Y$ is contained in $\Sing(\eta)$;
\item for some $m\in\ZZ_{>0}$ such that $\scrL^{[m]}$ is invertible, the restriction $\eta\ptensor[m]:\left.(\Omega_X^r)\ptensor[m]\right|_Y\to\left.\scrL^{[m]}\right|_Y$ factors through the natural map $\left.(\Omega_X^r)\ptensor[m]\right|_Y\to(\Omega_Y^r)\ptensor[m]$.
$Y$ is said invariant under a $\QQ$-Gorenstein foliation $\scrF$ on $X$ if $Y$ is invariant under its associated Pfaff field $\eta_{\scrF}$.
\end{itemize}
\end{defn}

\begin{rmq}
\label{rmk_defn_inv-subvar}
Suppose that $Y$ is a reduced analytic subspace of $X$ and that none of its irreducible components is contained in $\Sing(\eta)$. Then $Y$ is invariant under $\eta$ if and only if the restriction $\eta|_{Y_{\reg}}:\left.\Omega_X^r\right|_{Y_{\reg}}\to\scrL|_{Y_{\reg}}$ factors through $\left.\Omega_X^r\right|_{Y_{\reg}}\to\Omega_{Y_{\reg}}^r$. More generally, one can replace $Y_{\reg}$ above by any Zariski dense subset of $Y_{\reg}$. This results from the following useful lemma (by taking $Y_0=Y_{\reg}$ or any Zariski dense subset of $Y_{\reg}$\,, $\scrM=\left.(\Omega_X^r)\ptensor[m]\right|_Y$ and $\scrN=(\Omega_Y^r)\ptensor[m]$ and noting that by \cite[\S I\!I.8, Proposition 8.12, p.~]{Har77} the natural morphism $(\Omega_X^r)^{\otimes m}|_Y\to(\Omega_Y^r)^{\otimes m}$ is surjective):
\end{rmq}

\begin{lemme}[c.f.{\cite[Proof of Proposition 3.2, p.~10]{EK03}}]
\label{lemma_factorization-vanishing}
Let $Y$ be a reduced complex analytic space, and let $\scrL$, $\scrM$ and $\scrN$ be coherent sheaves on $Y$ with a surjective morphism $\alpha:\scrM\to\scrN$. Then a morphism $\beta:\scrM\to\scrL$ factors through $\alpha$ if and only if $\beta$ annihilates $\Ker(\alpha)$. Suppose that $\scrL$ is torsion free, then $\beta$ factors through $\alpha$ if and only if there is a Zariski dense subset $Y_0$ of $Y$ such that $\beta|_{Y_0}$ factors through $\alpha|_{Y_0}$. 
\end{lemme}
\begin{proof}
By arguing components by components we can assume that $Y$ is irreducible, so that $Y$ is a complex variety. Since $\alpha$ is surjective, $\scrN=\Image(\alpha)=\Coker(\Ker(\alpha)\to\scrM)$, then the first statement results from the universal property of cokernels. Now turn to the second statement: since $\beta|_{Y_0}$ factors through $\alpha|_{Y_0}$, then by the first statement 
\[
\beta|_{Y_0}(\Ker(\alpha|_{Y_0}))=\beta(\Ker(\alpha))|_{Y_0}=0;
\]
but $\beta(\Ker(\alpha))\subseteq\scrL$ is a subsheaf of a torsion free sheaf, hence also torsion free, thus a fortiori $\beta(\Ker(\alpha))=0$, which implies, by the first statement, that $\beta$ factors through $\alpha$.
\end{proof}

The following lemma gives a characterization of invariant subvarieties which are not contained in the singular locus (other examples of invariant subvarieties can be found in \cite[Lemma 3.5]{Druel18b}):
\begin{lemme}[c.f.{\cite[Lemma 2.7]{AD13}}]
\label{lemma_inv-subvar=leaf}
Let $X$ be a complex manifold and $\scrF$ a rank $r$ foliation on $X$ with associated Pfaff field $\eta=\eta_{\scrF}:\Omega_X^r\to\det\!\scrF^\ast$. Set $S:=\Sing(\scrF)_{\red}$. Let $Y$ be a closed subvariety of $X$ of dimension $r$ such that $Y$ is not contained in $S$. Then $Y$ is invariant under $\eta$ if and only if $Y\backslash S$ is a leaf of $\scrF$


\end{lemme}

\begin{proof}
First note that since $X$ is smooth, $S$ is characterized by {\hyperref[rmk_sing-foliation]{Remark \ref*{rmk_sing-foliation}}}; it is of codimension $\geqslant 2$ by \cite[Corollary 5.5.15, p.~147]{Kob87}. Up to replacing $X$ by $X\backslash S$ we can assume $\scrF$ is a subbundle of $T_X$ so that $S=\varnothing$ (i.e. $\scrF$ is a regular foliation). Now take $x\in Y_{\reg}$ and take $v_1,\,\cdots,v_r$ local holomorphic vector fields around $x$ that generate (locally trivialize) $\scrF$. By construction $\eta$ is the dual morphism of the inclusion map $\det\!\scrF\hookrightarrow \bigwedge^r\!T_X$, hence locally it is given by 
\begin{equation}
\label{eq_local-Pfaff-field}
\alpha\longmapsto \alpha(v_1\,,\cdots,v_r)\cdot\alpha_0
\end{equation}
where $\alpha_0$ is a a section of $\det\!\scrF^\ast$ such that $\alpha_0(v_1,\cdots,v_r)=1$. Since $\Omega_X^r|_Y\to\Omega_Y^r$ is surjective, by {\hyperref[lemma_factorization-vanishing]{Lemma \ref*{lemma_factorization-vanishing}}} $Y$ is invariant under $\eta$ if and only if $\Ker(\left.\Omega_X^r\right|_Y\to\Omega_Y^r)$ is annihilated by $\eta|_Y$. Locally around $x$ ($Y$ is smooth around $x$) $\Ker(\left.\Omega_X^r\right|_Y\to\Omega_Y^r)$ consists of the $r$-forms of the form $d\!f\wedge\beta$ with $f$ a local holomorphic function vanishing along $Y$ and $\beta$ any local differential $(r-1)$-form. Combined with \eqref{eq_local-Pfaff-field} we see easily that locally around $x$, $Y$ is invariant under $\eta$ if and only if 
\[
d\!f\wedge\beta(v_1\,,\cdots,v_r)\big|_Y=0.
\]
Since $d\!f\wedge\beta|_Y=d(f\beta)|_Y-fd\beta|_Y=d(f\beta)|_Y$ since $f|_Y=0$, hence by the formula of exterior derivative we get
\begin{align*}
df\wedge\beta(v_1\,,\cdots,v_r)\big|_Y &= d(f\beta)(v_1\,,\cdots,v_r)\big|_Y \\
&= \frac{1}{r}\sum_{i=1}^r (-1)^iv_i(f\beta(v_1\,\cdots,\hat{v}_i\,,\cdots,v_r))\big|_Y \\
&\quad +\frac{1}{r}\sum_{1\leqslant i<j\leqslant r}(-1)^{i+j}f\beta([v_i,v_j],v_1\,\cdots,\hat{v}_i\,,\cdots,\hat{v}_j\,,\cdots,v_r)\big|_Y \\
&= \frac{1}{r}\sum_{i=1}^r (-1)^iv_i(f\beta(v_1\,\cdots,\hat{v}_i\,,\cdots,v_r))\big|_Y \\
&= \frac{1}{r}\sum_{i=1}^r(-1)^i \left(v_i(f)\big|_Y\cdot\beta(v_1\,,\cdots,\hat{v}_i\,,\cdots,v_r)\big|_Y+f\cdot v_i(\beta(v_1\,,\cdots,\hat{v}_i\,,\cdots,v_r))\big|_Y\right) \\
&= \frac{1}{r}\sum_{i=1}^r(-1)^i v_i(f)\big|_Y\cdot\beta(v_1\,,\cdots,\hat{v}_i\,,\cdots,v_r)\big|_Y
\end{align*}
Hence $Y$ is invariant under $\eta$ around $x$ if and only if $v_i(f)\big|_Y=0$ for every local holomorphic function $f$ vanishing along $Y$ and for every $i=1,\cdots,r$. Since $Y$ is a $r$-dimensional holomorphic submanifold of $X$ at $x$, this condition is equivalent to saying that $T_{Y_{\reg}}=\scrF|_Y$ around $x$. In consequence, $Y$ is invariant under $\eta$ $\Leftrightarrow$ $Y_{\reg}$ is contained in a leaf of $\scrF$ $\Leftrightarrow$ $Y=Y_{\reg}$ is a leaf of $\scrF$ (noting that $\scrF$ is a regular foliation by our assumption).  
\end{proof}

To end this subsection, let us recall the following lemma concerning the extension of Pfaff fields to the normalization:
\begin{lemme}[{\cite[\S 4, Theorem C, Corollary, p.~170]{Sei66}},{\cite[Proposition 4.5]{ADK08}},{\cite[Lemma 3.7]{AD14}}]
\label{lemma_Pfaff-normalization}
Let $X$ be a normal complex variety and let $\eta:\Omega_X^r\to\scrL$ be a Pfaff field of rank $r$ on $X$ where $\scrL$ is reflexive sheaf of rank $1$ such that $\scrL^{[m]}$ is invertible for some $m\in\ZZ_{>0}$. Let $Y$ be a subvariety of $X$ invariant under $\eta$, whose normalization is denoted by $\nu=\nu_Y:\bar Y\to Y$. Then the morphism $(\Omega_Y^r)\ptensor[m]\to\left.\scrL^{[m]}\right|_Y$ extends (uniquely) to a generically surjective morphism $(\Omega_{\bar Y}^r)\ptensor[m]\to\left.\nu^\ast\left(\scrL^{[m]}\right|_Y\right)$.     
\end{lemme}


\subsection{Algebraically integrable foliations}
\label{ss_foliation_algebraic}
\begin{defn}
\label{defn_alg-foliation}
Let $X$ be a normal algebraic variety and let $\scrF$ be a foliation on $X$. A leaf of $\scrF$ is called {\it algebraic} if it is (Zariski) open in its Zariski closure. $\scrF$ is called {\it algebraically integrable} if every leaf of $\scrF$ is algebraic.
\end{defn}

\begin{rmq}
\label{rmk_example-alg-foliation}
A typical example of algebraically integrable foliation is one induced by a equidimensional fibre space, i.e. $\scrF=T_{X/Y}:=(\Omega_{X/Y}^1)^\ast$ with $\pi:X\to Y$ a proper equidimensional morphism between normal algebraic varieties with connected fibres. In fact, $\scrF$ is clearly reflexive, and by virtue of {\hyperref[lemma_involutive]{Lemma \ref*{lemma_involutive}}} one can easily prove that $\scrF$ is involutive by showing that $\scrF$ involutive over the smooth locus of $\pi$, hence $\scrF$ is a foliation on $X$. In addition, by \cite[Lemma 2.31]{CKT16} the canonical divisor of $\scrF$ is described by the following equality: 
\[
\scrO_X(K_{\scrF})=\det(\Omega_{X/Y}^1) \simeq \scrO_X( K_{X/Y}-\Ramification(\pi))\,,
\]
where the ramification divisor $\Ramification(\pi)$ is defined by:
\[
\Ramification(\pi)=\sum_{D\text{ prime divisor on }X}\max(0,\multiplicity_D(f^\ast f_\ast D)-1)\cdot D.
\]
Notice that $\pi$ is equidimensional, then $\pi\inv(Y\backslash Y_{\reg})$ is still of codimension $2$ in $X$, hence pullbacks of Weil divisors are well-defined (c.f. \cite[Construction 2.13]{CKT16}). 
\end{rmq}

The following proposition, due to \cite[Lemma 3.2]{AD13}, says that every algebraically integrable foliation on a normal projective variety is of the form as in {\hyperref[rmk_example-alg-foliation]{Remark \ref*{rmk_example-alg-foliation}}} up to pullback by a birational morphism. In particular, one can construct a family whose general fibre parametrizes the closure of a general leaf of $\scrF$.

\begin{prop}[c.f.{\cite[Lemma 3.2]{AD13}}]
\label{prop_family-leaves}
Let $X$ be a normal projective variety and let $\scrF$ a algebraically integrable foliation. Then there is a unique closed subvariety $T$ of $\Chow(X)$ whose general point parametrize the Zariski closure of a general leaf of $\scrF$. That is, let $U\subseteq T\times X$ be the universal cycle along with morphisms $\pi:U\to T$ and $\beta:U\to X$, then $\beta$ is birational and for a general point $t\in T$, $\beta(\pi\inv(t))\subseteq X$ is the Zariski closure of a leaf of $\scrF$.
\end{prop}
\begin{proof}
First note that the Zariski closure of any leaf of $\scrF$ is irreducible and reduced, hence a subvariety of $X$. Let $T_1$ be the Zariski closure of the points of $\Chow(X)$ that parametrize leaves of $\scrF$, then $T_1$ is a reduced subscheme of $\Chow(X)$; since $\Chow(X)$ has only countably many components (c.f. \cite[\S I.3, 3.20 Definition, 3.21 Theorem, pp.~51-52]{Kollar96}), then so is $T_1$. Consider the universal cycle $U_1$ over $T_1$. Since the leaves are integral, the universal cycle over each component of $T_1$ is irreducible, hence the irreducible components of $U_1$ are in one-to-one correspondence with that of $T_1$, in particular $U_1$ also has only countably many irreducible components; now the natural map $U_1\to X$ is surjective, there is a unique component $U$ of $U_1$ which is dominant over $X$, and denote by $T$ the component of $T_1$ corresponding to $U$. Let $\pi:U\to T$ and $\beta:U\to X$ be the natural morphisms. 

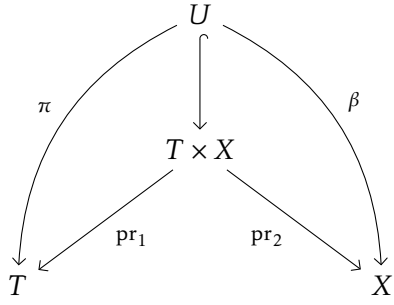
\begin{center}
\begin{tikzpicture}[scale=1.2]
\node (AB) at (0,0) {$T\times X$};
\node (A) at (-2,-1.5) {$T$};
\node (B) at (2,-1.5) {$X$};
\node (C) at (0,1.5) {$U$};
\path[->,font=\scriptsize,>=angle 90]
(AB) edge node[below right]{$\pr_1$} (A)
(AB) edge node[below left]{$\pr_2$} (B)
(C) edge[bend right] node[above left]{$\pi$} (A)
(C) edge[bend left] node[above right]{$\beta$} (B);
\path[right hook->, >=angle 90]
(C) edge (AB);
\end{tikzpicture}
\end{center}

Now it remains to show that for $t\in T$ general, $\beta(\pi\inv(t))\subseteq X$ is the Zariski closure of a general leaf of $\scrF$. To this end, first note that: up to replace $X$ by $X^\circ$ the Zariski open of $X_{\reg}$ where $\scrF$ is a subbundle of $T_{X_{\reg}}$, $T$ by $T^\circ$ where $T^\circ$ is a Zariski open of $T_{\reg}$ whose points correspond to the cycles that are not contained in $X\backslash X^\circ$\,, and $U$ by $U\cap\pr_1\inv(T^\circ)\cap\pr_2\inv(X^\circ)$ we can assume that $X$ and $T$ are smooth and $\scrF$ regular (by definition, a leaf is always contained in $X^\circ$, c.f. {\hyperref[defn_foliation]{Definition \ref*{defn_foliation}}}). In particular $K_{\scrF}$ is a Cartier divisor, and $\scrF$ induces a Pfaff field $\eta=\eta_{\scrF}:\Omega_X^r\to \scrO_X(K_{\scrF})$ where $r=\rank\scrF$. In the sequel we will use  {\hyperref[lemma_inv-subvar=leaf]{Lemma \ref*{lemma_inv-subvar=leaf}}} to conclude; to this end, we will show that $\eta$ induces a Pfaff field on $T\times X$ whose restriction to $U$ factors through $\Omega_{U/T}^r$. In fact, $\eta$ induces a Pfaff field on $T\times X$
\[
\pr_2\inv\!\eta:\Omega_{T\times X}^r\simeq\bigwedge^r(\pr_1^\ast\Omega_T^1\oplus\pr_2^\ast\Omega_X^1)\xrightarrow{\text{projection}}\bigwedge^r\pr_2^\ast\Omega_X^1\simeq \pr_2^\ast\Omega_X^r\xrightarrow{\pr_2^\ast\eta}\scrO_{T\times X}(\pr_2^\ast K_{\scrF}).
\]
Then we will show that the restriction morphism $\left.\pr_2\inv\!\eta\right|_U:\left.\Omega_{T\times X}^r\right|_U\to \scrO_U(\beta^\ast K_{\scrF})$ factors through the composition map $\left.\Omega_{T\times X}^r\right|_U\to\Omega_U^r\twoheadrightarrow\Omega_{U/T}^r$ (c.f. \cite[\S II.8, Proposition 8.11, p.~176]{Har77}). By construction there is a Zariski dense subset of $T$ whose points parametrize the leaves of $\scrF$; then by the proof of {\hyperref[lemma_inv-subvar=leaf]{Lemma \ref*{lemma_inv-subvar=leaf}}} $\Ker(\left.\Omega_{T\times X}^r\right|_U\to\Omega_{U/T}^r)$ is annihilated by $\left.\pr_2\inv\!\eta\right|_U$ over a Zariski dense of $U$, thus is annihilated by $\left.\pr_2\inv\!\eta\right|_U$ everywhere on $U$ since $\scrO_U(\beta^\ast K_{\scrF})$ is torsion free. By {\hyperref[lemma_factorization-vanishing]{Lemma \ref*{lemma_factorization-vanishing}}} we see that $\left.\pr_2\inv\!\eta\right|_U$ factors through $\left.\Omega_{T\times X}\right|_U\to\Omega_{U/T}^r$. By the base change for Kähler differentials (\cite[\S II.8, Proposition 8.10, p.~175]{Har77}) for every $t\in T$ we have $\left.\Omega_{U/T}^r\right|_{U_t}\simeq\Omega_{U_t}^r$ and thus every $U_t$ is invariant under $\pr_2\inv\eta$, which amounts to say that every $\beta(U_t)$ is invariant under $\eta$. By the generic flatness (c.f. \cite[\S 22, pp.~156-159]{Mat70}) and \cite[Th\'eor\`eme 12.2.1(x), pp.~179-180]{EGA4-3}, for general $t\in T$, $U_t$ is irreducible and reduced (then so is $\beta(U_t)$), hence by {\hyperref[lemma_inv-subvar=leaf]{Lemma \ref*{lemma_inv-subvar=leaf}}}, $\beta(U_t)$ is the closure of a (general) leaf of $\scrF$.

\end{proof}
 
The morphism $\pi:U\to T$ constructed above is called the {\it family of leaves} of $\scrF$. In the following proposition we study the relation between the canonical divisor of $\scrF$ and that of the pullback of $\scrF$ to the family of leaves. 
\begin{prop}[{\cite[Remark 3.12]{AD14}}, {\cite{AD14err}}, {\cite[\S 3.10]{AD16}}]
\label{prop_canonical-family-leaves}
Let $X$ be a projective normal variety and $\scrF$ be a algebraically integrable foliation on $X$. Let $\pi:U\to T$ the family of leaves of $\scrF$ as constructed in {\hyperref[prop_family-leaves]{Proposition \ref*{prop_family-leaves}}}. Let $T'\to T$ be any surjective morphism with $T'$ normal and let $U'=U_{T'}$ be the pullback of the universal family $U$, whose normalization is denoted by $\nu=\nu_{U'}:\bar U'\to U'$. Let $\beta_{T'}$ (resp. $\pi_{T'}$, resp. $\bar\beta_{T'}$, resp. $\bar\pi_{T'}$) be the induced morphism $U'\to X$ (resp. $U'\to T'$, resp. $\bar U'\to X$, resp. $\bar U'\to T'$). Then 
\begin{itemize}
\item[\rm(a)] The pullback foliation $\bar\beta_{T'}\inv\!\scrF$ is equal to $T_{\bar U'/T'}:=(\Omega_{\bar U'/T'}^1)^\ast$\,;
\item[\rm(b)] Assume that $\scrF$ is $\QQ$-Gorenstein, then there is a canonical effective Weil $\QQ$-divisor $\Delta_{T'}$ on $\bar U'$ such that $K_{\bar\beta_{T'}\inv\!\scrF}+\Delta_{T'}\sim_{\QQ}\bar\beta_{T'}^\ast K_{\scrF}$. If $T'\to T$ is birational then $\Delta_{T'}$ is $\bar\beta_{T'}$-exceptional. 
\end{itemize}
\end{prop}
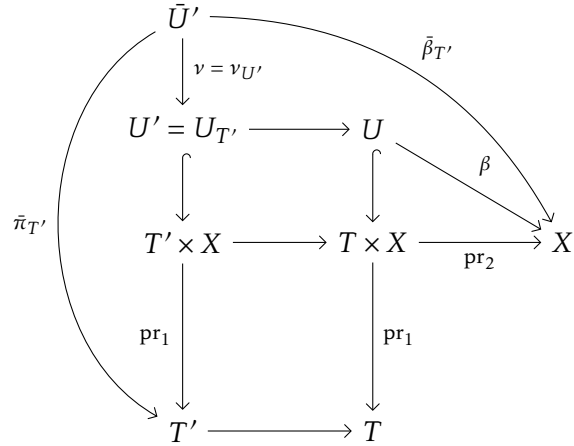
\begin{center}
\begin{tikzpicture}[scale=2.5]
\node (A) at (0,0) {$T$};
\node (A1) at (0,1) {$T\times X$};
\node (A2) at (0,1.6) {$U$};
\node (B) at (-1,0) {$T'$};
\node (B1) at (-1,1) {$T'\times X$};
\node (B2) at (-1,1.6) {$U'=U_{T'}$};
\node (B3) at (-1,2.2) {$\bar U'$};
\node (C) at (1,1) {$X$};
\path[->,font=\scriptsize,>=angle 90]
(A1) edge node[right]{$\pr_1$} (A)
(B1) edge node[left]{$\pr_1$} (B)
(A2) edge node[above right]{$\beta$} (C)
(A1) edge node[below]{$\pr_2$} (C)
(B3) edge node[right]{$\nu=\nu_{U'}$} (B2)
(B) edge (A)
(B1) edge (A1)
(B2) edge (A2)
(B3) edge[bend right=60] node[left]{$\bar\pi_{T'}$} (B)
(B3) edge[bend left] node[above right]{$\bar\beta_{T'}$} (C);
\path[right hook->, >=angle 90]
(A2) edge (A1)
(B2) edge (B1);
\end{tikzpicture}
\end{center}

\begin{proof}
First notice that since $\nu$ is a finite morphism, $\bar\pi_{T'}:\bar U'\to T'$ is still equidimensional and hence $T_{\bar U'/T'}$ is a foliation on $\bar U'$. Then (a) is clear : in fact, since $T'\to T$ is surjective, by {\hyperref[prop_family-leaves]{Proposition \ref*{prop_family-leaves}}} there is a Zariski open of $T'$ over which the fibres of $\pi_{T'}:U'\to T'$ are leaves of $\scrF$, hence $\bar\beta_{T'}\inv\!\scrF$ and $T_{\bar U'/T'}$ coincide over a Zariski open of $\bar U'$, then by the uniqueness in {\hyperref[defn_pullback-foliation]{Proposition-Definition \ref*{defn_pullback-foliation}}} a fortiori $\bar\beta_{T'}\inv\!\scrF=T_{\bar U'/T'}$. Now turn to the proof of (b). Consider the Pfaff field associated to $\scrF$
\[
\eta:=\eta_\scrF:\Omega_X^r\to\scrO_X(K_{\scrF}),
\]
as in the proof of {\hyperref[prop_family-leaves]{Proposition \ref*{prop_family-leaves}}} $\eta$ induces a Pfaff field on $T'\times X$ ($T'\times X$ is normal)
\[
\pr_2\inv\!\eta:\Omega_{T'\times X}^r\to \scrO_{T'\times X}(\pr_2^\ast K_{\scrF}),
\]
where $\pr_2^\ast$ above denotes the pullback of Weil divisors (or algebraic cycles) by equidimensional (or flat) morphisms (c.f.  \cite[Construction 2.13]{CKT16} or \cite[\S 1.7, pp.~18-21]{Ful84}); moreover, the restriction morphism
\[
\left.(\pr_2\inv\!\eta)\ptensor[m]\right|_{U'}:\left.(\Omega_{T'\times X}^r)\ptensor[m]\right|_{U'}\to \scrO_{U'}(m\beta_{T'}^\ast K_{\scrF})
\]
factors through $\left.(\Omega_{T'\times X}^r)\ptensor[m]\right|_{U'}\to(\Omega_{U'/T'}^r)\ptensor[m]$ where $m$ is a positive integer such that $mK_{\scrF}$ is Cartier, in particular $U'$ is invariant under $\pr_2\inv\!\eta$. Now by {\hyperref[lemma_Pfaff-normalization]{Lemma \ref*{lemma_Pfaff-normalization}}} we get a generically surjective morphism 
\[
(\Omega_{\bar U'}^r)\ptensor[m]\to\scrO_{\bar U'}(m\bar\beta_{T'}^\ast K_{\scrF})
\]
which factors through the natural surjection $\Omega_{\bar U'}\twoheadrightarrow\Omega_{\bar U'/T'}$. Then we get an injection of rank $1$ reflexive sheaves
\[
\det(\Omega_{\bar U'/T'}^1)\ptensor[m]\hookrightarrow \scrO_{\bar U'}(m\bar\beta_{T'}^\ast K_{\scrF}),
\]
hence there is a unique effective Weil $\QQ$-divisor $\Delta_{T'}$ ($m\Delta_{T'}$ is the Weil divisor defined by this injection) such that
\[
K_{\bar\beta_{T'}\inv\!\scrF}+\Delta_{T'}\sim_{\QQ} \bar\beta_{T'}^\ast K_{\scrF}.
\]
Moreover, combining this with {\hyperref[rmk_example-alg-foliation]{Remark \ref*{rmk_example-alg-foliation}}} we get
\[
K_{\bar U'/T'}-\Ramification(\bar\pi_{T'})+\Delta_{T'}\sim_{\QQ}\bar\beta_{T'}^\ast K_{\scrF}.
\]
If $T'\to T$ is birational, then by {\hyperref[prop_family-leaves]{Proposition \ref*{prop_family-leaves}}} $\Delta_{T'}$ is $\bar\beta_{T'}$-exceptional.
\end{proof}

We close this subsection by considering algebraically integrable foliations that are weakly regular (c.f. {\hyperref[defn_sing-foliation]{Definition \ref*{defn_sing-foliation}}}). It is clear that a foliation induced by a equidimensional fibre space (c.f. {\hyperref[rmk_example-alg-foliation]{Remark \ref{rmk_example-alg-foliation}}}) is weakly regular, the following result says that the converse is true for (weakly regular) foliations with canonical singularities over $\QQ$-factorial klt projective varieties.
\begin{thm}[{\cite[Theorem 6.1]{Druel18b}}]
\label{thm_weak-reg-foliation}
Let $X$ be a (normal) $\QQ$-factorial projective variety with klt singularities, and let $\scrG$ be a
weakly regular algebraically integrable foliation on $X$. Suppose in addition that $\scrG$ has canonical singularities.
Then $\scrG$ is induced by an equidimensional fibre space $\psi:X\to Y$ onto a normal projective variety $Y$. Moreover, there exists an open subset $Y^\circ$ with complement of codimension $\geqslant 2$ in Y such that $\psi\inv(y)$ is irreducible for any $y\in Y^\circ$.
\end{thm}

In the study of algebraically integrable foliations, the family of leaves is a quite useful tool which permits the enter of the algebro-geometric methods; nonetheless, by passing to the family of leaves, one loses the control of the singularities. The above {\hyperref[thm_weak-reg-foliation]{Theorem \ref*{thm_weak-reg-foliation}}} says that weakly regular foliations with canonical singularities on a projective variety $X$ with mild singularities have the advantage that there is no need to pass to the family of leaves (since it is isomorphic to $X$ itself) and hence there is no loss of the control of the singularities of $X$. C.f. also {\hyperref[rmk_lemma_num-triv-canonical-weak-reg]{Remark \ref*{rmk_lemma_num-triv-canonical-weak-reg}}}.

\subsection{Foliations transverse to holomorphic submersions}
In this subsection we consider regular foliations which are transverse to a smooth fibration and we recall the important (analytic version of) classical Ehresmann theorem. Let $f: V\to W$ be a smooth morphism (holomorphic submersion) between complex manifolds and let $\scrF$ be a regular foliation on $V$. Then $\scrF$ is said to be {\it transverse to} $f$ if the following two conditions are verified: 
\begin{itemize}
\item[(i)] The tangent bundle sequence of $f$ gives rise to a direct decomposition $T_V\simeq T_{V/W}\oplus\scrF$.
\item[(i\!i)] The restriction of $f$ to any leaf of $\scrF$ is an \'etale (not necessarily finite) cover.
\end{itemize}
By \cite[\S V.2, Proposition 1, pp.~91-92]{CLN85} (by \cite[\S 9.1, Proposition 9.5, pp.~209-210]{Voi02} $f$ can be viewed as a $\mathscr{C}^\infty$ fibre bundle), if $f$ is proper then the condition (i) implies (i\!i). The most important result for these foliations is the following analytic version of the classical Ehresmann theorem :
\begin{thm}[{\cite[3.17.Theorem]{Hor07}}]
\label{thm_classical-Ehresmann}
Let $f:V\to W$ be a holomorphic submersion between complex manifolds and let $\scrF$ be a regular foliation on $V$ transverse to $f$. Suppose that $W$ and the general fibre $F$ of $f$ are connected. Then $f$ is an analytic fibre bundle. Moreover, there is a representation $\rho:\pi_1(W)\to\Aut(F)$ such that $V$ is biholomorphic to $(\widetilde W\times F)/\pi_1(W)$ where $\pi_1(W)$ acts on $\widetilde W\times F$ via $\alpha:(y,s)\to(\alpha(y),\rho(\alpha)(s))$ where $\widetilde W\to W$ denotes the universal cover of $W$. In other words, $f$ is a locally constant fibration.   
\end{thm}
See \cite[\S V.3, Theorem 1 and Theorem 3, pp.~91-95]{CLN85} for the proof. The above statement is taken from \cite[3.17.Theorem]{Hor07}.
\section{Horizontal divisors and base changes}
\label{sec_horizon}
Let $V$ be a complex variety which is fibred over another complex variety $W$ (c.f. {\hyperref[defn_local-const-fibration]{Definition \ref*{defn_local-const-fibration}}}). By looking at the dimension of the image of its components in $W$, a (Weil) divisor on $V$ can be divided into a sum of the horizontal part plus the vertical part (c.f. \cite[\S 2.1.C, Proof of Corollary 2.1.38, p.~138]{Laz04}). The aim of this appendix is to show that the notion of "horizontality" for divisors on an equidimensional fibre space is stable under base change. This result is of course well known to experts, we nevertheless provide a detailed account here for the convenience of the readers. The main result is the following:
\begin{prop}
\label{prop_horizon-base-change}
Let $f:V\to W$ be an equidimensional fibre space between complex varieties with $W$ quasi-projective and let $D$ be a Cartier divisor on $V$ which is horizontal with respect to $V$, then for any morphism $g: W'\to W$, the pullback divisor $g_V^\ast\!D$ is horizontal with respect to the base change morphism $f':V'\to W'$ where $V':=V\underset{W}{\times}W'$ and $g_V:V'\to V$ is the natural morphism.  
\end{prop}

The key point in the proof of the proposition above consists in proving the following auxiliary:
\begin{lemme}
\label{lemma_horizon-equidim}
Let $f:V\to W$ be an equidimensional fibre space between complex varieties of relative dimension $d$ and let $D$ be an effective Weil divisor on $V$ which is horizontal with respect to $f$. Suppose that $W$ is quasi-projective. Then $D$ is equidimensional over $W$ of relative dimension $d-1$. 
\end{lemme}
\begin{proof} 
By induction on the dimension of $W$. If $\dim W=0$ there is nothing to prove. For arbitrary $W$, the result follows if $\scrO_D$ is flat over $W$. In general, we apply the generic flatness \cite[(22.A) Lemma 1, pp.~156-158]{Mat70} (c.f. also \cite[Lemma 2.1.6, p.~36]{HL10} and \cite[Proposition (3.9), p.~18]{ACG11}) to $f$ and $\scrO_D$ to find an effective divisor $H$ on $W$ (by using quasi-projectivity of $W$) such that $\scrO_D$ is flat over $W\backslash H$. Since $D$ is horizontal and $f$ is surjective, $D$ is mapped surjectively onto $W$, hence by \cite[(13.B) Theorem 19(2), p.~79]{Mat70} $D$ is equidimensional over $W\backslash H$ of relative dimension $d-1$. Again by horizontality of $D$, $f\inv(H)$ cannot be contained in $\Supp(D)$, hence $D|_{f\inv(H)_{\red}}$ is still an effective Weil divisor and is horizontal over $H$, then by applying the induction hypothesis to $f|_{f\inv(H)_{\red}}: f\inv(H)_{\red}\to H_{\red}$ which is still a equidimensional fibre space of relative dimension $d$  we see that $D$ is equidimensional over $W$ of relative dimension $d-1$ ($H_{\red}$ may be reducible, yet by considering component by component we can conclude).  
\end{proof}

Now let us turn to the proof of the {\hyperref[prop_horizon-base-change]{Proposition \ref*{prop_horizon-base-change}}}.

\begin{proof}[Proof of the {\hyperref[prop_horizon-base-change]{Proposition \ref*{prop_horizon-base-change}}}]
By definition, it suffices to treat the case that $D$ is a prime divisor. Suppose by contradiction that $g_V^\ast\!D$ contains a component $E$ which is vertical with respect to $f'$. Since $f$ is equidimensional, then so is $f'$, hence $f'(E)$ is of codimension $1$ in $W$. By \cite[(13.B) Theorem 19(1), p.~79]{Mat70} the restriction of $E$ to any fibre of $f'$ is either empty or of dimension $d$, where $d$ denotes the relative dimension of $f$ (hence also that of $f'$). In consequence $D$ must contain a $d$-dimensional component of a fibre of $f'$. But this is impossible by the {\hyperref[lemma_horizon-equidim]{Lemma \ref*{lemma_horizon-equidim}}}.  
\end{proof}
\end{appendices}

\bibliographystyle{alpha}
\bibliography{anti-nef(bibtex)}

\end{document}